\documentclass{amsart}

\title[The regularity of KP solitons]
{The Deodhar decomposition of the Grassmannian and the regularity of KP solitons}

\author{Yuji Kodama and Lauren Williams} 
\date{\today}
\thanks{The first author was partially
supported by NSF grant DMS-1108813.  The second author was 
partially supported by an NSF CAREER award and an 
Alfred Sloan Fellowship.}
\address{Department of Mathematics, Ohio State University,
Columbus, OH 43210}
\email{kodama@math.ohio-state.edu}
\address{Department of Mathematics, University of California,
Berkeley, CA 94720-3840}
\email{williams@math.berkeley.edu}
\subjclass[2000]{}

\usepackage{rotating}
\newcommand{\rotxc}[1]{\begin{sideways}#1\end{sideways}}
\newcommand{\invert}[1]{\rotxc{\rotxc{#1}}}

\def\Le{\hbox{\invert{$\Gamma$}}}
\countdef\x=23
\countdef\y=24
\countdef\z=25
\countdef\t=26

\def\tbox(#1,#2)#3{
\x=#1 \y=#2 
\multiply\x by 12 
\multiply\y by 12 
\z=\x \t=\y
\advance\z by 12 
\advance\t by 12 
\psline(\x,\y)(\x,\t)(\z,\t)(\z,\y)(\x,\y)
\advance\x by 6
\advance\y by 6 
\rput(\x,\y){{\bf #3}}}

\usepackage{amssymb}
\usepackage{graphicx}
\usepackage{epstopdf}
\usepackage{amscd}
\usepackage{graphics}
\def\proof{\par{\it Proof}. \ignorespaces}
\def\endproof{{\ \vbox{\hrule\hbox{%
     \vrule height1.3ex\hskip0.8ex\vrule}\hrule }}\par}

\theoremstyle{definition}

\theoremstyle{remark}

\numberwithin{equation}{section}

\let\trueint=\int
\let\truesum=\sum
\def\int{\mathop{\textstyle\trueint}\limits}
\def\sum{\mathop{\textstyle\truesum}\limits}
\let\<=\langle
\let\>=\rangle

\def\K{{\mathbb{K}}}
\def\F{{\mathbb{F}}}

\def\S{{\mathcal S}}
\def\J{{\mathcal J}}

\def\SL{\mathrm{SL}}
\def\Sym{{\mathfrak S}}

\def\rank{\mathop{\rm rank}\nolimits}

\def\v{\mathbf{v}}
\def\w{\mathbf{w}}
\def\wstn{\raisebox{0.12cm}{\hskip0.14cm\circle{4}\hskip-0.15cm}}
\def\bstn{\raisebox{0.12cm}{\hskip0.14cm\circle*{4}\hskip-0.15cm}}

\usepackage{amsmath, amsthm, amssymb, amsbsy,amsfonts,amscd}
\usepackage[mathscr]{eucal}
\usepackage{epsfig}
\usepackage{young}
\newtheorem{theorem}{Theorem}[section]

\newtheorem{definition}[theorem]{Definition}
\newtheorem{proposition}[theorem]{Proposition}
\newtheorem{lemma}[theorem]{Lemma}
\newtheorem{problem}[theorem]{Problem}
\newtheorem{example}[theorem]{Example}
\newtheorem{corollary}[theorem]{Corollary}

\newtheorem{remark}[theorem]{Remark}
\newtheorem{algorithm}[theorem]{Algorithm}
\usepackage{amsfonts}
\usepackage{xy}
\usepackage{amssymb}
\usepackage{amsmath}
\newcommand{\inv}{^{-1}}
\newcommand{\To}{\longrightarrow}

\newcommand{\R}{\mathbb R}

\newcommand{\B}{\mathcal{B}}

\newcommand{\Q}{\mathcal{Q}}

\DeclareMathOperator{\Dec}{Dec}

\DeclareMathOperator{\CC}{\mathcal C}

\DeclareMathOperator{\In}{in}
\DeclareMathOperator{\Out}{out}

\DeclareMathOperator{\GL}{GL}

\DeclareMathOperator{\M}{\mathcal M}
\DeclareMathOperator{\I}{\mathcal I}

\newcommand{\thmrefer}[1]{\renewcommand\thetheorem
  {\protect\ref{#1}}\addtocounter{theorem}{-1}}

\xyoption{all}
\CompileMatrices

\begin{document}

\begin{abstract}
Given a point $A$ in the real Grassmannian, it is well-known that
one can construct a soliton solution $u_A(x,y,t)$ to the KP equation. 
The \emph{contour plot} of such a solution provides a tropical 
approximation to the solution when the variables $x$, $y$, and $t$ 
are considered 
on a large scale and the time $t$ is fixed.
In this paper we use several decompositions of the Grassmannian
in order to gain an understanding of the contour plots of 
the corresponding soliton solutions.  First we use the 
\emph{positroid stratification} of the real Grassmannian in order to characterize the 
unbounded line-solitons in the contour plots at $y\gg 0$ and $y\ll 0$.  
Next we use 
the \emph{Deodhar decomposition} of the Grassmannian -- 
a refinement of the positroid stratification -- to study 
contour plots at $t\ll 0$.
More specifically, we index the components of the Deodhar decomposition of the 
Grassmannian by certain tableaux which we call 
\emph{Go-diagrams}, and then use these Go-diagrams 
to characterize the contour plots of solitons solutions when $t\ll 0$.
Finally we use these results to show that 
a soliton solution $u_A(x,y,t)$ is regular for all times $t$ 
if and only if $A$ comes from the \emph{totally non-negative part}
of the Grassmannian.
\end{abstract}

\maketitle

\setcounter{tocdepth}{1}
\tableofcontents

\section{Introduction}

The KP equation is a two-dimensional
nonlinear dispersive wave equation which was proposed by Kadomtsev and Peviashvili
in 1970 to study the stability problem of the soliton solution of the Korteweg-de Vries (KdV)
equation \cite{KP70}.  The KP equation can also be used to describe shallow water waves, and
in particular, the equation provides an excellent model for the resonant interaction of those waves.
The equation has a rich
mathematical structure, and is now considered
to be the prototype of an integrable nonlinear
dispersive wave equation with two spatial dimensions (see for example  \cite{NMPZ84,AC91,D91,MJD00,H04}).

One of the main breakthroughs in the KP theory was given by Sato \cite{Sato}, who
realized that solutions of the KP equation could be written in terms of points on an
infinite-dimensional Grassmannian.  The present paper deals with a real, finite-dimensional
version of the Sato theory; in particular, we are interested in solutions that
are localized along
certain rays in the $xy$ plane called \emph{line-solitons}.  Such a soliton 
solution
can be constructed from a point $A$ of the real Grassmannian.  More specifically,
one can apply
the {\it Wronskian form} \cite{Sato,Sa79,FN,H04}
to $A$ to produce
a $\tau$-function $\tau_A(x,y,t)$ which is a sum of exponentials, and from the
$\tau$-function 
one can construct a solution $u_A(x,y,t)$ to the KP equation.

Recently several authors
have studied the soliton solutions $u_A(x,y,t)$ which come from points $A$ of 
the {\it totally non-negative part of the Grassmannian} $(Gr_{k,n})_{\geq 0}$, that is,
those points of the real Grassmannian $Gr_{k,n}$
whose Pl\"ucker coordinates are all non-negative
\cite{BK,K04, BC06, CK1,CK3, KW, KW2}.  These solutions are 
\emph{regular}, and include
a large variety of soliton solutions which were previously overlooked by those using the
Hirota method of a perturbation expansion \cite{H04}.

One of the main goals of this paper is to understand the soliton solutions
$u_A(x,y,t)$ coming from arbitrary points $A$ of the real Grassmannian, not just 
the totally non-negative part. In general such solutions are no longer regular -- 
they may have singularities along rays in the $xy$ plane -- 
but it is possible, nevertheless, to understand a great deal about the asymptotics of such solutions.

Towards this end, we use two related decompositions
of the real Grassmannian.  The first decomposition is 
Postnikov's \emph{positroid stratification} of the Grassmannian 
\cite{Postnikov}, whose strata are indexed by various combinatorial
objects including decorated permutations and $\Le$-diagrams.  
Note that the intersection of each positroid stratum with $(Gr_{k,n})_{\geq 0}$
is a cell (homeomorphic to an open ball); 
when one intersects the positroid stratification of the Grassmannian with 
the totally non-negative part, one obtains a cell decomposition of 
$(Gr_{k,n})_{\geq 0}$ \cite{Postnikov}.

The second decomposition is the 
\emph{Deodhar decomposition} of the Grassmannian, which is 
a refinement of the positroid stratification.
Its components have explicit 
parameterizations due to Marsh and Rietsch \cite{MR}, and are indexed by 
\emph{distinguished subexpressions} of reduced words in the Weyl 
group.  The components may also be indexed by certain tableaux filled with black
and white stones which we call \emph{Go-diagrams}, and which provide
a generalization of $\Le$-diagrams.  Note that almost all Deodhar 
components have an empty intersection with the totally non-negative part 
of the Grassmannian.  More specifically, each positroid stratum is a union
of Deodhar components, precisely one of which has a non-empty 
intersection with $(Gr_{k,n})_{\geq 0}$.

By using the positroid stratification of the Grassmannian, we characterize
the unbounded line-solitons of KP soliton solutions coming from arbitrary points 
of the real Grassmannian.  More specifically, given $A \in Gr_{k,n}$, 
we show that the unbounded line-solitons of the solution $u_A(x,y,t)$
at $y\ll 0$ and $y \gg 0$ depend only on which positroid stratum $A$ belongs to, and that one 
can use the corresponding decorated permutation to read off the unbounded 
line-solitons.  This extends work of \cite{BC06, CK1, CK3, KW, KW2}
from the setting of the non-negative part of the Grassmannian to the entire real Grassmannian.

By using the Deodhar decomposition of the Grassmannian,
we give an explicit description of the {\it contour plots} of soliton solutions
in the $xy$-plane  when $t\ll 0$.  
The contour plot of the solution $u_A(x,y,t)$ at a fixed $t$ approximates
the locus where $|u_A(x,y,t)|$ takes on its maximum values or is singular.
More specifically, we provide an algorithm for constructing 
the contour plot of $u_A(x,y,t)$ at $t\ll 0$, which uses the Go-diagram indexing the 
Deodhar component of $A$.  We also show that when the Go-diagram 
$D$ is a $\Le$-diagram, then the corresponding contour plot at 
$t\ll 0$ gives rise to a \emph{positivity test} for the Deodhar component
$S_D$.

Finally we use our previous results to address the regularity problem for KP solitons.
We prove that a soliton solution $u_A(x,y,t)$ coming from a point $A$ of the real Grassmannian
is regular for all times $t$ if and only if $A$ is a point of the totally non-negative part of
the Grassmannian.

The structure of this paper is as follows.
In Section \ref{sec:Gr} we provide background on the Grassmannian 
and some of its decompositions, including the positroid stratification.
In Section \ref{sec:project} we describe
the Deodhar decomposition of the complete flag variety and its projection
to the Grassmannian, while in Section \ref{Deodhar-combinatorics} we
explain how to index Deodhar components in the Grassmannian by 
\emph{Go-diagrams} (Subsection \ref{subsec:Go}).
In Section \ref{sec:positivitytest} we provide explicit formulas for
certain Pl\"ucker coordinates of points in Deodhar components (Theorems \ref{p:maxmin} and \ref{p:Plucker}), and use
these formulas to provide \emph{positivity tests} for points
in the real Grassmannian (Theorem \ref{th:TPTest}).  Subsequent sections provide applications
of the previous results to soliton solutions of the KP equation.
In Section \ref{soliton-background} we give background on  how to produce
a soliton solution to the KP equation from a point of the real
Grassmannian.  In Section \ref{sec:solgraph} we 
define the \emph{contour plot} associated to a soliton solution
at a fixed time $t$ (Definition \ref{contour}), then in Section \ref{sec:unbounded} we use the positroid stratification to 
describe the unbounded line-solitons in contour plots of soliton solutions
 at $y\gg 0$ and $y \ll 0$ (Theorem \ref{thm:soliton-perm}).
 In Section \ref{sec:solitonplabic}
we define the more combinatorial notions of 
\emph{soliton graph} and \emph{generalized plabic graph}.
In Section \ref{sec:t<<0} we use the Deodhar
decomposition to describe contour plots of soliton solutions
for $t\ll 0$ (Theorem \ref{t<<0}), and in Section \ref{sec:slides} we provide 
some technical results on $X$-crossings in contour plots and
corresponding relations among Pl\"ucker
coordinates. 
Finally we use the results of the previous sections to address the 
\emph{regularity problem} for soliton solutions in Section \ref{sec:regularity}
(Theorem \ref{th:regularity}).


\section{Background on the Grassmannian and its decompositions}\label{sec:Gr}

The \emph{real Grassmannian} $Gr_{k,n}$ is the space of all
$k$-dimensional subspaces of $\R^n$.  An element of
$Gr_{k,n}$ can be viewed as a full-rank $k\times n$ matrix modulo left
multiplication by nonsingular $k\times k$ matrices.  In other words, two
$k\times n$ matrices represent the same point in $Gr_{k,n}$ if and only if they
can be obtained from each other by row operations.
Let $\binom{[n]}{k}$ be the set of all $k$-element subsets of $[n]:=\{1,\dots,n\}$.
For $I\in \binom{[n]}{k}$, let $\Delta_I(A)$
be the {\it Pl\"ucker coordinate}, that is, the maximal minor of the $k\times n$ matrix $A$ located in the column set $I$.
The map $A\mapsto (\Delta_I(A))$, where $I$ ranges over $\binom{[n]}{k}$,
induces the {\it Pl\"ucker embedding\/} $Gr_{k,n}\hookrightarrow \mathbb{RP}^{\binom{n}{k}-1}$.

We now describe several useful decompositions of the Grassmannian:
the matroid stratification, the Schubert decomposition, and the positroid
stratification.
Their relationship is as follows:
the  matroid stratification 
refines the positroid stratification which refines the Schubert
decomposition.  In Section \ref{sec:projections} we will describe the 
Deodhar decomposition, which is a refinement of the positroid 
stratification, and (as
verified in \cite{TW}) is refined by the matroid stratification.

\subsection{The matroid stratification of $Gr_{k,n}$}

\begin{definition}\label{def:matroid}
A \emph{matroid} of \emph{rank} $k$ on the set $[n]$ is a nonempty collection
$\M \subset \binom{[n]}{k}$ of $k$-element subsets in $[n]$, called \emph{bases}
of $\M$, that satisfies the \emph{exchange axiom}:\\
For any $I,J \in \M$ and $i \in I$ there exists $j\in J$ such that
$(I \setminus \{i\}) \cup \{j\} \in \M$.
\end{definition}

\begin{definition}
A \emph{loop} of a matroid on the set $[n]$ is an element $i\in [n]$ 
which is in every basis.  A \emph{coloop} is an element $i\in [n]$ which
is not in any basis.
\end{definition}

Given an element $A \in Gr_{k,n}$, there is an associated matroid
$\M_A$ whose bases are the $k$-subsets $I \subset [n]$ such that 
$\Delta_I(A) \neq 0$.

\begin{definition}
Let $\M \subset \binom{[n]}{k}$ be a matroid.
The \emph{matroid stratum}
$S_{\M}$ is defined to be 
$$S_{\M} = \{A \in Gr_{k,n} \ \vert \ \Delta_I(A) \neq 0 \text{ if and only if }
I\in \M \}.$$
This gives a stratification of $Gr_{k,n}$ called the 
\emph{matroid stratification}, or \emph{Gelfand-Serganova stratification}.
The matroids $\M$ with nonempty strata $S_{\M}$ are called \emph{realizable} over
$\R$.
\end{definition}

\subsection{The Schubert decomposition of $Gr_{k,n}$}

We now turn to the Schubert decomposition of the Grassmannian.
First recall that the partitions $\lambda \subset (n-k)^k$
are in bijection with $k$-element subset $I \subset [n]$.  
The boundary of the Young diagram of such a partition 
$\lambda$ forms a lattice path from the upper-right corner to the lower-left
corner of the rectangle $(n-k)^k$.  Let us label the $n$ steps 
in this path by the numbers $1,\dots,n$, and define
$I = I(\lambda)$ as the set of labels on the $k$ vertical steps in 
the path. Conversely, we let $\lambda(I)$ denote the 
partition corresponding to the subset $I$.

\begin{definition}
For each partition $\lambda \subset (n-k)^k$, one can define the 
\emph{Schubert cell} $\Omega_{\lambda}$ to be the set of 
all elements $A \in Gr_{k,n}$ such that when $A$ is represented
by a matrix in row-echelon form, it has pivots precisely
in the columns $I(\lambda)$.
As $\lambda$ ranges over the partitions contained in $(n-k)^k$,
this gives the \emph{Schubert decomposition} 
of the Grassmannian $Gr_{k,n}$, i.e.
\[
Gr_{k,n}=\bigsqcup_{\lambda\subset (n-k)^k}\,\Omega_{\lambda}.
\]
\end{definition}

\begin{definition}\label{def:prec}
Let $\{i_1, i_2,\dots,i_k\}$ and $\{j_1, j_2,\dots,j_k\}$ be 
two $k$-element subsets of $\{1,2,\dots,n\}$, such that 
$i_1 < i_2 < \dots < i_k$ and 
$j_1 < j_2 < \dots < j_k$.
We define the component-wise order $\preceq$ on $k$-element subsets
of $\{1,2,\dots,n\}$ as follows:
$$\{i_1, i_2, \dots, i_k \} \preceq \{j_1, j_2, \dots, j_k\}
\text{ if and only if }
i_1 \leq j_1, i_2 \leq j_2, \dots, \text{ and }i_k \leq j_k.$$
\end{definition}

\begin{lemma}\label{lem:order}
Let $A$ be an element of the Schubert cell 
$\Omega_{\lambda}$, and let $I = I(\lambda)$.  
If $\Delta_J(A) \neq 0$, then $I \preceq J$.
In particular, 
$$\Omega_{\lambda} = \{A \in Gr_{k,n} \ \vert \ I(\lambda) \text{ is
the lexicographically minimal base of }\M_A \}.$$
\end{lemma}
\begin{proof}
This follows immediately by considering the representation
of $A$ as a matrix
in row-echelon form.
\end{proof}

We now define the \emph{shifted linear order} $<_i$ (for $i\in [n]$) to be the total order on $[n]$
defined by $$i <_i i+1 <_i i+2 <_i \dots <_i n <_i 1 <_i \dots <_i i-1.$$
One can then define \emph{cyclically shifted Schubert cells} as follows.

\begin{definition} \label{def:Schubert}
For each partition $\lambda \subset (n-k)^k$ and $i \in [n]$, we define the 
\emph{cyclically shifted Schubert cell} $\Omega_{\lambda}^i$ by 
$$\Omega_{\lambda}^i = \{A \in Gr_{k,n} \ \vert \ I(\lambda) \text{ is
the lexicographically minimal base of }\M_A \text{ with respect to }<_i \}.$$
\end{definition}

Note that $\Omega_{\lambda} = \Omega_{\lambda}^1$.

\subsection{The positroid stratification of $Gr_{k,n}$}\label{sec:positroid}

The \emph{positroid stratification} of the real Grassmannian
$Gr_{k,n}$ is obtained by taking the
simultaneous refinement of the $n$ Schubert decompositions with respect to 
the $n$ shifted linear orders $<_i$.  This stratification 
was first considered
by Postnikov \cite{Postnikov}, who showed that
the strata are conveniently described 
in terms of 
\emph{Grassmann necklaces}, as well as \emph{decorated permutations}
and \emph{$\Le$-diagrams}.  Postnikov coined the terminology
\emph{positroid} because the intersection of the positroid stratification
with the \emph{totally non-negative part of the Grassmannian}
$(Gr_{k,n})_{\geq 0}$ gives a cell decomposition of 
$(Gr_{k,n})_{\geq 0}$  (whose cells are called 
\emph{positroid cells}).

\begin{definition}\cite[Definition 16.1]{Postnikov}
A \emph{Grassmann necklace} is a sequence
$\I = (I_1,\dots,I_n)$ of subsets $I_r \subset [n]$ such that,
for $i\in [n]$, if $i\in I_i$ then $I_{i+1} = (I_i \setminus \{i\}) \cup \{j\}$,
for some $j\in [n]$; and if $i \notin I_i$ then $I_{i+1} = I_i$.
(Here indices $i$ are taken modulo $n$.)  In particular, we have
$|I_1| = \dots = |I_n|$, which is equal to some $k \in [n]$.  We then say that 
$\I$ is a Grassmann necklace of \emph{type} $(k,n)$.
\end{definition}

\begin{example}\label{ex1}
$\mathcal I = (1257, 2357, 3457, 4567, 5678, 6789, 1789, 1289, 1259)$ is
an example of a Grassmann necklace of type $(4,9)$. 
\end{example}

\begin{lemma}\cite[Lemma 16.3]{Postnikov} \label{lem:necklace}
Given $A\in Gr_{k,n}$, 
let $\mathcal I(A) = (I_1,\dots,I_n)$ be the sequence of subsets in 
$[n]$ such that, for $i \in [n]$, $I_i$ is the lexicographically
minimal subset of $\binom{[n]}{k}$  with respect to the shifted linear order
$<_i$ such that $\Delta_{I_i}(A) \neq 0$.
Then $\I(A)$ is a Grassmann necklace of type $(k,n)$.
\end{lemma}

If $A$ is in the matroid stratum $S_{\M}$, 
we also use $\I_{\M}$ to denote the sequence
$(I_1,\dots,I_n)$ defined above.
This leads to the following description of the 
\emph{positroid stratification}
of $Gr_{k,n}$.

\begin{definition}\label{def:pos}
Let $\I = (I_1,\dots,I_n)$ be a Grassmann necklace of type $(k,n)$.
The \emph{positroid stratum}
$S_{\I}$ is defined to be 
$$S_{\I} = \{A \in Gr_{k,n} \ \vert \ \I(A) = \I \}.$$
\end{definition}

\begin{remark}
By comparing Definition \ref{def:pos} to Definition \ref{def:Schubert},
we see that given a Grassmann necklace $\I = (I_1,\dots,I_n)$,
$$S_{\I} = \bigcap_{i=1}^n ~\Omega_{\lambda(I_i)}^i\,.$$
In other words, each positroid stratum is an intersection of $n$ 
cyclically shifted Schubert cells.
\end{remark}


\begin{definition}\cite[Definition 13.3]{Postnikov}
A \emph{decorated permutation} $\pi^{:} = (\pi, col)$
is a permutation $\pi \in S_n$ together with a coloring
function $col$ from the set of fixed points
$\{i \ \vert \ \pi(i) = i\}$ to $\{1,-1\}$.  So
a decorated permutation is a permutation with fixed points
colored in one of two colors.  A \emph{weak excedance} of
$\pi^{:}$ is a pair $(i,\pi(i))$  such that either
$\pi(i)>i$ or $\pi(i)=i$ and $col(i)=1$. We call
$i$ the \emph{weak excedance position}.
If $\pi(i)>i$ (respectively $\pi(i)<i$) then 
$(i,\pi(i))$ is called an excedance (respectively, nonexcedance).
\end{definition}

\begin{example}\label{ex2}
The decorated permutation (written in one-line notation)
$(6,7,1,2,8,3,9,4,5)$ has no fixed points,
and 
four weak excedances, in positions $1, 2, 5$ and $7$.
\end{example}

\begin{definition}\cite[Definition 6.1]{Postnikov}\label{def:Le}
Fix $k$, $n$. If $\lambda$ is a partition, let
$Y_{\lambda}$ denote its Young diagram.  A {\it $\Le$-diagram}
$(\lambda, D)_{k,n}$ of type $(k,n)$
is a partition $\lambda$ contained in a $k \times (n-k)$ rectangle
together with a filling $D: Y_{\lambda} \to \{0,+\}$ which has the
{\it $\Le$-property}: 
there is no $0$ which has a $+$ above it and a $+$ to its
left.\footnote{This forbidden pattern is in the shape of a backwards $L$,
and hence is denoted $\Le$ and pronounced ``Le."}  (Here, ``above" means above and in the same column, and
``to its left" means to the left and in the same row.)
\end{definition}
In Figure \ref{LeDiagram} we give
an example of a $\Le$-diagram.
\begin{figure}[h]
\centering
\includegraphics[height=1.2in]{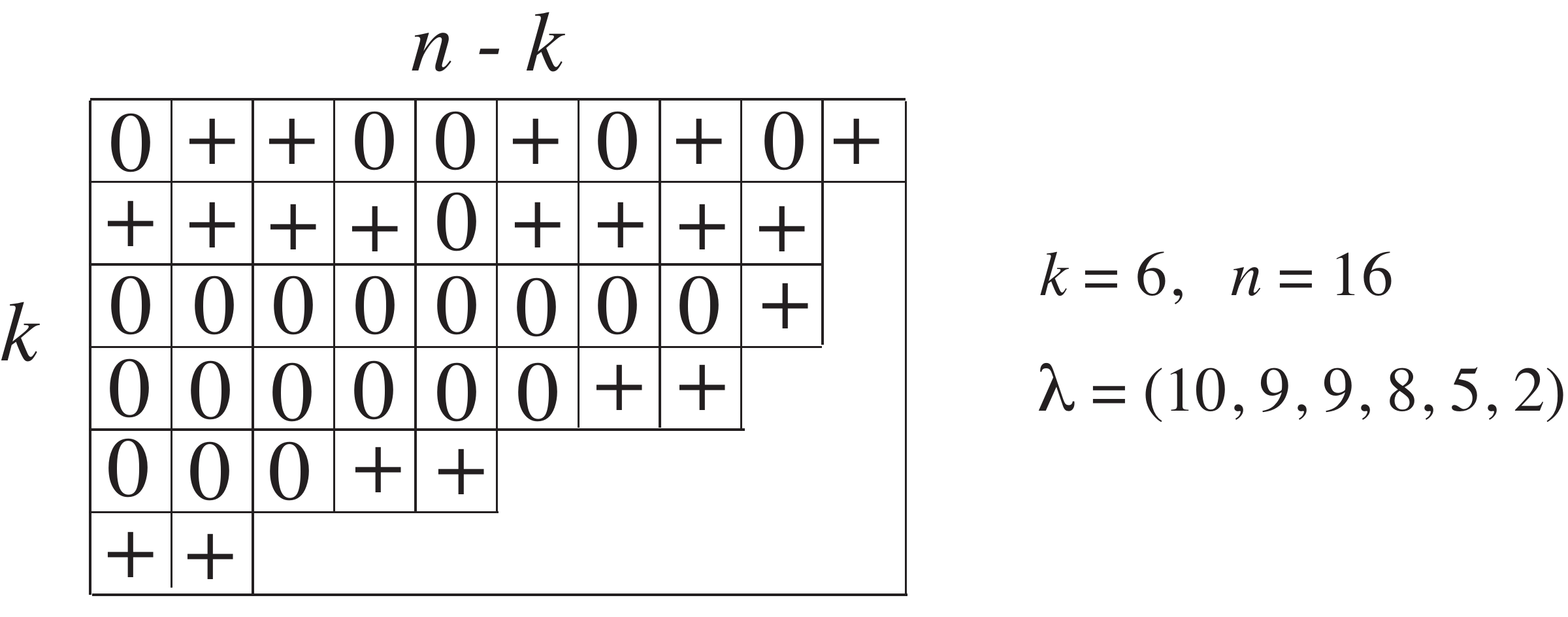}
\caption{A Le-diagram $L=(\lambda,D)_{k,n}$.
\label{LeDiagram}}
\end{figure}

We now review some of the bijections among these objects.

\begin{definition}\cite[Section 16]{Postnikov}\label{necklace-to-perm}
Given a Grassmann necklace $\mathcal I$, define
a decorated permutation $\pi^{:}=\pi^{:}(\mathcal I)$ by requiring that
\begin{enumerate}
\item if $I_{i+1} = (I_i \setminus \{i\}) \cup \{j\}$,
for $j \neq i$, then $\pi(j)=i$.
\footnote{Actually Postnikov's convention was to set $\pi(i)=j$ above,
so the decorated permutation we are associating is the inverse one to his.}
\item if $I_{i+1}=I_i$ and $i \in I_i$ then $\pi(i)=i$ is colored
 with $col(i)=1$.
\item if $I_{i+1}=I_i$ and $i \notin I_i$ then $\pi(i)=i$ is colored
 with $col(i)=-1$.
\end{enumerate}
As before, indices are taken modulo $n$.
\end{definition}
If $\pi^{:}=\pi^{:}(\mathcal I)$, then we also use
the notation
$S_{\pi^{:}}$ to refer to the positroid stratum
$S_{\I}$.  


\begin{example}\label{ex3}
Definition \ref{necklace-to-perm} carries the Grassmann necklace
of Example \ref{ex1} to the decorated permutation of Example \ref{ex2}.
\end{example}

\begin{lemma}\cite[Lemma 16.2]{Postnikov}\label{Postnikov-permutation}
The map $\mathcal I \to \pi^{:}(\mathcal I)$
is a bijection from
Grassmann necklaces $\mathcal I=(I_1,\dots,I_n)$
of size $n$ to
decorated permutations $\pi^{:}(\mathcal I)$ of size $n$.
Under this bijection, the weak excedances of
$\pi^{:}(\mathcal I)$ are in positions $I_1$.
\end{lemma}

\begin{remark}\label{perm-pivot}
Use the notation of Lemma \ref{Postnikov-permutation}.  It follows from the 
definition of the positroid stratification that if 
$A \in S_{\I}$ is written in row-echelon form, then the pivots are located
in position $I_1$.  It follows from Lemma \ref{Postnikov-permutation}
that the pivot positions coincide with the 
weak excedance positions of $\pi^{:}(\I)$.
\end{remark}

\subsection{Irreducible elements of $Gr_{k,n}$}


\begin{definition}
We say that a full rank $k \times n$ matrix 
is \emph{irreducible} if, after passing to its reduced
row-echelon form $A$, the matrix $A$ 
has the following properties:
\begin{enumerate}
\item Each column of $A$ contains at least one nonzero element.
\item Each row of $A$ contains at least one nonzero element in 
addition to the pivot.
\end{enumerate}
\end{definition}

An \emph{irreducible Grassmann necklace of type 
$(k,n)$} is a sequence $\mathcal I = (I_1,\dots,I_n)$ of 
subsets $I_r$ of $[n]$ of size $k$  such that, for $i\in [n]$, 
$I_{i+1}=(I_i \setminus \{i\}) \cup \{j\}$ for some $j\neq i$.
(Here indices $i$
are taken modulo $n$.)  
A \emph{derangement} $\pi=(\pi_1,\dots,\pi_n)$
is a permutation $\pi \in S_n$ which has no fixed points.


In the language of matroids,
an element $A \in S_{\M}$ is irreducible if and only if
the matroid $\M$ has no loops or coloops.
It is easy to see that if $A$ is irreducible, then 
$\I(A)$ is an irreducible Grassmann necklace and 
$\pi^:(\I)$ is a derangement.



\section{Projecting the Deodhar decomposition of $G/B$ to the Grassmannian}
\label{sec:project}

In this section we review Deodhar's decomposition of the 
flag variety  $G/B$ \cite{Deodhar}.  By projecting it, one may 
obtain a decomposition of any partial flag variety $G/P$ 
(and in particular the Grassmannian), 
obtaining the decomposition which Deodhar described in
\cite{Deodhar2}.
We also review the parameterizations of the components due to Marsh 
and Rietsch \cite{MR}.  

\subsection{The flag variety}

The following definitions can be made for any 
split, connected, simply connected, semisimple algebraic group over a field $\K$.
However this paper will be concerned with $G = \SL_n = \SL_n(\R)$.

We fix a maximal torus $T$, and opposite Borel subgroups $B^+$ and $B^-$, which
consist of the diagonal, upper-triangular, and lower-triangular
matrices, respectively.  
We let $U^+$ and $U^-$ be the unipotent radicals of $B^+$ and $B^-$; these
are the subgroups of upper-triangular and lower-triangular matrices with $1$'s on the diagonals.
For each $1 \leq i \leq n-1$ we have a homomorphism 
$\phi_i:{\rm SL}_2\to {\rm SL}_n$ such that
\[
\phi_i\begin{pmatrix} a& b\\c&d\end{pmatrix}=
\begin{pmatrix}
1 &             &       &       &           &     \\
   &\ddots  &        &      &            &        \\
   &             &   a   &   b  &          &       \\
   &             &   c    &   d  &         &       \\
   &            &          &       &  \ddots  &     \\
   &            &          &       &              & 1 
   \end{pmatrix} ~\in {\rm SL}_n,
\]
that is, $\phi_i$ replaces a $2\times 2$ block of the identity matrix with $\begin{pmatrix} a&b\\c&d\end{pmatrix}$.
Here $a$ is at the $(i+1)$st diagonal entry counting from the southeast corner.\footnote{Our
numbering differs from that in \cite{MR} in that
the rows of our matrices in $\SL_n$ are numbered from the bottom.}
We use this to construct $1$-parameter subgroups in $G$ 
(landing in $U^+$ and $U^-$,
respectively) defined by
\begin{equation*}
x_i(m) = \phi_i \left(
                   \begin{array}{cc}
                     1 & m \\ 0 & 1\\
                   \end{array} \right)  \text{ and }\ 
y_i(m) = \phi_i \left(
                   \begin{array}{cc}
                     1 & 0 \\ m & 1\\
                   \end{array} \right) ,\
\text{ where }m \in \R.
\end{equation*}
The datum $(T, B^+, B^-, x_i, y_i; i \in
I)$ for $G$ is called a {\it pinning}.  

Let $W$ denote the Weyl group $ N_G(T) / T$, 
where $N_G(T)$ is the normalizer of $T$.  
The simple reflections $s_i \in W$ are given explicitly by
$s_i:= \dot{s_i} T$ where $\dot{s_i} :=
                 \phi_i \left(
                   \begin{array}{cc}
                     0 & -1 \\ 1 & 0\\
                   \end{array} \right)$
and any $w \in W$ can be expressed as a product $w = s_{i_1} s_{i_2}
\dots s_{i_m}$ with $m=\ell(w)$ factors.  We set $\dot{w} =
\dot{s}_{i_1} \dot{s}_{i_2} \dots \dot{s}_{i_m}$.
For $G = \SL_n$, we have $W = \Sym_n$, the symmetric group on $n$ letters,
and $s_i$ is the transposition exchanging $i$ and $i+1$.



We can identify the flag variety $G/B$ with the variety
$\B$ of Borel subgroups, via
\begin{equation*}
gB \longleftrightarrow g \cdot B^+ := gB^+ g^{-1}.
\end{equation*}
We have two opposite Bruhat
decompositions of $\B$:
\begin{equation*}
\mathcal B=\bigsqcup_{w\in W} B^+\dot w\cdot B^+=\bigsqcup_{v\in W}
B^-\dot v\cdot B^+.
\end{equation*}
Note that $B^-\dot v\cdot B^+\cong\mathbb R^{\ell(w_0)-\ell(v)}$. The
closure relations for these opposite Bruhat cells are
given by $B^-\dot v'\cdot B^+\subset 
\overline{B^-\dot v\cdot B^+}$ if and only if $v\le v'$.
We define
\begin{equation*}
\mathcal R_{v,w}:=B^+\dot w\cdot B^+\cap B^-\dot v\cdot B^+,
\end{equation*}
the intersection of opposite Bruhat cells. This intersection is empty
unless $v\le w$, in which case it is smooth of dimension
$\ell(w)-\ell(v)$, see \cite{KaLus:Hecke2,Lusztig2}.
The strata $\mathcal R_{v,w}$ are often called \emph{Richardson varieties}.

\subsection{Distinguished expressions}

We now provide background on distinguished and positive
distinguished subexpressions, as in
\cite{Deodhar} and \cite{MR}. We will assume that the reader is familiar
with the (strong) Bruhat order $<$ on the Weyl group $W=\Sym_n$, and the 
basics of reduced expressions, as in  \cite{BB}.

Let $\w:= s_{i_1}\dots s_{i_m}$ be a reduced expression for $w\in W$.
We define a {\it subexpression} $\v$ of $\w$
to be a word obtained from the reduced expression $\w$ by replacing some of
the factors with $1$. For example, consider a reduced expression in $\Sym_4$, say $s_3
s_2 s_1 s_3 s_2 s_3$.  Then $s_3 s_2\, 1\, s_3 s_2\, 1$ is a
subexpression of $s_3 s_2 s_1 s_3 s_2 s_3$.
Given a subexpression $\v$, 
we set $v_{(k)}$ to be the product of the leftmost $k$ 
factors of $\v$, if $k \geq 1$, and $v_{(0)}=1$.
The following definition was given in \cite{MR}
and was implicit in \cite{Deodhar}.

\begin{definition}\label{d:Js}
Given a subexpression $\v$ of a reduced expression $\w=
s_{i_1} s_{i_2} \dots s_{i_m}$, we define
\begin{align*}
J^{\circ}_\v &:=\{k\in\{1,\dotsc,m\}\ |\  v_{(k-1)}<v_{(k)}\},\\
J^{\Box}_\v\, &:=\{k\in\{1,\dotsc,m\}\ |\  v_{(k-1)}=v_{(k)}\},\\
J^{\bullet}_\v &:=\{k\in\{1,\dotsc,m\}\ |\  v_{(k-1)}>v_{(k)}\}.
\end{align*}
The expression  $\v$
is called {\it
non-decreasing} if $v_{(j-1)}\le v_{(j)}$ for all $j=1,\dotsc, m$,
e.g.\ $J^{\bullet}_\v=\emptyset$.
\end{definition}

The following definition is from {\cite[Definition
2.3]{Deodhar}}:  
\begin{definition}[Distinguished subexpressions]
A subexpression $\v$ of $\w$ is called {\it distinguished}
if we have
\begin{equation}\label{e:dist}
v_{(j)}\le v_{(j-1)}\ s_{i_j}\qquad \text{for all
$~j\in\{1,\dotsc,m\}$}.
\end{equation}
In other words, if right multiplication by $s_{i_j}$ decreases the
length of $v_{(j-1)}$, then in a distinguished subexpression we
must have
$v_{(j)}=v_{(j-1)}s_{i_j}$.

We write $\v\prec\w$ if $\v$ is a distinguished subexpression of
$\w$.
\end{definition}

\begin{definition}[Positive distinguished subexpressions]
We call a
subexpression $\v$ of $\w$ a {\it positive distinguished subexpression} 
(or a PDS for short) if
 \begin{equation}\label{e:PositiveSubexpression}
v_{(j-1)}< v_{(j-1)}s_{i_j} \qquad \text{for all
$~j\in\{1,\dotsc,m\}$}.
 \end{equation}
In other words, it is distinguished and non-decreasing.
\end{definition}

\begin{lemma}\label{l:positive}\cite{MR}
Given $v\le w$ 
and a reduced expression $\w$ for $w$,
there is a unique PDS $\v_+$ for $v$ in $\w$.
\end{lemma}

\subsection{Deodhar components in the flag variety}

We now describe the Deodhar decomposition of the  flag variety.  This is a further
refinement of the decomposition of $G/B$ into Richardson varieties $\mathcal R_{v,w}$.
Marsh and Rietsch \cite{MR} gave explicit parameterizations for each Deodhar
component, identifying each one with a subset in the group.


\begin{definition}\cite[Definition 5.1]{MR}\label{d:factorization}
Let $\w=s_{i_1} \dots s_{i_m}$ be a reduced expression for $w$,
and let $\v$ be a distinguished subexpression.
Define a subset $G_{\v,\w}$ in $G$ by
\begin{equation}\label{e:Gvw}
G_{\v,\w}:=\left\{g= g_1 g_2\cdots g_m \left
|\begin{array}{ll}
 g_\ell= x_{i_\ell}(m_\ell)\dot s_{i_\ell}\inv& \text{ if $\ell\in J^{\bullet}_\v$,}\\
 g_\ell= y_{i_\ell}(p_\ell)& \text{ if $\ell\in J^{\Box}_\v$,}\\
 g_\ell=\dot s_{i_\ell}& \text{ if $\ell\in J^{\circ}_\v$,}
 \end{array}\quad \text{
for $p_\ell\in\R^*,\, m_\ell\in\R$. }\right. \right\}.
\end{equation}
There is an obvious map $(\R^*)^{|J^{\Box}_\v|}\times \R^{|J^{\bullet}_\v|}\to
G_{\v,\w}$ defined by the parameters $p_\ell$ and $m_\ell$ in
\eqref{e:Gvw}. 
For $v =w=1$ 
we define $G_{\v,\w}=\{1\}$.
\end{definition}

\begin{example}\label{ex:g}
Let $W=\Sym_5$, $\w = s_2 s_3 s_4 s_1 s_2 s_3$ and 
$\v = s_2 1 1 1 s_2 1$. 
Then the corresponding element $g\in G_{\v,\w}$ is given by $g=s_2 y_3(p_2)y_4(p_3)y_1(p_4)x_2(m_5)s_2^{-1}y_3(p_6)$,
which is 
\[
g=\begin{pmatrix}
1 & 0 & 0 & 0 & 0 \\
p_3 & 1 & 0 & 0 & 0 \\
0 & p_6 & 1 & 0 & 0 \\
p_2 p_3 & p_2-m_5 p_6 & -m_5 & 1 & 0 \\
0 & -p_4 p_6 & -p_4 & 0 & 1
\end{pmatrix}.
\]
\end{example}

The following result from \cite{MR} gives an explicit parametrization for
the Deodhar component $\mathcal R_{\v,\w}$.
We will take the description below as the \emph{definition}
of $\mathcal R_{\v,\w}$.

\begin{proposition}\label{p:parameterization}\cite[Proposition 5.2]{MR}
The map $(\R^*)^{|J^{\Box}_\v|}\times \R^{|J^{\bullet}_\v|}\to G_{\v,\w}$ from
Definition~\ref{d:factorization} is an isomorphism. The set
$G_{\v,\w}$ lies in $U^-\dot v\cap B^+\dot w B^+$, and the
assignment $g\mapsto g\cdot B^+$ defines an isomorphism
\begin{align}\label{e:parameterization}
G_{\v,\w}&~\overset\sim\To ~\mathcal R_{\v,\w}
\end{align}
between the subset $G_{\v,\w}$ of the group,
and the Deodhar component $\mathcal R_{\v,\w}$
 in $G/B$.
\end{proposition}

Suppose that for each $w\in W$ we choose a reduced expression 
$\w$ for $w$.  Then it follows from Deodhar's work (see \cite{Deodhar} and
\cite[Section 4.4]{MR}) that 
\begin{equation}\label{e:DeoDecomp}
\mathcal R_{v,w} = \bigsqcup_{\v \prec \w} \mathcal R_{\v,\w}\qquad  \text{ and }
\qquad  G/B=\bigsqcup_{w\in W}\left(\bigsqcup_{\v\prec \w} \mathcal
 R_{\v,\w}\right).
\end{equation}
These are called the \emph{Deodhar decompositions} of $\mathcal R_{v,w}$ and  $G/B$.

\begin{remark}\label{rem:KLpolys}
One may define the Richardson variety $\mathcal R_{v,w}$ over a finite field 
$\F_q$.  In this setting the number of points
determine the $R$-polynomials $R_{v,w}(q) = \#(\mathcal R_{v,w}(\F_q))$
introduced by Kazhdan and Lusztig \cite{KazLus} to give a formula
for the Kazhdan-Lusztig polynomials.  This was the original motivation
for Deodhar's work.
Therefore 
the isomorphisms
$\mathcal R_{\v,\w} \cong (\F_q^*)^{|J^{\Box}_\v|} 
                       \times \F_q^{|J^{\bullet}_\v|}$
together with the decomposition \eqref{e:DeoDecomp} 
give formulas for the $R$-polynomials.
\end{remark}

\begin{remark}\label{rem:dependence}
Note that the Deodhar decomposition of $\mathcal R_{v,w}$ depends on the choice of 
reduced expression for $w$.  
However, we will show in Proposition \ref{prop:independence} that 
its projection to the Grassmannian does not depend on the choice of reduced expression.
\end{remark}

\begin{remark} The Deodhar decomposition of the complete flag variety is not
a stratification -- e.g. the closure of a component is not a union of 
components \cite{Dudas}.
\end{remark}

This decomposition has a beautiful restriction to the 
totally non-negative part $(G/B)_{\geq 0}$ of $G/B$.
See \cite[Section 11]{MR} and also \cite{Rietsch} for more
definitions and details.
\begin{remark}\label{rem:TPcell}
Suppose we choose a reduced expression $\w$ for $w$, 
and for each $v \leq w$ we let $\v_+$ denote the unique 
positive distinguished subexpression for $v$ in $\w$.  
Note that $\v_+$ is non-decreasing so $J^{\bullet}_{\v_+} =\emptyset$.
Define $G^{>0}_{\v_+,\w}$ to be the subset of 
$G_{\v_+,\w}$ obtained by letting the parameters 
$p_{\ell}$ range over the positive reals.  
Let $\mathcal R^{>0}_{v,w}$ denote the image of 
$G^{>0}_{\v_+,\w}$  under the isomorphism 
$G_{\v_+,\w}\overset\sim\To \mathcal R_{\v_+,\w}$.
Then $\mathcal R^{>0}_{v,w}$ depends only on $v$ and $w$,
not on $\v_+$ and $\w$.  Moreover, 
the totally non-negative part $(G/B)_{\geq 0}$ of $G/B$
has a cell decomposition  
\begin{equation}
 (G/B)_{\geq 0}=\bigsqcup_{w\in W}\left(\bigsqcup_{v \leq w} \mathcal
 R^{>0}_{v,w}\right).
\end{equation}
\end{remark}

\subsection{Deodhar components in the Grassmannian}\label{sec:projections}

As we will explain in this section,
one obtains the Deodhar decomposition of the Grassmannian by 
projecting the Deodhar decomposition of the flag variety to the 
Grassmannian \cite{Deodhar2}.

The Richardson stratification of $G/B$ has an analogue for partial flag varieties $G/P_J$
introduced by Lusztig \cite{Lusztig2}.  Let $W_J$ be the parabolic subgroup of $W$ 
corresponding to $P_J$, and let $W^J$ be the set of minimal-length coset representatives
of $W/W_J$.  Then for each $w\in W^J$, the projection 
$\pi: G/B \to G/P_J$ is an isomorphism on each Richardson variety $\mathcal R_{v,w}$. 
Lusztig showed that we have a decomposition of the 
partial flag variety 
\begin{equation}\label{Richardson}
G/P_J = \bigsqcup_{w \in W^J} \left(\bigsqcup_{v \leq w} \pi({\mathcal R}_{v,w}) \right).
\end{equation}

Now consider the case that our partial flag variety is the Grassmannian
$Gr_{k,n}$ for $k<n$.
The corresponding parabolic subgroup of $W = \Sym_n$ is 
$W_k = \langle s_1,s_2,\dots,\hat{s}_{n-k},\dots,s_{n-1} \rangle$.
Let $W^k$ denote the set of minimal-length
coset representatives of $W/W_k$. 
Recall that a \emph{descent} of a permutation $\pi$
is a position $j$ such that $\pi(j)>\pi(j+1)$.
Then $W^k$ is the subset of permutations of $\Sym_n$
which have at most one descent; and that descent must be in position $n-k$.

Let $\pi_k: G/B \to Gr_{k,n}$ be the projection from the flag variety to the Grassmannian.
For each $w \in W^k$ and $v \leq w$, define 
$\mathcal P_{v,w} = \pi_k(\mathcal R_{v,w})$.  Then by \eqref{Richardson}
we have a decomposition 
\begin{equation}\label{projected-Richardson}
Gr_{k,n} = \bigsqcup_{w \in W^k} \left(\bigsqcup_{v \leq w} \mathcal P_{v,w} \right).
\end{equation}

\begin{remark}\label{rem:coincide}
The decomposition in \eqref{projected-Richardson} coincides with the
positroid stratification from Section \ref{sec:positroid}.  
This was verified in \cite[Theorem 5.9]{KLS}.  The appropriate bijection 
between the strata is defined in Lemma \ref{lem:bijection} below, and was first 
given in \cite[Lemma A.4]{Williams}.
\end{remark}

\begin{lemma}\cite[Lemma A.4]{Williams}\label{lem:bijection}
Let $\Q^k$ denote
the set of pairs $(v,w)$ where $v\in W$, $w\in W^k$,  and $v\leq w$;
let $\Dec_n^k$ denote the set of decorated permutations in $S_n$
with $k$ weak excedances.
We consider both sets as partially ordered sets, where the cover
relation corresponds to containment of closures of the corresponding 
strata.  Then there is an order-preserving bijection 
$\Phi$ from $\Q^k$ to $\Dec_n^k$ which is defined as follows.
Let $(v,w) \in \Q^J$.  Then $\Phi(v,w) = (\pi, col)$
where $\pi = v w^{-1}$.  We also let $\pi^:(v,w)$ denote $\Phi(v,w)$.
To define $col$, we color any fixed point that occurs
in one of the positions $w(1), w(2), \dots, w(n-k)$
with the color $-1$, and color any other fixed point with the color $1$.
\end{lemma}

Since $\pi_k$ is an isomorphism from $\mathcal R_{v,w}$ to $\mathcal P_{v,w}$, 
it also makes sense to consider projections of Deodhar components in $G/B$
to the Grassmannian.  
For each reduced decomposition $\w$ for $w \in W^k$, and each $\v \prec \w$,
we define $\mathcal P_{\v,\w} = \pi_k(\mathcal R_{\v,\w})$.
Now if for each $w \in W^k$ we choose a reduced decomposition $\w$, then we have 
\begin{equation}\label{e:ProjDeoDecomp}
\mathcal P_{v,w} = \bigsqcup_{\v \prec \w} \mathcal P_{\v,\w}\qquad \text{ and }\qquad
 Gr_{k,n} =\bigsqcup_{w\in W^k} \left(\bigsqcup_{\v\prec \w} \mathcal
 P_{\v,\w}\right).
\end{equation}

\begin{remark}\label{rem:Deodhar-positroid}
By Remark \ref{rem:coincide} and Lemma \ref{lem:bijection},
each projected Deodhar component $\mathcal P_{\v,\w}$ lies in the 
positroid stratum $S_{\pi^:}$, where $\pi^: = (\pi, col)$,
$\pi = vw^{-1}$, and $col$ is given by Lemma \ref{lem:bijection}.
Moreover, each Deodhar component is a union of matroid strata
\cite{TW}.  Therefore the Deodhar decomposition of the Grassmannian
refines the positroid stratification, and is refined by the
matroid stratification.
\end{remark}

Proposition \ref{p:parameterization} gives us a concrete way to 
think about the projected Deodhar components $\mathcal P_{\v,\w}$.  
The projection $\pi_k: G/B \to Gr_{k,n}$ maps 
each $g \in G_{\v,\w}$ to the span of its leftmost $k$ columns.  
More specifically, it maps 
\[
g=\begin{pmatrix}
g_{n,n} & \dots& g_{n,n-k+1} & \dots & g_{n,1}  \\
\vdots & &\vdots&  & \vdots  \\
g_{1,n} & \dots& g_{1,n-k+1} & \dots & g_{1,1} \\
\end{pmatrix}
\quad \longrightarrow\quad
A=
\begin{pmatrix}
g_{1,n-k+1}&  \dots & g_{n,n-k+1}\\
\vdots &  & \vdots\\
g_{1,n} &  \dots & g_{n,n} \\
\end{pmatrix}
\]
Alternatively, we may identify $A\in Gr_{k,n}$ with its image in the Pl\"ucker 
embedding. Let $e_i$ denote the column vector in $\R^n$ 
such that the $i$th entry from the bottom contains a $1$, and all other 
entries are $0$, e.g. $e_n=(1,0, \ldots,0)^T$, the transpose of the row vector $(1,0,\ldots,0)$.
Then the projection $\pi_k$ maps
each $g \in G_{\v,\w}$ (identified with $g \cdot B^+ \in \mathcal R_{\v,\w}$) to 
\begin{align}\label{Plucker}
g \cdot e_{n-k+1} \wedge \ldots \wedge e_{n} &=
\sum_{1\le 
j_1<\ldots<j_k\le n}\Delta_{j_1,\ldots,j_k}(A)e_{j_1}\wedge \cdots\wedge e_{j_k}.
\end{align}
That is, the Pl\"ucker coordinate $\Delta_{j_1,\ldots,j_k}(A)$ is given by
\[
\Delta_{j_1,\ldots,j_k}(A)=\langle e_{j_1}\wedge\cdots\wedge e_{j_k},\,
g\cdot e_{n-k+1}\wedge\cdots\wedge e_{n}\rangle,
\]
where $\langle\cdot,\cdot\rangle$ is the usual inner product on $\wedge^k\mathbb{R}^n$.

\begin{example}
We continue Example \ref{ex:g}.  Note that 
$w \in W^k$ where $k=2$.  Then the map $\pi_2: G_{\v,\w}\to Gr_{2,5}$ 
is given by
\[
g=\begin{pmatrix}
1 & 0 & 0 & 0 & 0 \\
p_3 & 1 & 0 & 0 & 0 \\
0 & p_6 & 1 & 0 & 0 \\
p_2 p_3 & p_2-m_5 p_6 & -m_5 & 1 & 0 \\
0 & -p_4 p_6 & -p_4 & 0 & 1
\end{pmatrix} \quad\longrightarrow \quad
A = 
\begin{pmatrix}
-p_4p_6&p_2-m_5 p_6 & p_6 & 1 & 0\\
0 & p_2 p_3 & 0 & p_3 & 1 \\
\end{pmatrix}.
\]
\end{example}



\section{Combinatorics of projected Deodhar components in the Grassmannian}\label{Deodhar-combinatorics}
In this section we explain how to index the Deodhar components in
the Grassmannian $Gr_{k,n}$ by certain
tableaux.  We will display the tableaux in two equivalent ways -- as 
fillings of Young diagrams by $+$'s and $0$'s, which we call 
{\emph{Deodhar diagrams}}, and by fillings of Young diagrams by 
empty boxes, \raisebox{0.12cm}{\hskip0.14cm\circle*{7}\hskip-0.1cm}'s and \raisebox{0.12cm}{\hskip0.14cm\circle{7}\hskip-0.1cm}'s, which we call {\emph{Go-diagrams}}.  We refer to the symbols
 \raisebox{0.12cm}{\hskip0.14cm\circle*{7}\hskip-0.1cm} and \raisebox{0.12cm}{\hskip0.14cm\circle{7}\hskip-0.1cm}  as \emph{black} and \emph{white stones}.

Recall that $W_k = \langle s_1,s_2,\dots,\hat{s}_{n-k},\dots,s_{n-1} \rangle$
is a parabolic subgroup of $W =\Sym_n$ and $W^k$ is the set of minimal-length
coset representatives of $W/W_k$. 

An element $w \in W$ is {\it fully commutative} if 
every pair of reduced words for $w$ are related by
a sequence of relations of the form $s_i s_j = s_j s_i$.  
The following result is due to Stembridge \cite{Ste} and Proctor~\cite{Pro}.

\begin{theorem}\label{th:fully}
$W^k$ consists of fully commutative
elements.  Furthermore the Bruhat order on $W^k$ 
is a distributive lattice.
\end{theorem}

Let $Q^k$ be the poset  such that $W^k =
J(Q^k)$, where $J(P)$ denotes the distributive lattice of upper order
ideals in $P$.  The figure below (at the left) shows an example
of the Young diagram of $Gr_{3,8}$.
(The reader should temporarily ignore the labeling of boxes by 
$s_i$'s.)
The Young diagram should be interpreted as follows: each
box represents an element of the poset $Q^k$, and if $b_1$ and $b_2$
are two adjacent boxes such that $b_2$ is immediately to the left or
immediately above $b_1$, we have a cover relation $b_1 \lessdot b_2$
in $Q^k$. The partial order on $Q^k$ is the transitive closure of
$\lessdot$. 
Note that the 
minimal and maximal elements of $Q^k$
are the lower right and upper left boxes, respectively.

We now state some facts about $Q^k$ which can be found in
\cite{Ste}.  Let $w_0^k \in W^k$ denote the longest element in
$W^k$.  The simple generators $s_i$ used in a reduced expression
for $w_0^k$ can be used to label  $Q^k$ in a way which
reflects the bijection between the minimal length coset
representatives $w \in W^k$ and upper order ideals $O_w \subset
Q^k$.  Such a labeling is shown in the figure below.
If $b \in O_w$
is a box labelled by $s_i$, we denote the simple generator labeling
$b$ by $s_b:= s_i$.
Given this labeling,
if $O_w$ is an upper order ideal in $Q^k$,
the set of linear extensions
$\{e: O_w \to [1,\ell(w)]\}$ of $O_w$ are in bijection with the
reduced words $R(w)$ of $w$: the reduced word (written down from 
left to right) is obtained by reading the labels of $O_w$ in the order
specified by $e$.  We will call the linear extensions of $O_w$ {\it
reading orders}. 
\setlength{\unitlength}{0.7mm}
\begin{center}
  \begin{picture}(60,30)
  
  \put(5,32){\line(1,0){45}}
  \put(5,23){\line(1,0){45}}
  \put(5,14){\line(1,0){45}}
  \put(5,5){\line(1,0){45}}
  \put(5,5){\line(0,1){27}}
  \put(14,5){\line(0,1){27}}
  \put(23,5){\line(0,1){27}}
  \put(32,5){\line(0,1){27}}
  \put(41,5){\line(0,1){27}}
  \put(50,5){\line(0,1){27}}

  \put(7,27){$s_5$}
  \put(16,27){$s_4$}
  \put(25,27){$s_3$}
  \put(34,27){$s_2$}
  \put(43,27){$s_1$}
  \put(7,18){$s_6$}
  \put(16,18){$s_5$}
  \put(25,18){$s_4$}
  \put(34,18){$s_3$}
  \put(43,18){$s_2$}
  \put(7,9){$s_7$}
  \put(16,9){$s_6$}
  \put(25,9){$s_5$}
  \put(34,9){$s_4$}
  \put(43,9){$s_3$}
 
  \end{picture} 
  \quad
  \begin{picture}(60,40)
  
  \put(5,32){\line(1,0){45}}
  \put(5,23){\line(1,0){45}}
  \put(5,14){\line(1,0){45}}
  \put(5,5){\line(1,0){45}}
  \put(5,5){\line(0,1){27}}
  \put(14,5){\line(0,1){27}}
  \put(23,5){\line(0,1){27}}
  \put(32,5){\line(0,1){27}}
  \put(41,5){\line(0,1){27}}
  \put(50,5){\line(0,1){27}}

  \put(7,26){$15$}
  \put(16,26){$14$}
  \put(25,26){$13$}
  \put(34,26){$12$}
  \put(43,26){$11$}
  \put(7,17){$10$}
  \put(16,17){$~9$}
  \put(25,17){$~8$}
  \put(34,17){$~7$}
  \put(43,17){$~6$}
  \put(7,8){$~5$}
  \put(16,8){$~4$}
  \put(25,8){$~3$}
  \put(34,8){$~2$}
  \put(43,8){$~1$}
  
 \end{picture} 
  \quad
  \begin{picture}(60,40)
  
  \put(5,32){\line(1,0){45}}
  \put(5,23){\line(1,0){45}}
  \put(5,14){\line(1,0){45}}
  \put(5,5){\line(1,0){45}}
  \put(5,5){\line(0,1){27}}
  \put(14,5){\line(0,1){27}}
  \put(23,5){\line(0,1){27}}
  \put(32,5){\line(0,1){27}}
  \put(41,5){\line(0,1){27}}
  \put(50,5){\line(0,1){27}}
  
  \put(7,26){$15$}
  \put(16,26){$12$}
  \put(25,26){$~9$}
  \put(34,26){$~6$}
  \put(43,26){$~3$}
  \put(7,17){$14$}
  \put(16,17){$11$}
  \put(25,17){$~8$}
  \put(34,17){$~5$}
  \put(43,17){$~2$}
  \put(7,8){$13$}
  \put(16,8){$10$}
  \put(25,8){$~7$}
  \put(34,8){$~4$}
  \put(43,8){$~1$}

  \end{picture} 
\end{center}

\begin{remark}
The upper order ideals of $Q^k$ can be identified with the Young diagrams
contained in a $k \times (n-k)$ rectangle,
and the linear extensions
of $O_w$ can be identified with the \emph{reverse} standard tableaux of shape $O_w$, i.e. entries decrease from left to right in rows and from top to bottom
in columns.
\end{remark}

\subsection{$\oplus$-diagrams and Deodhar diagrams}

The goal of this section is to identify subexpressions of reduced words
for elements of $W^k$ 
with certain fillings of the boxes of upper order ideals of $Q^k$.  In particular
we will be concerned with distinguished subexpressions.  

\begin{definition}\cite[Definition 4.3]{LW}
Let $O_w$ be an upper order ideal of $Q^k$, where $w \in W^k$.
An $\oplus$-diagram (``o-plus diagram") of shape $O_w$ is a filling of the boxes of $O_w$
with the symbols $0$ and $+$.
\end{definition}
Clearly there are $2^{\ell(w)}$ $\oplus$-diagrams of shape $O_w$.
The value of an $\oplus$-diagram $D$ at a box $x$ is denoted $D(x)$.
Let $e$ be a reading order for $O_w$; this gives rise to a reduced
expression $\w = \w_e$ for $w$.  The $\oplus$-diagrams $D$ of shape
$O_w$ are in bijection with subexpressions $\v(D)$ of $\w$: we will
make the convention that if a box $b \in O_w$
is filled with a $0$ then the corresponding simple generator $s_b$
is present in the subexpression, while if $b$ is filled with a $+$
then we omit the corresponding simple generator.  The subexpression
$\v(D)$ in turn defines a Weyl group element $v:= v(D) \in W$,
where $v \leq w$.

\begin{example}\label{ex4-1}
Consider the upper order ideal $O_w$ which is $Q^k$ itself for $\Sym_5$ and $k=2$.
Then $Q^k$ is the poset shown in the left diagram.  Let us choose the reading order (linear extension) indicated by the
labeling shown in the right diagram.
\setlength{\unitlength}{0.7mm}
\begin{center}
    \begin{picture}(50,25)
  
  \put(5,23){\line(1,0){27}}
  \put(5,14){\line(1,0){27}}
  \put(5,5){\line(1,0){27}}
  \put(5,5){\line(0,1){18}}
  \put(14,5){\line(0,1){18}}
  \put(23,5){\line(0,1){18}}
  \put(32,5){\line(0,1){18}}

  \put(7,18){$s_3$}
  \put(16,18){$s_2$}
  \put(25,18){$s_1$}
  \put(7,8){$s_4$}
  \put(16,8){$s_3$}
  \put(25,8){$s_2$}
 
  \end{picture} 
\qquad
  \begin{picture}(50,25)
  
  \put(5,23){\line(1,0){27}}
  \put(5,14){\line(1,0){27}}
  \put(5,5){\line(1,0){27}}
  \put(5,5){\line(0,1){18}}
  \put(14,5){\line(0,1){18}}
  \put(23,5){\line(0,1){18}}
  \put(32,5){\line(0,1){18}}

  \put(8,17){$6$}
  \put(17,17){$5$}
  \put(26,17){$4$}
  \put(8,8){$3$}
  \put(17,8){$2$}
  \put(26,8){$1$}
  
  \end{picture} 

\end{center}
Then the $\oplus$-diagrams given by
\setlength{\unitlength}{0.7mm}
\begin{center}
    \begin{picture}(40,25)
  
  \put(5,23){\line(1,0){27}}
  \put(5,14){\line(1,0){27}}
  \put(5,5){\line(1,0){27}}
  \put(5,5){\line(0,1){18}}
  \put(14,5){\line(0,1){18}}
  \put(23,5){\line(0,1){18}}
  \put(32,5){\line(0,1){18}}

  \put(8,17){$0$}
  \put(17,17){$0$}
  \put(26,17){$0$}
  \put(8,8){$0$}
  \put(17,8){$0$}
  \put(26,8){$0$}
 
  \end{picture} 
\quad
  \begin{picture}(40,25)
  
  \put(5,23){\line(1,0){27}}
  \put(5,14){\line(1,0){27}}
  \put(5,5){\line(1,0){27}}
  \put(5,5){\line(0,1){18}}
  \put(14,5){\line(0,1){18}}
  \put(23,5){\line(0,1){18}}
  \put(32,5){\line(0,1){18}}

  \put(8,17){$0$}
  \put(17,17){$+$}
  \put(26,17){$0$}
  \put(8,8){$0$}
  \put(17,8){$0$}
  \put(26,8){$+$}
  
  \end{picture} 
  \quad
\begin{picture}(40,25)
  
  \put(5,23){\line(1,0){27}}
  \put(5,14){\line(1,0){27}}
  \put(5,5){\line(1,0){27}}
  \put(5,5){\line(0,1){18}}
  \put(14,5){\line(0,1){18}}
  \put(23,5){\line(0,1){18}}
  \put(32,5){\line(0,1){18}}

  \put(8,17){$0$}
  \put(17,17){$+$}
  \put(26,17){$0$}
  \put(8,8){$+$}
  \put(17,8){$0$}
  \put(26,8){$+$}
  
  \end{picture} 
    \quad
\begin{picture}(40,25)
  
  \put(5,23){\line(1,0){27}}
  \put(5,14){\line(1,0){27}}
  \put(5,5){\line(1,0){27}}
  \put(5,5){\line(0,1){18}}
  \put(14,5){\line(0,1){18}}
  \put(23,5){\line(0,1){18}}
  \put(32,5){\line(0,1){18}}

  \put(8,17){$+$}
  \put(17,17){$+$}
  \put(26,17){$0$}
  \put(8,8){$+$}
  \put(17,8){$0$}
  \put(26,8){$+$}
  
  \end{picture} 

\end{center}
correspond to the expressions $s_2 s_3 s_4 s_1 s_2 s_3$, $1 s_3 s_4 s_1 1 s_3$,
$1 s_3 1 s_1 1 s_3$, and 
$1 s_3 1 s_1 1 1$.  The first and second are
PDS's (so in particular are distinguished); the third one is not a PDS but it is distinguished; 
and the fourth is not distinguished.
\end{example}

Parts (1) and (2) of this proposition come from \cite[Lemma 4.5 and Proposition 4.6]{LW}.
\begin{proposition}\label{p:whether}
If $b, b' \in O_w$ are two incomparable boxes,  $s_b$ and
$s_{b'}$ commute.  Furthermore, if 
$D$ is an $\oplus$-diagram, then 
\begin{enumerate}
\item
the element $v:=v(D)$ is independent of the choice of reading
word $e$.
\item
whether $\v(D)$ is a PDS depends only on $D$ (and not $e$).
\item whether $\v(D)$ is distinguished depends only on $D$ (and not on $e$).
\end{enumerate}
\end{proposition}
\begin{proof}
The commutation of $s_b$ and $s_{b'}$ follows by inspection.
For part (1), note that two linear extensions of the same poset
(viewed as permutations of the elements of
the poset) can
be connected via transpositions of pairs of incomparable elements.
Therefore $v(D)$ is independent
of the choice of reading word.

Suppose $D$ is an $\oplus$-diagram of shape $O_w$, and consider the
reduced expression $\w:=\w_e = s_{i_1} \dots s_{i_n}$ corresponding
to a linear extension $e$. Suppose $\v(D)$ is a PDS 
of $\w$.
For part (2), it suffices to show that if we swap the $k$-th and
\mbox{$(k+1)$-st} letters of both $\w$ and $\v(D)$, where these
positions correspond to incomparable boxes in $O_w$, then the
resulting subexpression $\v'$ will be a PDS 
of the resulting reduced
expression $\w'$.  If we examine the four cases (based on whether
the $k$-th and $(k+1)$-st letters of $\v(D)$ are $1$ or $s_{i_k}$)
it is clear from the definition  that $\v'$ is a PDS.  The same argument
holds if $\v(D)$ is distinguished.
\end{proof}

This leads to the following definitions. Note that by Theorem \ref{characterize-Le},
Definitions \ref{def:Le2} and \ref{def:Le} agree.
\begin{definition}\cite[Definition 4.7]{LW}\label{def:Le2}
A \emph{$\Le$-diagram} of shape $O_w$ is an $\oplus$-diagram $D$ of shape
$O_w$ such that $\v(D)$ is a PDS.
\end{definition}

\begin{definition}
A \emph{Deodhar diagram} of shape $O_w$ is an $\oplus$-diagram $D$ of shape
$O_w$ such that $\v(D)$ is distinguished.
\end{definition}

\begin{theorem}\cite[Theorem 5.1]{LW} and \cite[Lemma 19.3]{Postnikov}\label{characterize-Le}
An $\oplus$-diagram is a $\Le$-diagram if and only if there is no $0$
which has a $+$ above it (in the same row) and a $+$ to its left (in the same column).
\end{theorem}

Theorem \ref{characterize-Le} motivates the following open 
problem (which is slightly reformulated in Problem \ref{prob:Go}).
\begin{problem}
Find an analogue of 
Theorem \ref{characterize-Le} for Deodhar diagrams which characterizes them 
by forbidden patterns.
\end{problem}

\begin{definition}\label{def:DtoPi}
Let $O_w$ be an upper order ideal of $Q^k$, where 
$w\in W^k$ and $W = S_n$.
Consider a Deodhar diagram $D$ of shape $O_w$;
this is contained in a $k \times (n-k)$ rectangle,
and the shape $O_w$ gives rise to a lattice path
from the northeast corner to the 
southwest corner of the rectangle.  Label the steps
of that lattice path from $1$ to $n$; this gives
a natural labeling to every row and column of the rectangle.
We now let $v$ be the permutation with 
reduced decomposition $\v(D)$, and 
we define $\pi^:(D)$ to be the decorated permutation
$(\pi(D),col)$ where $\pi = \pi(D) = vw^{-1}$.  The fixed points of $\pi$
correspond precisely to rows and columns of the rectangle with no $+$'s.
If there are no $+$'s in the row (respectively, column) labeled by $h$,
then $\pi(h)=h$ and this fixed point gets colored with color $1$ (respectively,
$-1$.)
\end{definition}

\begin{remark}
It follows from Remark \ref{rem:Deodhar-positroid} and the way we defined
Deodhar diagrams that the projected Deodhar component $\mathcal P_D$ corresponding to 
$D$ is contained in the positroid stratum $S_{\pi^:(D)}$.
\end{remark}

\subsection{From Deodhar diagrams to Go-diagrams 
and labeled Go-diagrams}\label{subsec:Go}

It will  be useful for us to depict Deodhar diagrams in a 
slightly different way.
Consider the distinguished subexpression $\v$ of $\w$: 
for each $k \in J_{\v}^{\circ}$ we will place a \raisebox{0.12cm}{\hskip0.15cm\circle{4}\hskip-0.15cm} in the corresponding box;
for each 
$k \in J_{\v}^{\bullet}$ we will place a \raisebox{0.12cm}{\hskip0.15cm\circle*{4}\hskip-0.15cm} in the corresponding box of $O_w$; 
and for each $k \in J_{\v}^{\Box}$ we will leave the corresponding box blank.
We call the resulting diagram a \emph{Go-diagram}, and refer to the symbols
\wstn ~ and \bstn~  
as {\it white} and {\it  black stones}.

\begin{remark}
Note that a Go-diagram has no black stones if and only if it corresponds to a 
Deodhar diagram $D$ such that $\v(D)$ is a PDS, i.e. a $\Le$-diagram.  Therefore,
slightly abusing terminology,
we will often refer to a Go-diagram with no black stones as a 
$\Le$-diagram.\footnote{Since
$\Le$-diagrams are a special case of Go-diagrams, one might also refer to them
as \emph{Lego} diagrams.}
\end{remark}

Note that the Go-diagrams corresponding to the first three $\oplus$-diagrams
in Example \ref{ex4-1} are 
\setlength{\unitlength}{0.7mm}
\begin{center}
    \begin{picture}(45,30)
  
  \put(5,25){\line(1,0){30}}
  \put(5,15){\line(1,0){30}}
  \put(5,5){\line(1,0){30}}
  \put(5,5){\line(0,1){20}}
  \put(15,5){\line(0,1){20}}
  \put(25,5){\line(0,1){20}}
  \put(35,5){\line(0,1){20}}

  \put(8,20){\hskip0.15cm\circle{5}}
  \put(18,20){\hskip0.15cm\circle{5}}
  \put(28,20){\hskip0.15cm\circle{5}}
  \put(8,10){\hskip0.15cm\circle{5}}
  \put(18,10){\hskip0.15cm\circle{5}}
  \put(28,10){\hskip0.15cm\circle{5}}
 
  \end{picture} 
\qquad
     \begin{picture}(45,30)
  
  \put(5,25){\line(1,0){30}}
  \put(5,15){\line(1,0){30}}
  \put(5,5){\line(1,0){30}}
  \put(5,5){\line(0,1){20}}
  \put(15,5){\line(0,1){20}}
  \put(25,5){\line(0,1){20}}
  \put(35,5){\line(0,1){20}}

  \put(8,20){\hskip0.15cm\circle{5}}
  \put(18,20){}
  \put(28,20){\hskip0.15cm\circle{5}}
  \put(8,10){\hskip0.15cm\circle{5}}
  \put(18,10){\hskip0.15cm\circle{5}}
  \put(28,10){}
 
  \end{picture} 
\quad
    \begin{picture}(45,30)
  
  \put(5,25){\line(1,0){30}}
  \put(5,15){\line(1,0){30}}
  \put(5,5){\line(1,0){30}}
  \put(5,5){\line(0,1){20}}
  \put(15,5){\line(0,1){20}}
  \put(25,5){\line(0,1){20}}
  \put(35,5){\line(0,1){20}}

  \put(8,20){\hskip0.15cm\circle*{5}}
  \put(18,20){}
  \put(28,20){\hskip0.15cm\circle{5}}
  \put(8,10){}
  \put(18,10){\hskip0.15cm\circle{5}}
  \put(28,10){}
 
  \end{picture} 
\end{center}

\begin{problem}\label{prob:Go}
Characterize the fillings of Young diagrams by blank boxes, white stones, and 
black stones which are Go-diagrams.
\end{problem}

\begin{remark}\label{rem:KLpolys2}
Recall from Remark \ref{rem:KLpolys}
that the isomorphisms
$\mathcal R_{\v,\w} \cong (\F_q^*)^{|J^{\Box}_\v|} \times \F_q^{|J^{\bullet}_\v|}$
together with the decomposition \eqref{e:DeoDecomp} 
give formulas for the $R$-polynomials.  Therefore a good combinatorial 
characterization of the Go-diagrams (equivalently, Deodhar diagrams) contained in a given Young diagram
could lead to explicit formulas for the corresponding $R$-polynomials.
\end{remark}

If we choose a reading order of $O_w$, then we will also associate
to a Go-diagram of shape $O_w$ a \emph{labeled Go-diagram},
as defined below.  Equivalently,
a labeled Go-diagram is associated to a pair $(\v,\w)$.

\begin{definition}\label{def:pi}
Given a reading order of $O_w$ and a Go-diagram of shape $O_w$,
we obtain a \emph{labeled Go-diagram} by replacing each \raisebox{0.12cm}{\hskip0.14cm\circle{4}\hskip-0.15cm} with a $1$,
each \raisebox{0.12cm}{\hskip0.14cm\circle*{4}\hskip-0.15cm} with a $-1$, and putting a $p_i$ in each blank square $b$,
where the subscript $i$ corresponds to the label of $b$ inherited from the 
linear extension.
\end{definition}

The labeled Go-diagrams corresponding to 
the examples above using the reading order from
Example \ref{ex4-1} are:
\setlength{\unitlength}{0.7mm}
\begin{center}
    \begin{picture}(45,30)
  
  \put(5,25){\line(1,0){30}}
  \put(5,15){\line(1,0){30}}
  \put(5,5){\line(1,0){30}}
  \put(5,5){\line(0,1){20}}
  \put(15,5){\line(0,1){20}}
  \put(25,5){\line(0,1){20}}
  \put(35,5){\line(0,1){20}}

  \put(8,19){1}
  \put(18,19){1}
  \put(28,19){1}
  \put(8,9){1}
  \put(18,9){1}
  \put(28,9){1}
 
  \end{picture} 
\qquad
     \begin{picture}(45,30)
  
  \put(5,25){\line(1,0){30}}
  \put(5,15){\line(1,0){30}}
  \put(5,5){\line(1,0){30}}
  \put(5,5){\line(0,1){20}}
  \put(15,5){\line(0,1){20}}
  \put(25,5){\line(0,1){20}}
  \put(35,5){\line(0,1){20}}

  \put(8,19){1}
  \put(18,19){$p_5$}
  \put(28,19){$1$}
  \put(8,9){$1$}
  \put(18,9){$1$}
  \put(28,9){$p_1$}
 
  \end{picture} 
\quad
    \begin{picture}(45,30)
  
  \put(5,25){\line(1,0){30}}
  \put(5,15){\line(1,0){30}}
  \put(5,5){\line(1,0){30}}
  \put(5,5){\line(0,1){20}}
  \put(15,5){\line(0,1){20}}
  \put(25,5){\line(0,1){20}}
  \put(35,5){\line(0,1){20}}

  \put(7,19){$-1$}
  \put(18,19){$p_5$}
  \put(28,19){$1$}
  \put(8,9){$p_3$}
  \put(18,9){$1$}
  \put(28,9){$p_1$}
 
  \end{picture} 
\end{center}

In  future work we intend to explore further aspects of Go-diagrams and 
 Deodhar strata.

\subsection{The projected Deodhar decomposition does not depend on 
the expressions $\w$}

Recall from Remark \ref{rem:dependence} that the Deodhar decomposition
depends on the choices of reduced decompositions $\w$ of each $w\in W$.
However, its projection to the Grassmannian has a nicer behavior.

\begin{proposition}\label{prop:independence}
Let $w \in W^k$ and choose a reduced expression $\w$ for $w$.  Then
the components of 
$\bigsqcup_{\v\prec \w} \mathcal
 R_{\v,\w}$
do not depend on $\w$, only on $w$.
\end{proposition}

\begin{proof}
Recall from Theorem \ref{th:fully} that 
any two reduced expressions of $w \in W^k$ can be obtained from 
each other by a sequence
of commuting moves ($s_i s_j = s_j s_i$ where $|i-j| \geq 2$).  
And it is easy to check that if $s_i s_j = s_j s_i$, then 
\begin{enumerate}
\item $y_i(a) y_j(b) = y_j(b) y_i(a)$
\item $y_i(a) \dot s_j = \dot s_j y_i(a)$
\item $(x_i(a) \dot s_i^{-1}) \dot s_j = \dot s_j (x_i(a) \dot s_i^{-1})$
\item $(x_i(a) \dot s_i^{-1}) y_j(b) = y_j(b) (x_i(a) \dot s_i^{-1})$.
\end{enumerate}
The result now follows from Definition \ref{d:factorization} and Proposition
\ref{p:parameterization}.
\end{proof}


\section{Pl\"ucker coordinates and positivity tests for  projected Deodhar components}\label{sec:positivitytest}

Consider $\mathcal P_{\v,\w} \subset Gr_{k,n}$, where 
$\w$ is a reduced expression for $w \in W^k$ and $\v \prec \w$.
In this section we will provide some formulas for the Pl\"ucker coordinates
of the elements of $\mathcal P_{\v,\w}$, in terms of the parameters used
to define $G_{\v,\w}$.  Some of these formulas are related to corresponding
formulas for $G/B$ in \cite[Section 7]{MR}.

\subsection{Formulas for Pl\"ucker coordinates}

\begin{lemma}\label{lem:minmax}
Choose any element $A$ of $\mathcal P_{\v,\w} \subset Gr_{k,n}$. Let 
\begin{equation*}
I = w\,\{n-k+1,\dots,n-1,n\} \qquad \text{ and }\qquad
I' = v\,\{n-k+1,\dots,n-1,n\}.
\end{equation*}
Then if $\Delta_J(A) \neq 0$, we have
$I \preceq J \preceq I',$ where $\preceq$ is the component-wise
order from Definition \ref{def:prec}.
In particular, the lexicographically minimal and maximal 
nonzero Pl\"ucker coordinates of $A$ are 
$\Delta_I$ and $\Delta_{I'}$.
Note that if we write $I = \{i_1,\dots,i_k\}$, then
$I' = vw^{-1}\{i_1,\dots,i_k\}$.
\end{lemma}

\begin{proof}
Recall that $\mathcal P_{\v,\w} = \pi_k(\mathcal R_{\v,\w})$, where
$\mathcal R_{\v,\w} \subset \mathcal R_{v,w}$, and 
$\mathcal R_{v,w}=B^+\dot w\cdot B^+\cap B^-\dot v\cdot B^+$.
Now it is easy to check (and well-known) that the lexicographically minimal nonzero minor of each 
element in the Schubert cell $\pi_k(B^+\dot w\cdot B^+)$ is $\Delta_I$
and the lexicographically maximal minor of each element in the opposite
Schubert cell $\pi_k(B^- \dot v\cdot B^+)$ is $\Delta_{I'}$ where 
$I$ and $I'$ are as above.
\end{proof}

Our next goal is to provide formulas for the lexicographically minimal 
and maximal nonzero Pl\"ucker coordinates of the projected Deodhar components.

\begin{theorem}\label{p:maxmin}
Let $\w=s_{i_1} \dots s_{i_m}$ be a reduced expression for $w \in W^k$ and $\v \prec \w$.
Let $I = w\{n-k+1,\dots,n\}$ and 
$I' = v\{n-k+1,\dots,n\}$.
Let $A = \pi_k(g)$ for any $g \in G_{\v,\w}$.
If we write $g = g_1 \dots g_m$ as in Definition \ref{d:factorization}, then
\begin{equation}
\Delta_I(A) = (-1)^{|J_{\v}^{\bullet}|} \prod_{i\in J_{\v}^{\Box}} p_i\qquad  \text{ and }\qquad
\Delta_{I'}(A) = 1.
\end{equation}
Note that $\Delta_I(A)$ equals the product of all the labels
from the labeled Go-diagram associated to $(\v,\w)$.
\end{theorem}

Before proving Theorem \ref{p:maxmin}, we record the following lemma,
which can be easily verified.

\begin{lemma}\label{lem:act}
For $1 \leq i \leq n-1$, we have
\begin{enumerate}
\item $\dot s_i e_i = -e_{i+1}$, 
 $\dot s_i e_{i+1} = e_i$, and 
$\dot s_i e_j = e_j \text{ if } j\neq i \text{ or } i+1$.
\item $y_i(a) e_{i+1} = e_{i+1} + ae_i$ and
 $y_i(a) e_j = e_j$ if $j\neq i+1$.
\item $(x_i(a) \dot s_i^{-1}) e_i = e_{i+1}$,   
 $(x_i(a) \dot s_i^{-1}) e_{i+1} = -(e_i+ae_{i+1})$, and 
 $(x_i(a) \dot s_i^{-1}) e_j = e_j$ for $j\neq i$ or $i+1$.
\end{enumerate}
\end{lemma}

We now turn to the proof of Theorem \ref{p:maxmin}.

\begin{proof}
Recall from \eqref{Plucker} how to identify each $A \in Gr_{k,n}$ with its Pl\"ucker embedding.
We first verify that $\Delta_{I'}(A)=1$. Since 
$G_{\v,\w} \subset U^{-} \dot v$ (see Proposition \ref{p:parameterization}),
we can write $g \in G_{\v,\w}$ as $g = h \dot v$ with $h \in U^-$.
Let $\lambda = e_n \wedge e_{n-1} \wedge \dots \wedge e_{n-k+1}$.
Then $\Delta_{I'}(A) = \langle\dot v\cdot \lambda, g\cdot \lambda\rangle=\langle \dot v\cdot\lambda,
h\dot v \cdot \lambda\rangle=1$. 

Now we compute the value of $\Delta_I(A)$.
Recall from Proposition \ref{prop:independence}
that for $w\in W^k$, the Deodhar component 
$\mathcal R_{\v,\w}$ does not depend on the choice of reduced
expression $\w$ for $w$.  Therefore we will fix a linear extension
of $Q^k$, and use that to construct our reduced expressions
for each $w\in W^k$.

It follows that each reduced expression $\w$ for $w\in W^k$ 
where $W= \Sym_n$ has the form
\begin{equation}\label{red-exp}
(s_{j_a} s_{j_a +1}\dots  s_{n-k+a-1})  (s_{j_{a-1}} s_{j_{a-1}+1} \dots s_{n-k+a-2}) 
\dots (s_{j_2} s_{j_2+1} \dots s_{n-k+1})   (s_{j_1} s_{j_1+1} \dots s_{n-k}).
\end{equation}
The four factors above correspond to the products of generators corresponding to the 
last, next-to-last, second, and top rows of the Young diagram, respectively.  
In particular,
$1 \leq a \leq k$ ($a$ is the number of rows in the Young diagram
corresponding to $w$), and $j_1 < j_2 < \dots < j_{a-1} < j_a$.
Moreover, it is easy to check that 
$\{j_1, j_2, \dots, j_a,  n-k+a+1, n-k+a+2,\dots,n-1,n\}$ are the positions
of the pivots  of $A$ (they correspond to the shape of the Young diagram), so 
$I = \{j_1, j_2, \dots, j_a,  n-k+a+1, n-k+a+2,\dots,n-1,n\}$.

Each $g\in G_{\v,\w}$ will be obtained from 
\eqref{red-exp} by replacing the $s_i$'s by 
$\dot s_i$'s, $y_i(a)$'s, or $x_i(m) \dot s_i^{-1}$'s.
Let us write $g = g^{(1)} g^{(2)} \dots g^{(a)}$ 
where $g^{(1)}$ is the product of $g_i$'s corresponding to 
$(s_{j_a} s_{j_a +1}\dots  s_{n-k+a-1})$,
$g^{(2)}$ is the product of $g_i$'s corresponding to 
$(s_{j_{a-1}} s_{j_{a-1}+1} \dots s_{n-k+a-2})$, etc. 
Now consider how such a $g$ acts on $e_n, e_{n-1}, \dots$.
Looking at Lemma \ref{lem:act}, we see that $g^{(1)}$ is the only portion of $g$
which can affect $e_{n-k+a}$ (or any $e_j$ with $j>n-k+a$).  
This is because every $s_i$ appearing in the other factors of \eqref{red-exp}
has the property that $i \leq n-k+a-2$, and in this case,
$\dot s_{i}$, $y_{i}(a)$, and 
$x_{i}(m) \dot s_{i}^{-1}$ all act as the identity on $e_{n-k+a}$ (or any $e_j$ with $j>n-k+a$).
Similarly $g^{(1)} g^{(2)}$ is the only portion of $g$ which can affect $e_{n-k+a-1}$,
and $g^{(1)} g^{(2)} g^{(3)}$ is the only portion of $g$ which can affect $e_{n-k+a-2}$,
etc.

Now we want to determine the value of the lexicographically minimal Pl\"ucker
coordinate $\Delta_I(A).$  So we need to determine the coefficient of $E_I$ in 
$g \cdot e_n \wedge  \dots \wedge e_{n-k+1}$.  From Lemma \ref{lem:act}, 
we see that 
$\dot s_i e_{i+1} = e_i$, 
$y_i(a) e_{i+1} = a e_i + \text{ a higher term}$, and 
$x_i(a) \dot s_i^{-1} e_{i+1} = -e_i + \text{ a higher term}$.
Therefore from \eqref{red-exp}, we see that the expansion of 
$g \cdot e_{n-k+a}$  in the basis $e_1,\dots,e_n$ has a 
nonzero coefficient in front of $e_{j_a}$.  And that coefficient
is $(-1)^q$ times the product of all the parameters $p$ occurring in 
$g^{(1)}$, where $q$ is the number of $x$-factors in $g^{(1)}$.

Similarly, from \eqref{red-exp}, the expansion of 
$g \cdot e_{n-k+a-1}$ in the basis $e_1,\dots,e_n$ has a 
nonzero coefficient in front of $e_{j_{a-1}}$, and that coefficient is 
$(-1)^q$ times the product of all the parameters $p$ occurring in 
$g^{(2)}$, where $q$ is the number of $x$-factors in $g^{(2)}$.

Continuing in this fashion, 
the expansion of $g \cdot e_{n-k+1}$ in the basis $e_1,\dots,e_n$ has a 
nonzero coefficient in front of $e_{j_{1}}$, and that coefficient is 
$(-1)^q$ times the product of all the parameters $p$ occurring in 
$g^{(a)}$, where $q$ is the number of $x$-factors in $g^{(a)}$.

Additionally, $g$ acts as the identity on $e_{n-k+a+1}$, \dots,
$e_{n-1}$, and $e_n$.  It follows that the coefficient of 
$E_I$ in the expansion of $g \cdot e_n \wedge \dots \wedge e_{n-k+1}$ in the standard basis is 
$(-1)^{|J_{\v}^{\bullet}|} \prod_{i\in J_{\v}^{\Box}} p_i$, as desired. 
\end{proof}

Our next goal is to give a formula for some other Pl\"ucker coordinates besides the lexicographically 
minimal and maximal ones.  First it will be helpful to define some notation.

\begin{definition}\label{inout}
Let $W = \Sym_n$, 
let $\w=s_{i_1} \dots s_{i_m}$ be a reduced expression for $w \in W^k$ and 
choose $\v \prec \w$.   This determines a Go-diagram $D$
in a Young diagram $Y$.
Let $b$ be any box of $D$.  Note that the set of all boxes of $D$ which are weakly southeast
of $b$ forms a Young diagram $Y_b^{\In}$; also the complement of $Y_b^{\In}$ in $Y$ is a Young diagram
which we call $Y_b^{\Out}$ (see Example \ref{ex:out} below).  By looking at the restriction of $\w$ to the positions corresponding
to boxes of $Y_b^{\In}$, we obtained a reduced expression $\w_b^{\In}$ for some permutation 
$w_b^{\In}$, together with a distinguished subexpression $\v_b^{\In}$ for some permutation $v_b^{\In}$.
Similarly, by using the positions corresponding to boxes of $Y_b^{\Out}$, we obtained
$\w_b^{\Out}$, $w_b^{\Out}$, $\v_b^{\Out}$, and $v_b^{\Out}$.  When the box $b$ is understood,
we will often omit the subscript $b$.

For any box $b$, note that it is always possible to choose a linear extension of $O_w$ which orders 
all the boxes of $Y^{\Out}$ after those of $Y^{\In}$.  We can then adjust $\w$ accordingly;
Proposition \ref{p:whether} implies that this does not affect whether the corresponding
expression $\v$ is distinguished.
Having chosen such a linear extension, we can then write
$\w = \w^{\In} \w^{\Out}$ and $\v = \v^{\In} \v^{\Out}$.  We then use
$g^{\In}$ and $g^{\Out}$ to denote the corresponding factors of 
$g \in G_{\v,\w}$.
We define $J^{\Box}_{\v^{\Out}}$ to be the subset of $J^{\Box}_{\v}$ 
coming from the factors of $\v$ contained in $\v^{\Out}$.  Similarly,
for $J^{\circ}_{\v^{\Out}}$  
and $J^{\bullet}_{\v^{\Out}}$. 
\end{definition}

\begin{example}\label{ex:out}
Let $W =\Sym_7$ and $\w = s_4 s_5 s_2 s_3 s_4 s_6 s_5 s_1 s_2 s_3 s_4$
be a reduced expression for $w \in W^3$.  
Let $\v = s_4 s_5 1 1 s_4 1 s_5 s_1 1 1 s_4$ be a distinguished subexpression.
So $w = (3,5,6,7,1,2,4)$ and $v=(2,1,3,4,6,5,7)$.
We can represent this data by the poset $O_w$ and the corresponding
 Go-diagram:\\
\setlength{\unitlength}{0.7mm}
\begin{center}
  \begin{picture}(60,30)
  
  \put(5,35){\line(1,0){40}}
  \put(5,25){\line(1,0){40}}
  \put(5,15){\line(1,0){40}}
  \put(5,5){\line(1,0){30}}
  \put(5,5){\line(0,1){30}}
  \put(15,5){\line(0,1){30}}
  \put(25,5){\line(0,1){30}}
  \put(35,5){\line(0,1){30}}
  \put(45,15){\line(0,1){20}}

  \put(7,28){$s_4$}
  \put(17,28){$s_3$}
  \put(27,28){$s_2$}
  \put(37,28){$s_1$}
  \put(7,18){$s_5$}
  \put(17,18){$s_4$}
  \put(27,18){$s_3$}
  \put(37,18){$s_2$}
  \put(7,8){$s_6$}
  \put(17,8){$s_5$}
  \put(27,8){$s_4$}
 
  \end{picture} 
  \qquad
  \begin{picture}(50,30)
  
  \put(5,35){\line(1,0){40}}
  \put(5,25){\line(1,0){40}}
  \put(5,15){\line(1,0){40}}
  \put(5,5){\line(1,0){30}}
  \put(5,5){\line(0,1){30}}
  \put(15,5){\line(0,1){30}}
  \put(25,5){\line(0,1){30}}
  \put(35,5){\line(0,1){30}}
  \put(45,15){\line(0,1){20}}

  \put(10,30){\circle*{5}}
  \put(40,30){\circle{5}}
   \put(10,20){\circle*{5}}
   \put(20,20){\circle{5}}
 
  \put(30,10){\circle{5}}
  \put(20,10){\circle{5}}
 
  \end{picture} 
\end{center}

Let $b$ be the box of the Young diagram which is in the second row
and the second column (counting from left to right).
Then the diagram below shows: the 
boxes of
$Y^{\In}$ and 
$Y^{\Out}$; 
a linear extension which 
puts the boxes of $Y^{\Out}$ after those of $Y^{\In}$; and 
the corresponding labeled Go-diagram.
Using this linear extension,
$\w^{\In} = s_4 s_5 s_2 s_3 s_4$, $\w^{\Out} = s_6 s_5 s_1 s_2 s_3 s_4$,
$\v^{\In} = s_4 s_5 1 1 s_4$, and $\v^{\Out} = 1 s_5 s_1 1 1 s_4$.
\setlength{\unitlength}{0.7mm}
\begin{center}
  \begin{picture}(60,40)
  
  \put(5,35){\line(1,0){40}}
  \put(5,25){\line(1,0){40}}
  \put(5,15){\line(1,0){40}}
  \put(5,5){\line(1,0){30}}
  \put(5,5){\line(0,1){30}}
  \put(15,5){\line(0,1){30}}
  \put(25,5){\line(0,1){30}}
  \put(35,5){\line(0,1){30}}
  \put(45,15){\line(0,1){20}}

  \put(6,28){out}
  \put(16,28){out}
  \put(26,28){out}
  \put(36,28){out}
  \put(6,18){out}
  \put(18,18){in}
  \put(28,18){in}
  \put(38,18){in}
  \put(6,8){out}
  \put(18,8){in}
  \put(28,8){in}
  \end{picture} 
  \begin{picture}(60,40)
  \put(5,35){\line(1,0){40}}
  \put(5,25){\line(1,0){40}}
  \put(5,15){\line(1,0){40}}
  \put(5,5){\line(1,0){30}}
  \put(5,5){\line(0,1){30}}
  \put(15,5){\line(0,1){30}}
  \put(25,5){\line(0,1){30}}
  \put(35,5){\line(0,1){30}}
  \put(45,15){\line(0,1){20}}

  \put(7,28){$11$}
  \put(17,28){$10$}
  \put(27,28){$~9$}
  \put(37,28){$~8$}
  \put(7,18){$~7$}
  \put(17,18){$~5$}
  \put(27,18){$~4$}
  \put(37,18){$~3$}
  \put(7,8){$~6$}
  \put(17,8){$~2$}
  \put(27,8){$~1$}
 
  \end{picture} 
   \begin{picture}(45,40)
  \put(5,35){\line(1,0){40}}
  \put(5,25){\line(1,0){40}}
  \put(5,15){\line(1,0){40}}
  \put(5,5){\line(1,0){30}}
  \put(5,5){\line(0,1){30}}
  \put(15,5){\line(0,1){30}}
  \put(25,5){\line(0,1){30}}
  \put(35,5){\line(0,1){30}}
  \put(45,15){\line(0,1){20}}
  \put(7,28){$-1$}
  \put(17,28){$p_{10}$}
  \put(27,28){$~p_{9}$}
  \put(37,28){$~1$}
  \put(7,18){$-1$}
  \put(17,18){$~1$}
  \put(27,18){$~p_4$}
  \put(37,18){$~p_3$}
  \put(7,8){$~p_6$}
  \put(17,8){$~1$}
  \put(27,8){$~1$}
  \end{picture} 
\end{center}
Note that $J^{\bullet}_{\v^{\Out}}=\{7, 11\}$ and 
$J^{\Box}_{\v^{\Out}}=\{6,9,10\}$. 
Then  $g \in G_{\v,\w}$ has the form
$$g = g^{\In} g^{\Out} = (\dot s_4 \dot s_5 y_2(p_3) y_3(p_4) \dot s_4 )\ (y_6(p_6) x_5(m_7) \dot s_5^{-1} \dot s_1
y_2(p_9) y_3(p_{10}) x_4(m_{11}) s_4^{-1}).$$
When we project the resulting $7 \times 7$ matrix to its first three columns, we get
the matrix 
\[
A=
\begin{pmatrix}
-p_9 p_{10} & -p_3 p_{10}& -p_{10} & -m_{11} & 0 & -1 & 0\\
0 & -p_3 p_4 & -p_4 & -m_7 & 1 & 0 & 0\\
0 & 0 & 0 & p_6 & 0 & 0 & 1\\
\end{pmatrix}
\]
\end{example}





%

\begin{theorem}\label{p:Plucker}
Let $\w=s_{i_1} \dots s_{i_m}$ be a reduced expression for $w \in W^k$ and $\v \prec \w$,
and let $D$ be the corresponding Go-diagram.  Choose any box $b$ of $D$,
and let $v^{\In} = v_b^{\In}$ and $w^{\In} = w_b^{\In}$, and   
 $v^{\Out} = v_b^{\Out}$ and $w^{\Out} = w_b^{\Out}$.
Let $A = \pi_k(g)$ for any $g \in G_{\v,\w}$, and let
$I = w\{n,n-1,\dots,n-k+1\}$.
If $b$ is a blank box, define $I_b = v^{\In} (w^{\In})^{-1} I \in {[n] \choose k}$.  
If $b$ contains a white or black stone, define $I_b = v^{\In} s_{b} (w^{\In})^{-1} I \in {[n] \choose k}$.  
If we write $g = g_1 \dots g_m$ as in Definition \ref{d:factorization}, then
\begin{enumerate}
\item If $b$ is a blank box, then $\Delta_{I_b}(A) = (-1)^{|J_{\v^{\Out}}^{\bullet}|} \prod_{i\in J_{\v^{\Out}}^\Box} p_i.$ 
\item If $b$ contains a white stone, then $\Delta_{I_b}(A) = 0.$
\item If $b$ contains a black stone, then $\Delta_{I_b}(A) = (-1)^{|J_{\v^{\Out}}^{\bullet}|+1} m_b \prod_{i\in J_{\v^{\Out}}^\Box} p_i  
+ \Delta_{I_b}(A_b),$ where $m_b$ is the parameter corresponding to $b$,
and $A_b$ is the matrix $A$ with $m_b = 0$.
\end{enumerate}
\end{theorem}

\begin{remark}
The Pl\"ucker coordinates given by Theorem \ref{p:Plucker} (1) 
are monomials 
in the $p_i$'s.  In particular, they are nonzero, and do not depend on the values of
the $m$-parameters from the $x_i(m)$-factors.

Those minors $\Delta_{I_b}(A)$ correspond to the chamber minors defined in
\cite[Definition 6.3]{MR}.   See also Lemmas 7.4 and 7.5 in \cite{MR}, and note that
the dominant weight for the present case is
$\lambda=e_{n-k+1}\wedge\cdots\wedge e_n$.
\end{remark}

Before proving  Theorem \ref{p:Plucker}, we mention an immediate
corollary.
\begin{corollary}\label{signatblack}
Use the notation of Theorem \ref{p:Plucker}.  Let $b$ be a box
of the Go-diagram, and let $e$, $s$, and $se$ denote the neighboring boxes
which are at the east, south, and southeast of $b$.
Then we have
\[
\frac{\Delta_{I_e}(A)\Delta_{I_s}(A)}{\Delta_{I_b}(A)\Delta_{I_{se}}(A)}=\left\{
\begin{array}{lll}
1\quad& {\rm if}\quad {\rm box} ~b~{\rm contains~a~white~stone}\\
-1\quad&{\rm if}\quad {\rm box} ~b~{\rm contains~a~black~stone}\\
p_b\quad &{\rm if}\quad  {\rm box} ~b~{\rm is~blank~and~the~labeled~Go~diagram~contains~}p_b\\
\end{array}\right.
\]
\end{corollary}

\begin{remark} \label{rem:2-term} 
Each black and white stone corresponds to a \emph{two-term}
Pl\"ucker
relation, that is, a three-term Pl\"ucker relation in which 
one term vanishes.
And each black stone
implies that there are two
Pl\"ucker coordinates with opposite signs.
This will be useful when we discuss the regularity of solitons in Section \ref{sec:regularity}.  Also note that the formulas in Corollary \ref{signatblack}
correspond to the \emph{Generalized Chamber Ansatz} in \cite[Theorem 7.1]{MR}.
\end{remark}

\begin{example}\label{ex:out2}
We continue Example \ref{ex:out}.  
By Theorem \ref{p:maxmin}, $I = w \{5,6,7\} = \{1,2,4\}$ and 
$I' = v \{5,6,7\} = \{5,6,7\}$, and
the lexicographically minimal and maximal
nonzero Pl\"ucker coordinates for $A$ are
$\Delta_I(A) = p_3 p_4 p_6 p_9 p_{10}$ and 
$\Delta_{I'}(A) = 1$; this can be verified
for the matrix $A$ above.

We now verify Theorem \ref{p:Plucker}
for the box $b$ chosen earlier.  Then 
$I_b = v^{\In} (w^{\In})^{-1}I = \{1,4,6\}$.
Theorem \ref{p:Plucker} says that 
$\Delta_{I_b}(A) = 0$, since this box contains a white stone.
The analogous computations for the boxes labeled 
$7$, $6$, $4$, $3$, $2$, $1$, respectively, 
yield $\Delta_{1,5,7} = -p_9 p_{10}$,
$\Delta_{1,2,7} = p_3 p_4  p_9 p_{10}$,
$\Delta_{1,4,5} = p_6 p_9 p_{10}$,
$\Delta_{1,3,4} = p_4 p_6 p_9 p_{10}$,
$\Delta_{1,2,4} =p_3 p_4 p_6 p_9 p_{10}$, and
$\Delta_{1,2,4} =p_3 p_4 p_6 p_9 p_{10}$.
These can be checked for the matrix $A$ above.
\end{example}

\subsection{The proof of Theorem \ref{p:Plucker}.}

For simplicity of notation, we assume that when we write $A$ in 
row-echelon form, its first pivot is $i_1 = 1$
and its last non-pivot is $n$. (The same proof works without this
assumption, but the notation required would be more cumbersome.)

Choose the box $b$ which is located at the northwest
corner of the Young diagram obtained by  removing the topmost row
and the leftmost column; this is the box labeled $5$ in the diagram
from Example \ref{ex:out}.  We will explain the proof of the theorem
for this box $b$.  The same argument works if $b$ lies in the 
top row or leftmost column; and such an argument can be iterated
to prove Theorem \ref{p:Plucker} for boxes which are 
(weakly) southeast of $b$.

Choose a linear extension of 
$O_w$ which orders all the boxes of $Y^{\Out}$ after those of $Y^{\In}$,
and which orders the  boxes of the top row so that they 
come after those of the leftmost column.  The linear extension
from Example \ref{ex:out} is one such an example.
Choosing the reduced expression $\w$ correspondingly, we  write
$\w = \w^{\In} \w^{\Out}$ and $\v = \v^{\In} \v^{\Out}$, then choose
$g \in G_{\v,\w}$ and write it as $g=g^{\In} g^{\Out}$.
Note that from our choice of linear extension, we have
\begin{equation}\label{w-Out}
\w^{\Out} = (s_{n-1} s_{n-2} \dots s_{n-k+1})(s_1 s_2 \dots s_{n-k}).
\end{equation}

Recall that $I_b = v^{\In} (w^{\In})^{-1} I$ if $b$ is a blank box
and otherwise
$I_b = v^{\In} s_{b} (w^{\In})^{-1} I$,
where 
$I = \{i_1,\dots,i_k \}$, with $i_1 = 1$.  In our case,
$s_b = s_{n-k}$.
Also $w^{-1} I = \{n-k+1,\dots,n-1,n\}$, which implies that 
\begin{equation}
(w^{\In})^{-1}I = w^{\Out} \{n-k+1,\dots,n-1,n\} = 
\{1,n-k+1,n-k+2,\dots,n-1\}.
\end{equation}

Since there is no factor of $s_1$ or $s_{n-1}$ in $\v^{\In}$
(respectively $\v^{\In} s_{n-k}$),
and $I_b = v^{\In} \{1,n-k+1,n-k+2,\dots,n-1\}$
(respectively 
$I_b = v^{\In} s_{n-k} \{1,n-k+1,n-k+2,\dots,n-1\}$),
we have
\begin{equation}
1 \in I_b \qquad \text{ and }\qquad  n \notin I_b.
\end{equation}
Write $I_b = \{j_1,\dots,j_k\}$ with $j_1 = 1$.
Our goal is to compute 
$\Delta_{I_b}(A) = \langle e_{j_1} \wedge \dots
\wedge e_{j_k}, ~g \cdot e_{n-k+1} \wedge \dots \wedge e_{n} \rangle.$

Let $f_{\ell} = g \cdot e_{n-k+\ell}$.
Let $q_{\ell}$ be the product of all labels in the 
``out" boxes of the $\ell$th row of the labeled Go-diagram.
Using Lemma \ref{lem:act} and 
equation \eqref{w-Out}, we obtain
\begin{align*}
f_k = g \cdot  e_n &= g^{\In} \cdot (q_k e_{n-1} + c_n^k e_n) \\
f_{k-1}  = g \cdot e_{n-1} &= g^{\In} \cdot (q_{k-1} e_{n-2}
  +c_{n-1}^{k-1} e_{n-1} + c_n^{k-1} e_n)\\
    & \vdots \hspace{3cm} \vdots \\
f_2 = g \cdot e_{n-k+2} &= g^{\In} \cdot (q_2 e_{n-k+1} + 
c^2_{n-k+2} e_{n-k+2} + \dots + c_n^2 e_n)\\
f_1 = g \cdot e_{n-k+1} &= g^{\In} \cdot (q_1 e_1 + c_2^1 e_2 + 
\dots + c_n^1 e_n).
\end{align*}
Here the $c_i^j$'s are constants depending on $g^{\Out}$.

We now claim that only the first term with coefficient
$q_{\ell}$ in each $f_{\ell}$ contributes to the Pl\"ucker
coordinate $\Delta_{I_b}(A)$.  To prove this claim, 
note that:
\begin{enumerate}
\item Since $n \notin I_b$ and $g^{\In} \cdot e_n = e_n$,
the terms $c_n^{\ell} e_n$ do not affect $\Delta_{I_b}(A)$.
Therefore, we may as well assume that each $c_n^{\ell} = 0$.
Define $\tilde{f}_k = q_k g^{\In} \cdot e_{n-1}$.
\item Now note that the term $c_{n-1}^{k-1} e_{n-1}$ does
not affect the wedge product 
$\tilde{f}_k \wedge f_{k-1}$.  In particular, 
$\tilde{f}_k \wedge f_{k-1} = \tilde{f}_k \wedge \tilde{f}_{k-1}$
where $\tilde{f}_{k-1} = q_{k-1} g^{\In} \cdot e_{n-2}$.
\item Applying the same argument for $2 \leq \ell \leq k-2$, we
can replace
each $f_{\ell}$ by $\tilde{f}_{\ell} = q_{\ell} g^{\In} \cdot e_{n-k+\ell}$,
without affecting the wedge product.
\item Since $1 \in I_b$ and $e_1$ does not appear in any $f_{\ell}$
except $f_1$, for the purpose of computing 
$\Delta_{I_b}(A)$ we may replace $f_1$ by $\tilde{f}_1 = q_1 e_1$.
\end{enumerate}

Now we have
\begin{align}
\Delta_{I_b} (A) &= \langle e_{j_1} \wedge \dots \wedge e_{j_k},\nonumber
 ~ f_1 \wedge \dots \wedge f_k \rangle \\
&= \langle e_{j_1} \wedge \dots \wedge e_{j_k}, \nonumber
~  \tilde{f}_1 \wedge \dots \wedge \tilde{f}_k \rangle \\
&= \bigl(\prod_{j=1}^k q_j \bigr) \label{eq:useful}
\langle e_{j_1} \wedge \dots \wedge e_{j_k}, 
~g^{\In} \cdot (e_{1} \wedge e_{n-k+1} \wedge\dots \wedge e_{n-1}) \rangle \\
&= \bigl(\prod_{j=1}^k q_j \bigr) 
\langle e_{j_2} \wedge \dots \wedge e_{j_k}, 
~g^{\In} \cdot (e_{n-k+1} \wedge \dots \wedge e_{n-1}) \rangle, \label{eq:minor}
\end{align}
where in the last step we used $j_1=1$.
Finally we need to compute the wedge product in \eqref{eq:minor}.

Consider the case that $b$ is a blank box.
Then from the definition of $I_b = \{j_1,\dots,j_k\}$,
we have $\{j_2,\dots,j_k\} = v^{\In} \{n-k+1,n-k+2,\dots,n-1\}$.
It follows that 
\[
\langle e_{j_2} \wedge \dots \wedge e_{j_k}, ~
g^{\In} \cdot (e_{n-k+1} \wedge \dots \wedge e_{n-1}) \rangle =1,
\]
because this is the lexicographically maximal minor for the 
matrix $A' = \pi_{k-1}(g^{\In}) \in Gr_{k-1,n-2}$ corresponding
to the sub Go-diagram obtained by removing the top row and 
leftmost column.
Therefore $\Delta_{I_b}(A) = \prod_{j=1}^k q_j = 
(-1)^{|J_{\v^{\Out}}^{\bullet}|} \prod_{i\in J_{\v^{\Out}}^\Box} p_i$,
as desired. 

Now consider the case that $b$ contains a white or black stone.
Then from the definition of $I_b = \{j_1,\dots,j_k\}$,
we have $\{j_2,\dots,j_k\} = v^{\In} s_{n-k} \{n-k+1,n-k+2,\dots,n-1\}.$
The wedge product in \eqref{eq:minor} is equal to 
$\langle v^{\In} s_{n-k} \cdot (e_{n-k+1} \wedge \dots \wedge e_{n-1}), 
g^{\In} \cdot (e_{n-k+1} \wedge \dots \wedge e_{n-1}) \rangle.$

If $b$ contains a white stone, then the last factor in $v^{\In}$
is $s_{n-k}$ and the last factor in $g^{\In}$ is $\dot s_{n-k}$,
so we can write 
$v^{\In} = \tilde{v}^{\In} s_{n-k}$ and 
$g^{\In} = \tilde{g}^{\In} \dot s_{n-k}$, where
$\tilde{\v}^{\In}$ is also a distinguished expression.
Then $\tilde{g}^{\In} \in G_{\tilde{\v}^{\In}, \w^{\In}}$
so $\tilde{g}^{\In} = h\tilde{v}^{\In}$ where $h\in U^{-}$.
Then we have  
$\langle v^{\In} s_{n-k} \cdot (e_{n-k+1} \wedge \dots \wedge e_{n-1}), 
g^{\In} \cdot (e_{n-k+1} \wedge \dots \wedge e_{n-1}) \rangle =
\langle \tilde{v}^{\In}  \cdot (e_{n-k+1} \wedge \dots \wedge e_{n-1}), 
h \tilde{v}^{\In} \cdot (e_{n-k+1} \wedge \dots \wedge e_{n-1}) \rangle$. 
Since $b$ contains a white stone,
$\tilde{v}^{\In} s_{n-k} > \tilde{v}^{\In}$ in the Bruhat order, and hence
$\tilde{v}^{\In} \{n-k\} < \tilde{v}^{\In} \{n-k+1\}$.  Since
$h\in U^-$, it follows that this wedge product equals $0$.

If $b$ contains a black stone then the last factor in $v^{\In}$
is $s_{n-k}$ and the last two factors in $g^{\In}$ are 
$x_{n-k}(m_b) \dot s_{n-k}^{-1}$.  So we can write 
$v^{\In} = \tilde{v}^{\In} s_{n-k}$ and 
$g^{\In} = \tilde{g}^{\In} x_{n-k}(m_b) \dot s_{n-k}^{-1}$.  
Then we have
\begin{align}
&g^{\In} \cdot (e_{n-k+1} \wedge \dots \wedge e_{n-1})\nonumber \\
 = &\label{eq:2nd}
\tilde{g}^{\In} x_{n-k}(m_b) \dot s_{n-k}^{-1} \cdot (e_{n-k+1} \wedge \dots
\wedge e_{n-1})\\
= &-\tilde{g}^{\In} \cdot (m_b  (e_{n-k+1} \wedge \dots \wedge e_{n-1}) +
     (e_{n-k} \wedge e_{n-k+2}\wedge \dots \wedge e_{n-1})) \label{eq:3rd}\\
  = &-m_b \tilde{g}^{\In} \cdot (e_{n-k+1} \wedge \dots \wedge e_{n-1})
    - \tilde{g}^{\In} \cdot (e_{n-k} \wedge e_{n-k+2}\wedge \dots \wedge e_{n-1}). \label{eq:2terms}
\end{align}
Note that to go from \eqref{eq:2nd} to \eqref{eq:3rd} we used
Lemma \ref{lem:act}.

Let us compute the wedge product of the first term in \eqref{eq:2terms} 
with 
$v^{\In} s_{n-k} \cdot (e_{n-k+1} \wedge \dots \wedge e_{n-1}).$
Using $v^{\In} = \tilde{v}^{\In} s_{n-k}$, this can be expressed as 
\begin{align*}
&-m_b \cdot \langle v^{\In} \cdot (e_{n-k} \wedge e_{n-k+2}\wedge \dots \wedge e_{n-1}), ~\tilde{g}^{\In} \cdot 
     (e_{n-k+1} \wedge \dots \wedge e_{n-1})\rangle \\
=&-m_b \cdot \langle \tilde{v}^{\In} \cdot 
             (e_{n-k+1} \wedge \dots \wedge e_{n-1}),
~\tilde{g}^{\In} \cdot (e_{n-k+1} \wedge \dots \wedge e_{n-1})\rangle.
\end{align*}
Since 
we again have $\tilde{g}^{\In} = h\tilde{v}^{\In}$ where $h\in U^{-}$, 
the above quantity equals $-m_b$.

Let us now compute the wedge product of the second term
in \eqref{eq:2terms} with 
$v^{\In} s_{n-k} \cdot (e_{n-k+1} \wedge \dots \wedge e_{n-1}).$
This wedge product can be written as 
\begin{align*}
&\langle v^{\In} \cdot (e_{n-k} \wedge e_{n-k+2}\wedge \dots \wedge e_{n-1}), ~\tilde{g}^{\In} \cdot (e_{n-k} \wedge e_{n-k+2}\wedge \dots 
\wedge e_{n-1}) \rangle \\
=&\langle v^{\In} \cdot (e_{n-k} \wedge e_{n-k+2}\wedge  \dots \wedge e_{n-1}), ~ \tilde{g}^{\In} \dot s_{n-k}^{-1} \cdot (e_{n-k+1} \wedge \dots 
\wedge e_{n-1}) \rangle \\
= &\Delta_{j_1,\dots,j_k} (A_b),
\end{align*}
where $A_b$ is the matrix obtained from $A$ by setting $m_b = 0$.
This completes the proof of the theorem.

\begin{corollary} For any box $b$, the rescaled Pl\"ucker coordinate
$$\frac{\Delta_{I_b}(A)}{\prod_{i\in J_{\v}^\Box} p_i}$$ depends only on 
the parameters $p_{b'}$ and $m_{b'}$ which correspond to boxes
$b'$ weakly southeast of $b$ in the Go-diagram.
\end{corollary}
\begin{proof} This follows immediately from 
\eqref{eq:useful} and the fact that 
$\prod_{j=1}^k q_j = 
(-1)^{|J_{\v^{\Out}}^{\bullet}|} \prod_{i\in J_{\v^{\Out}}^\Box} p_i$.
\end{proof}

\subsection{Positivity tests for projected Deodhar components in the Grassmannian}

We can use our results on Pl\"ucker coordinates 
to obtain \emph{positivity tests} for Deodhar components in the Grassmannian.

\begin{definition}\label{def:pos-test}
Let $D$ be a Go-diagram and $S_D \subset Gr_{k,n}$.
A collection $\J$ of $k$-element subsets of $\{1,2,\dots,n\}$
is called a \emph{positivity test} for $S_D$ if for any $A \in S_D$,
the condition that $\Delta_I(A) >0$ for all $I \in \J$ 
implies that $A \in (Gr_{k,n})_{\geq 0}$.
\end{definition}

\begin{theorem}\label{th:TPTest}
Consider $A\in Gr_{k,n}$ lying in some Deodhar component 
$S_D$, where $D$ is a Go-diagram.
Consider the collection of minors
$\J = \{\Delta_I(A)\} \cup \{\Delta_{I_b}(A) \ \vert \ b \text{ a box of }D\}$,
where $I$ and $I_b$ are defined as in Theorem \ref{p:Plucker}.
If all of these minors are positive, then 
$D$ has no black stones, and all of the parameters $p_i$ must be positive.
It follows that the Deodhar diagram corresponding to $D$ is a $\Le$-diagram,
and $A$ lies in the positroid cell $S_D^{tnn} \subset (Gr_{k,n})_{\geq 0}$.
In particular, $\J$ is a positivity test for $S_D$.
\end{theorem}

\begin{proof}
By Remark \ref{rem:2-term}, if all the minors in $\J$ are positive,
then $D$ cannot have  a black stone.

By Theorem \ref{p:maxmin} and Theorem \ref{p:Plucker} we have that 
$$\Delta_I(A) = (-1)^{|J_{\v}^{\bullet}|} \prod_{i\in J_{\v}^{\Box}} p_i \qquad 
\text{ and }\qquad \Delta_{I_b}(A) = (-1)^{|J_{\v^{\Out}}^{\bullet}|} \prod_{i\in J_{\v^{\Out}}^\Box} p_i.$$
Since we are assuming that both of these are positive, it follows that
for any box $b$, we have that 
$$\frac{\Delta_I(A)}{\Delta_{I_b}(A)}=
(-1)^{|J_{\v^{\In}}^{\bullet}|} \prod_{i\in J_{\v^{\In}}^\Box} p_i$$
is also positive.
Now by considering the boxes $b$ of $D$
in an order proceeding from southeast to northwest,
it is clear that every parameter $p_i$ in the labeled
Go-diagram must be positive,
because each 
$\frac{\Delta_I(A)}{\Delta_{I_{b}}(A)}$ must be positive.

Let $v$ and $w$ be the Weyl group elements corresponding to $D$.
Then it follows from Remark \ref{rem:TPcell} that 
$A$ lies in the projection of the totally positive cell 
$\mathcal{R}_{v,w}^{>0}$. And the projection of  
$\mathcal{R}_{v,w}^{>0}$ is precisely the positroid cell
$S_D^{tnn}$ of $(Gr_{k,n})_{\geq 0}$.
\end{proof}



\section{Soliton solutions to the KP equation}\label{soliton-background}

We now explain how to obtain a soliton solution 
to the KP equation from a point of $Gr_{k,n}$.
Each soliton solution can be considered as an orbit
with the flow parameters $(x,y,t)\in \mathbb{R}^3$
on $Gr_{k,n}$.

\subsection{From a point of the Grassmannian to a $\tau$-function.}

We start by fixing real parameters $\kappa_j$ such that 
\[
\kappa_1~<~\kappa_2~<~\cdots~<\kappa_n,
\]
which  are {\it generic}, in the sense
that 
the sums $\sum_{m=1} ^p\kappa_{j_m}$ are all distinct for any $p$ with
$1<p<n$.
We also 
assume that the differences between consecutive $\kappa_i$'s are similar,
that is, 
$\kappa_{i+1}-\kappa_i$ is of order one (e.g. one can take all $\kappa_j$ to be integers).

We now give a realization of $Gr_{k,n}$ with a specific basis of $\mathbb{R}^n$.  We define a set of vectors $\{\mathsf{E}_j^{\bf 0}: j=1,\ldots,n\}$ by
\[
\mathsf{E}_j^{\bf 0}:=\begin{pmatrix}
\kappa_j^{n-1}\\ \kappa_j^{n-2} \\ \vdots \\ \kappa_j\\ 1
\end{pmatrix} \,\in\,\mathbb{R}^n.
\]
Since all $\kappa_j$'s are distinct, the set $\{\mathsf{E}_j^{\bf 0}:j=1,\ldots,n\}$ forms a basis of $\mathbb{R}^n$.
Now define an $n\times n$ matrix $E^{\bf 0}=(\mathsf{E}_1^{\rm 0},\ldots,\mathsf{E}_n^{\bf 0})$, and let $A$ be a full-rank  $k\times n$ matrix  parametrizing a point on $Gr_{k,n}$. 
Then 
the vectors $\{\mathsf{F}_i^{\bf 0}\in\mathbb{R}^n :i=1,\ldots,k\}$
span a $k$-dimensional subspace in $\mathbb{R}^n$, 
where $\mathsf{F}_i^{\bf 0}$ is defined by
\[
\mathsf{F}_i^{\bf 0}:=\sum_{j=1}^na_{i,j}\,\mathsf{E}_j^{\bf 0},\qquad {\rm or}\qquad
(\mathsf{F}_1^{\bf 0},\ldots, \mathsf{F}_k^{\bf 0})={E}^{\bf 0}A^T.
\]
For $I=\{i_1,\ldots,i_k\}$, define the vector
$\mathsf{E}_{I}^{\bf 0}=\mathsf{E}_{i_1}^{\bf 0}\wedge\cdots\wedge \mathsf{E}_{i_k}^{\bf 0}
$.
Then we have a realization of the  Pl\"ucker embedding:
\[
\mathsf{F}_1^{\bf 0}\wedge\cdots\wedge \mathsf{F}_k^{\bf 0} =\sum_{I\in\binom{n}{k}}\Delta_I(A)\mathsf{E}_I^{\bf 0}.
\]

In \cite{Sato}, Sato showed that each solution of the KP equation is given by an orbit on the Grassmannian. To construct such an orbit, 
we consider a deformation $\mathsf{E}_j^{\bf t}$ of the vector $\mathsf{E}_j^{\bf 0}$, 
defined by:
\[
{\bf t}:= (x,y,t), \qquad
\theta_j(x,y,t)=\kappa_j x+\kappa_j^2 y+\kappa_j^3 t, \qquad 
\mathsf{E}_j^{\bf t}:=\mathsf{E}_j^{\bf 0}\exp\left(\theta_j(x,y,t)\right).
\]

\begin{remark}
Let $E^{\bf t}$ be the $n\times n$ matrix function whose columns are the
vectors $\{\mathsf{E}_j^{\bf t}\}$:
\[
E^{\bf t}:=(\mathsf{E}_1^{\bf t},\ldots,\mathsf{E}_n^{\bf t})=E^{\bf 0}{\rm diag}(e^{\theta_1},e^{\theta_{2}},\ldots,e^{\theta_n}).
\]
Note that $E^{\bf 0}$ is a Vandermonde matrix.
The vector functions $\{\mathsf{E}_j^{\bf t}\}$ form a fundamental set of solutions of a system of differential equations.
More concretely, if we define 
elementary symmetric polynomials in the $\kappa_j$'s by
\[
\sigma_1=\sum_{j=1}^n\kappa_j,\quad \sigma_2=\sum_{i<j}\kappa_i\kappa_j,\quad \sigma_3=\sum_{i<j<k}\kappa_i\kappa_j\kappa_k,\quad \cdots
\]
and let $C_K$ be the 
companion matrix 
\[
C_K=\begin{pmatrix}
\sigma_1  &  -\sigma_2 &  \cdots  &  \cdots  &  \pm\sigma_n \\
1                &    0               &   \cdots  &\cdots  &      0      \\
0                &     1             &     \ddots  &  \vdots  &   0     \\
\vdots      & \vdots          &  \ddots   &     0  & \vdots   \\
0               &     0               &  \cdots   &     1         &    0
\end{pmatrix},
\]
then 
the matrix $E^{\bf t}$ satisfies
\[
\mathcal{L}E^{\bf t}:=\left(\frac{\partial}{\partial x}-C_K\right) E^{\bf t}=0.
\]
So for any ${\bf t}=(x,y,t)$, we have
\[
\mathbb{R}^n\cong {\rm ker}(\mathcal{L})={\rm Span}_{\mathbb R}\{\mathsf{E}_j^{\bf t}:j=1,\ldots,n\}.
\]
Note that $C_K$ can be diagonalized by the Vandermonde matrix $E^{\bf 0}$, i.e.
\[
C_KE^{\bf 0}=E^{\bf 0}D, \qquad {\rm where}\quad D={\rm diag}(\kappa_n,\ldots, \kappa_1).
\]
Each vector function $E^{\bf t}$ satisfies the following  \emph{linear} equations with respect to $y$ and $t$:
\[
 \frac{\partial E^{\bf t}}{\partial y}=\frac{\partial^2E^{\bf t}}{\partial x^2}=C_K^2E^{\bf t}\qquad
 {\rm and}\qquad
\frac{\partial E^{\bf t}}{\partial t}=\frac{\partial^3 E^{\bf t}}{\partial x^3}=C_K^3E^{\bf t}.
\]
This is a key of the ``integrability'' of the KP equation, that is, the solutions of the linear equations
provide a solution of the KP equation.
\end{remark}

We now define an orbit generated by the 
matrix $E^{\bf t}$ on elements of $G = \GL_n$, 
\[
g^{\bf t}:=E^{\bf t}g \qquad {\rm for~each}\quad g\in \GL_n. 
\]
Then $\{g^{\bf t}\cdot e_{n-k+1}\wedge\cdots\wedge e_{n-1}\}$ is 
a flow (orbit) of the highest weight vector on the 
corresponding fundamental representation 
of ${\GL}_n$.

Next we define the $\tau$-function as 
\begin{align*}\label{tau1}
\tau(x,y,t):=&\langle e_1\cdots\wedge e_k,~ \mathsf{F}_1^{\bf t}\wedge\cdots\wedge\mathsf{F}^{\bf t}_k\rangle\\
=&\langle e_1\wedge\cdots\wedge e_k,\,
g^{\bf t}\cdot e_{n-k+1}\wedge \cdots\wedge e_{n}\rangle,
\end{align*}
where $\mathsf{F}^{\bf t}_j:=g^{\bf t}\cdot e_{n-k+j}$.
Given 
$I=\{i_1,\ldots,i_k\} \subset [n]$,  we let 
$E_I(x,y,t)$ denote the scalar function
\begin{equation}\label{Efunction}
\begin{array}{lll}
E_I(x,y,t)&=\langle e_1\wedge\cdots\wedge e_k,\,
\mathsf{E}^{\bf t}_{i_1}\wedge \cdots\wedge\mathsf{E}^{\bf t}_{i_k}\rangle=
\langle e_1\wedge\cdots\wedge e_k,\,
\mathsf{E}^{\bf 0}_{i_1}\wedge \cdots\wedge\mathsf{E}^{\bf 0}_{i_k}\rangle\,e^{\theta_{i_1} + \cdots + \theta_{i_k}}\\[1.5ex]
&=\displaystyle{\left(\prod_{l<m}(\kappa_{i_m}-\kappa_{i_l})\right)\,e^{\theta_{i_1} + \cdots + \theta_{i_k}}}.
\end{array}
\end{equation}
With the projection $\pi_k:\SL_n 
\to Gr_{k,n}, \ g\mapsto A$, the $\tau$-function
can be also written as
 \begin{equation}\label{tau}
\tau(x,y,t)=\tau_A(x,y,t) = \sum_{I\in\binom{[n]}{k}}\Delta_I(A)\,E_I(x,y,t).
\end{equation}
It follows that if  $A\in (Gr_{k,n})_{\ge 0}$, then 
$\tau_A>0$ for all $(x,y,t)\in\mathbb{R}^3$.  

\begin{remark}
The present definition of the $\tau$-function is quite useful for the study of
the Toda lattice whose solutions are defined on a complete flag manifold.
We will discuss the totally non-negative flag variety and the Toda lattice in
a forthcoming paper.
\end{remark}

\subsection{From the $\tau$-function to solutions of the KP equation}\label{sec:tau}
The KP equation for $u(x,y,t)$ 
\[
\frac{\partial}{\partial x}\left(-4\frac{\partial u}{\partial t}+6u\frac{\partial u}{\partial x}+\frac{\partial^3u}{\partial x^3}\right)+3\frac{\partial^2 u}{\partial y^2}=0
\]
was proposed by Kadomtsev and Petviashvili in 1970 \cite{KP70}, in order to 
study the stability of the soliton solutions of the Korteweg-de Vries (KdV) equation
under the influence of weak transverse perturbations.
The KP equation can be also used to describe  two-dimensional shallow 
water wave phenomena (see for example \cite{K10}).   This equation is now considered to be a prototype of
an integrable nonlinear partial differential equation.
For more background, see \cite{NMPZ84, D91, AC91, H04, MJD00}.

Note that the $\tau$-function defined in \eqref{tau} can be also
written in the Wronskian form 
\begin{equation}\label{tauA}
\tau_A(x,y,t)={\rm Wr}(f_1,f_2,\ldots, f_k),
\end{equation}
with the scalar functions $\{f_j:j=1,\ldots,k\}$ given by
\[
(f_1,f_2,\ldots, f_k)^T=A\cdot (\exp\theta_1,\exp\theta_2,\ldots,\exp\theta_n)^T,
\]
where $(\ldots)^T$ denotes the transpose of the (row) vector $(\ldots)$.

It is then well known
(see \cite{H04, CK1, CK2, CK3})   
that for each choice of constants
$\{\kappa_1,\ldots,\kappa_n\}$ and element
$A\in Gr_{k,n}$,
 the $\tau$-function defined in \eqref{tauA} 
provides a soliton solution of the KP equation,
\begin{equation}\label{KPsolution}
u_A(x,y,t)=2\frac{\partial^2}{\partial x^2}\ln\tau_A(x,y,t).
\end{equation}

If $A \in (Gr_{k,n})_{\ge 0}$, then it is obvious that
$u_A(x,y,t)$ is regular for all $(x,y,t)\in\mathbb{R}^3$.  A 
main result
of this paper is that the converse also holds -- see Theorem \ref{th:regularity}.
Throughout this paper when we speak of a {\it soliton solution to the KP equation},
we will mean a solution $u_A(x,y,t)$ which has the form
\eqref{KPsolution}, where the $\tau$-function is given by \eqref{tau}.

\begin{remark}
The function $E_I(x,y,t)$ in the $\tau$-function \eqref{tau} can be expressed as the Wronskian form in terms of 
$\{E_{i_j}=e^{\theta_{i_j}}:j=1,\ldots,k\}$, i.e.
\[
E_I(x,y,t)={\rm Wr}(E_{i_1},E_{i_2},\ldots,E_{i_k}).
\]
\end{remark}


\section{Contour plots of soliton solutions}
\label{sec:solgraph}


One can visualize a solution $u_A(x,y,t)$ to the KP equation
by drawing 
level sets of the solution in the $xy$-plane, when the coordinate $t$ is fixed.
For each $r\in \mathbb{R}$, we denote the corresponding level set by 
\[
C_r(t):=\{(x,y)\in\mathbb{R}^2: u_A(x,y,t)=r\}. 
\]
Figure \ref{fig:1soliton} depicts both a three-dimensional image of a solution $u_A(x,y,t)$,
as well as multiple level sets $C_r$. 
These level sets are lines parallel to the line of the wave peak.

\begin{figure}[h]
\begin{center}
\includegraphics[height=1.7in]{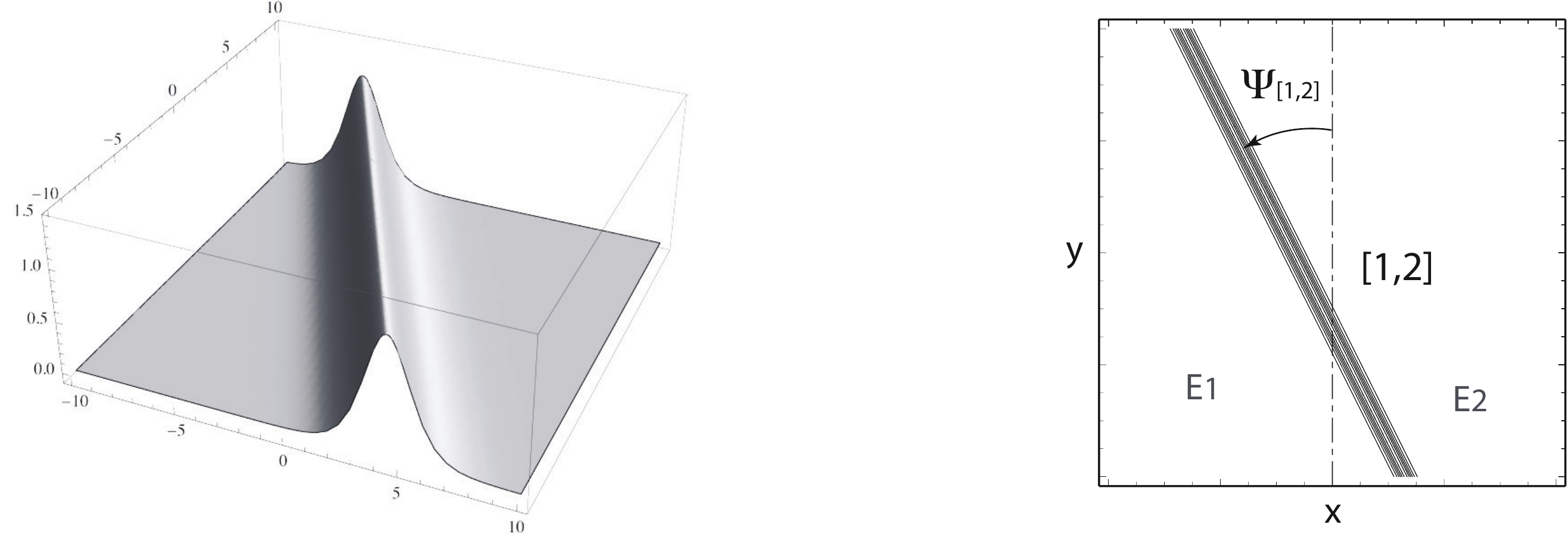}
\par
\end{center}
\caption{A line-soliton solution $u_A(x,y,t)$ where $A=(1,1) \in (Gr_{1,2})_{\geq 0}$, depicted via
the 3-dimensional profile
$u_A(x,y,t)$, and the level sets of $u_A(x,y,t)$ for some $t$.
$E_i$ represents the dominant exponential in each region.
\label{fig:1soliton}}
\end{figure}

To study the behavior of $u_A(x,y,t)$ for $A\in S_{\M} \subset Gr_{k,n}$,
we consider the dominant exponentials in the $\tau$-function \eqref{tau} at each point $(x,y,t)$.  First we write the $\tau$-function in the form
\begin{align*}
\tau_A(x,y,t)&=\sum_{J\in\binom{[n]}{k}}\Delta_J(A)E_J(x,y,t)\\
&=\sum_{J\in\mathcal{M}}
\exp\left(\sum_{i=1}^n(\kappa_{j_i}x+\kappa_{j_i}^2y+\kappa_{j_i}^3t)+\ln(\Delta_J(A)K_J)\right), 
\end{align*}
where $K_J:=\prod_{\ell < m} (\kappa_{j_m}-\kappa_{j_{\ell}})>0$.
Note that in general the terms $\ln(\Delta_J(A)K_J)$ could be imaginary
when some $\Delta_J(A)$ are negative.

Since we are interested in the behavior of the soliton solutions when the variables
$(x,y,t)$ are on a large scale, we rescale the variables with a small positive number $\epsilon$, 
\[
x~\longrightarrow ~\frac{x}{\epsilon},\qquad y~\longrightarrow~\frac{y}{\epsilon},\qquad
t~\longrightarrow~\frac{t}{\epsilon}.
\]
This leads to
\[
\tau_A^{\epsilon}(x,y,t)
=\sum_{J\in\mathcal{M}}
\exp\left(\frac{1}{\epsilon}\,\sum_{i=1}^n(\kappa_{j_i}x+\kappa_{j_i}^2y+\kappa_{j_i}^3t)+\ln(\Delta_J(A)K_J)\right).
\]
Then we define a function $f_A(x,y,t)$ as the limit
\begin{equation}\label{f-contour}
\begin{array}{lll}
{f}_A(x,y,t)&=\displaystyle{\lim_{\epsilon\to 0}\epsilon\ln\left(\tau^{\epsilon}_A(x,y,t)\right)}\\[1.5ex]
&=  \underset{J\in\mathcal{M}}\max \left\{
                     \sum_{i=1}^k (\kappa_{j_i} x + \kappa^2_{j_i} y +\kappa^3_{j_i} t)\right\}.
\end{array}
\end{equation}
Since the above function depends only on the collection $\mathcal{M}$,
we also denote it as $f_{\mathcal{M}}(x,y,t)$.

\begin{definition}\label{contour}
Given a solution $u_A(x,y,t)$ of the KP equation as in 
(\ref{KPsolution}), we define its \emph{contour plot} $\mathcal{C}(u_A)$ to be the 
locus in $\mathbb{R}^3$ where $f_A(x,y,t)$ is not linear.  If we fix $t=t_0$, 
then we let $\mathcal{C}_{t_0}(u_A)$ be the locus in $\mathbb{R}^2$ where $f_A(x,y,t=t_0)$ is not 
linear, and we also refer to this as a \emph{contour plot}.
Because these contour plots depend only on $\mathcal{M}$ and not on $A$, we also refer to them
as $\mathcal{C}(\mathcal{M})$ and $\mathcal{C}_{t_0}(\mathcal{M})$.
\end{definition}

\begin{remark}
The contour plot approximates the locus where 
$|u_A(x,y,t)|$ takes on its maximum values or is singular.
\end{remark}

\begin{remark}
Note that the contour plot generated by the function $f_A(x,y,t)$ at $t=0$ consists of a set of semi-infinite lines attached to the origin $(0,0)$ in the $xy$-plane.
And if $t_1$ and $t_2$ have the same sign,
then the corresponding contour plots
$\CC_{t_1}(\M)$ and $\CC_{t_2}(\M)$ are self-similar.

Also note that because  our definition of the contour plot ignores the constant 
terms $\ln(\Delta_J(A)K_J)$, there are no phase-shifts in our picture,
and the contour plot for $f_A(x,y,t) = f_{\M}(x,y,t)$ does
not depend on the signs of the Pl\"ucker coordinates.
\end{remark}

It follows from Definition \ref{contour} that $\CC(u_A)$ 
and $\CC_{t_0}(u_A)$
are piecewise linear subsets of $\R^3$ and $\R^2$, respectively, of 
codimension $1$.  In fact it is easy to verify the following.

\begin{proposition}\cite[Proposition 4.3]{KW2}
If each $\kappa_{i}$ is an integer, then 
$\CC(u_A)$ is a tropical hypersurface in $\R^3$,
and $\CC_{t_0}(u_A)$ is a tropical hypersurface (i.e. a tropical curve)
in $\R^2$.  
\end{proposition}

The contour plot $\CC_{t_0}(u_A)$ consists of line segments 
called \emph{line-solitons}, some of 
which have finite length, while others are unbounded
and extend in the $y$ direction to 
$\pm \infty$. 
Each region of the complement of 
$\CC_{t_0}(u_A)$ 
in $\R^2$ is a domain of linearity for $f_A(x,y,t)$, 
and hence each region is  
naturally associated to a {\it dominant exponential}  
$\Delta_J(A) E_J(x,y,t)$ from the $\tau$-function \eqref{tau}.
We label this region by $J$ or $E_J$.
Each line-soliton
represents a balance between two dominant exponentials
in the 
$\tau$-function.



Because of the genericity of the $\kappa$-parameters, the following
lemma is immediate.
\begin{lemma}\label{separating}\cite[Proposition 5]{CK3}
The index sets of the 
dominant exponentials of the $\tau$-function in adjacent regions
of the contour plot in the $xy$-plane are of the form $\{i,l_2,\dots,l_k\}$ and
$\{j, l_2,\dots,l_k\}$.
\end{lemma}

We call the line-soliton separating the two dominant exponentials in 
Lemma \ref{separating} a {\it line-soliton of type $[i,j]$}.  Its
equation is 
\begin{equation}\label{eq-soliton}
x+(\kappa_i+\kappa_j)y+(\kappa_i^2+\kappa_i\kappa_j+\kappa_j^2)t=0.
\end{equation}

\begin{remark}\label{slope}
Consider a line-soliton given by (\ref{eq-soliton}).
Compute the angle $\Psi_{[i,j]}$ between 
the positive $y$-axis and 
the line-soliton 
of type $[i,j]$, 
measured in the counterclockwise direction, so that the negative $x$-axis
has an angle of $\frac{\pi}{2}$ and the positive $x$-axis has an 
angle of $-\frac{\pi}{2}$. Then $\tan \Psi_{[i,j]} = \kappa_i+\kappa_j$,
so we refer to $\kappa_i+\kappa_j$ as the \emph{slope} of the 
$[i,j]$ line-soliton (see Figure \ref{fig:1soliton}).
\end{remark}

In Section \ref{sec:solitonplabic} we will explore 
the combinatorial structure of contour plots,
that is, the ways in which line-solitons may interact. 
Generically we expect a point at which several line-solitons
meet to have degree $3$; we regard such a point as a trivalent 
vertex.  Three line-solitons meeting at a trivalent vertex
exhibit a {\it resonant interaction} (this corresponds to the 
{\it balancing condition} for a tropical curve). 
See \cite[Section 4.2]{KW2}.  
One may also have two line-solitons which cross over
each other, forming an $X$-shape: we call this
an \emph{$X$-crossing}, but do not regard it as a vertex.
See Figure \ref{contour-soliton}. 
Vertices of degree greater than $4$
are also possible.  

\begin{definition}\label{def:blackwhiteX}
Let $i<j<k<\ell$ be positive integers.  
An $X$-crossing involving two line-solitons of types
$[i,k]$ and $[j,\ell]$ is called a 
\emph{black $X$-crossing}.  An $X$-crossing involving
two line-solitons of types $[i,j]$ and $[k,\ell]$, or 
of types $[i,\ell]$ and $[j,k]$, is called a 
\emph{white $X$-crossing}.
\end{definition}

\begin{definition}\label{def:generic}
A contour plot $\CC_{t}(u_A)$ is called \emph{generic} if all interactions of line-solitons
are at trivalent vertices or are $X$-crossings. 
\end{definition}


\begin{example}\label{ex:Gr49CP}
Consider some $A\in Gr_{4,9}$ which is the projection of an element
 $g\in G_{\v,\w}$ with
\begin{equation*}
{\bf w}=s_7s_8s_4s_5s_6s_7s_2s_3s_4s_5s_6s_1s_2s_3s_4s_5 \quad \text{ and }\quad
{\bf v}=s_711s_51s_7s_21s_4111s_21s_4s_5.
\end{equation*}
Then  $v=1$ and $\pi=vw^{-1} = 
(6,7,1,8,2,3,9,4,5).$
The matrix $g\in G_{\v,\w}$ is given by
\begin{align*}
g=
\dot s_7y_8(p_2)
&
y_4(p_3)\dot s_5y_6(p_5)x_7(m_6)\dot s_7^{-1} \dot s_2y_3(p_8)\dot s_4y_5(p_{10})y_6(p_{11})\\
&\cdot  
y_1(p_{12})x_2(m_{13})\dot s_2^{-1}y_3(p_{14})x_4(m_{15})\dot s_4^{-1}x_5(m_{16})\dot s_5^{-1}.
\end{align*}
The Go-diagram and the labeled Go-diagram are as follows:
\medskip
\[
\young[5,5,4,2][10][\hskip0.4cm\circle*{5},\hskip0.4cm\circle*{5},,\hskip0.4cm\circle*{5},,,,\hskip0.4cm\circle{5},,\hskip0.4cm\circle{5},\hskip0.4cm\circle*{5},,\hskip0.4cm\circle{5},,
,\hskip0.4cm\circle{5}]\hskip2cm \young[5,5,4,2][10][$-1$,$-1$,$p_{14}$,$-1$,$p_{12}$,
$p_{11}$,$p_{10}$,$1$,$p_8$,$1$,$-1$,$p_5$,$1$,$p_3$,$p_2$,$1$]
\]
\medskip\noindent
The $A$-matrix is then given by
\[
A=\begin{pmatrix}
-p_{12}p_{14}& q_{13} & p_{14} & q_{15} & -m_{16} & 1 & 0 & 0 & 0 \\
0 & p_8p_{10}p_{11} & 0 & p_{11}(p_3+p_{10}) & p_{11} & 0 & 1 & 0 & 0 \\
0  &  0 &  0  & -p_3p_5 & -p_5 & 0 & -m_6 & 1 & 0 \\
0 & 0 & 0 & 0 & 0 & 0 & p_2 & 0 & 1
\end{pmatrix},
\]
where the matrix entry $q_{13}=-m_{13}p_{14}+m_{15}p_8-m_{16}p_8p_{10}$ and $q_{15}=m_{15}-m_{16}(p_3+p_{10})$.  
In Figure \ref{fig:CP}, we show
contour plots $\CC_t(u_A)$ for the solution $u_A(x,y,t)$ at $t=-10, 0, 10$,
using the choice of 
parameters 
$(\kappa_1,\ldots,\kappa_9)=(-5,-3,-2,-1,0,1,2,3,4)$,  
$p_j=1$ for all $j$, and $m_l=0$ for all $\ell$. Note that:

\begin{itemize}
\item[(a)] For $y\gg 0$, there are four unbounded 
line-solitons, whose types from right to left are:
\[
[1,6],\quad [2,7],\quad [4,8],\quad {\rm and}\quad [7,9]
\]
\item[(b)] For $y\ll 0$, there are five unbounded 
line-solitons, whose types  from left to right are:
\[
[1,3],\quad [2,5],\quad [3,6],\quad [4,8],\quad {\rm and}\quad [5,9]
\]
\end{itemize}
Apparently the line-solitons for $y\gg0$ correspond to the excedances 
in $\pi=
(6,7,1,8,2,3,9,4,5)$, while 
those for $y\ll0$ correspond to the nonexcedances.
In Section \ref{sec:unbounded} we will give a theorem explaining
the relationship between the unbounded line-solitons of 
$\CC_t(u_A)$ and the positroid stratum containing $A$.
\begin{figure}[h]
\begin{center}
\includegraphics[height=1.9in]{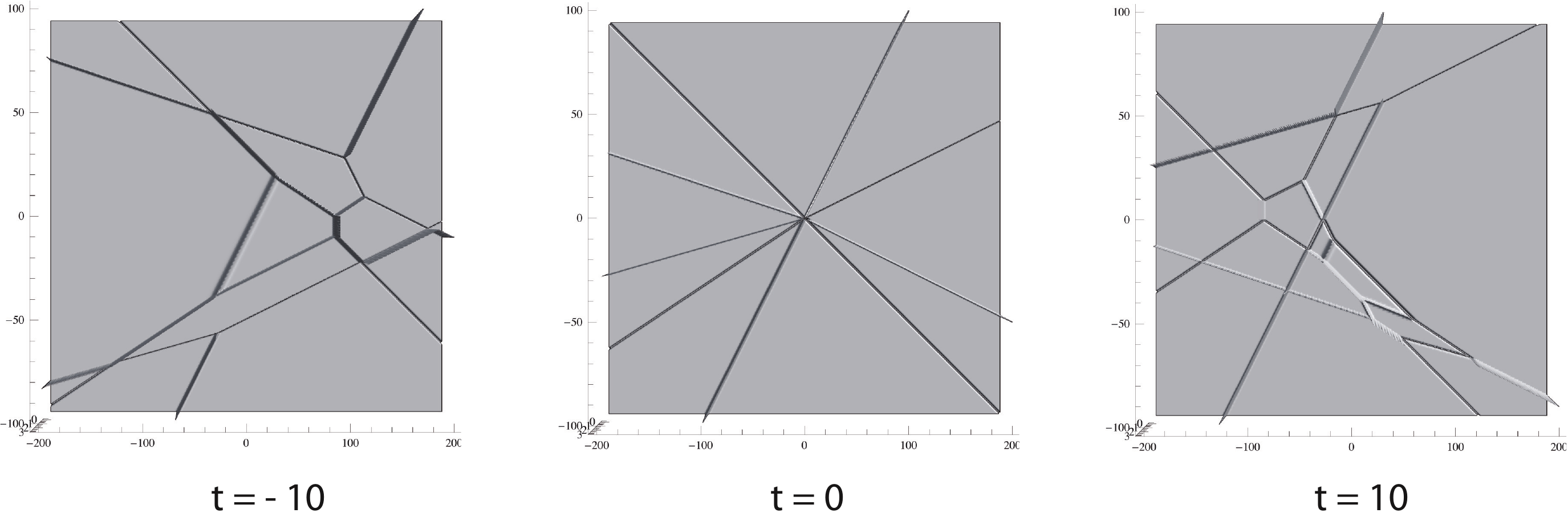}
\end{center}
\caption{Example of contour plots $\CC_t(u_A)$ for
$A\in Gr_{4,9}$.
The contour plots are obtained by ``Plot3D'' of Mathematica 
(see the details in the text).
 \label{fig:CP}}
\end{figure}

Note that if there are two adjacent regions of the contour plot
whose Pl\"ucker coordinates have different signs, then the line-soliton
separating them is singular.  
For example,
the line-soliton of type $[4,8]$ (the second soliton from the left in $y\gg0$) is 
singular, because the Pl\"ucker coordinates corresponding 
to the (dominant exponentials of the) adjacent regions are 
\[
\Delta_{1,2,4,9}=p_3p_5p_8p_{10}p_{11}p_{12}p_{14}=1\qquad {\rm and}\qquad
\Delta_{1,2,8,9}=-p_8p_{10}p_{11}p_{12}p_{14}=-1.
\]
\end{example}

\section{Unbounded line-solitons at $y\gg0$ and $y\ll0$}\label{sec:unbounded}

In this section we show that the unbounded line-solitons
at $|y|\gg0$ of a contour plot
$\CC_{t}(u_A)$ are determined by which positroid stratum
contains  $A$. 
Conversely, the unbounded line-solitons of
$\CC_{t}(u_A)$ determine which positroid stratum
$A$ lies in.
The main result
of this section is Theorem \ref{thm:soliton-perm}.

\begin{theorem}\label{thm:soliton-perm}
Let $A \in Gr_{k,n}$ lie in the positroid stratum 
$S_{\pi^:}$, where $\pi^: = 
(\pi,col)$.
Consider the contour plot
$\CC_{t}(u_A)$ for any time $t$.  
Then the excedances (respectively, nonexcedances)
of $\pi$ are in bijection with the 
unbounded line-solitons of $\CC_t(u_A)$ at $y\gg0$ (respectively, $y\ll0$).
More specifically, in 
$\CC_t(u_A)$, 
\begin{itemize}
\item[(a)]  
there is 
an unbounded line-soliton of  $[i,h]$-type at $y\gg0$
if and only if
$\pi(i)=h$ for $i<h$, 
\item[(b)] 
there is 
an unbounded line-soliton of $[i,h]$-type at $y\ll 0$  
if and only if
$\pi(h)=i$ for $i<h$.
\end{itemize}
Therefore $\pi^:$ determines
the unbounded line-solitons at $y\gg 0$ and $y\ll0$
of 
$\CC_{t}(u_A)$ for any time $t$.  

Conversely, given a 
contour plot
$\CC_{t}(u_A)$ at any time $t$ where $A \in Gr_{k,n}$,
one can construct $\pi^:=(\pi,col)$ such that $A \in S_{\pi^:}$ as follows.
The excedances and nonexcedances of $\pi$ are constructed as above
from the unbounded line-solitons. 
If there is an $h \in [n]$ such that $h\in J$ for every 
dominant exponential $E_J$ labeling the contour plot, 
then set $\pi(h)=h$ with $col(h)=1$.
If there is an $h\in [n]$ such that $h\notin J$ for any
dominant exponential $E_J$ labeling the contour plot,
then set $\pi(h)=h$ with $col(h)=-1$.
\end{theorem}
\begin{proof}
This result will follow immediately from 
Theorems \ref{th:rank} and \ref{th:unbounded} below.
\end{proof}

\begin{remark} Chakravarty and Kodama \cite[Prop. 2.6 and 2.9]{CK1} and 
\cite[Theorem 5]{CK3} associated a derangement to each 
\emph{irreducible} element $A$
in the \emph{totally non-negative part} $(Gr_{k,n})_{\geq 0}$ 
of the Grassmannian. 
Theorem \ref{thm:soliton-perm} generalizes their result by 
dropping the hypothesis of irreducibility and extending the 
setting from 
$(Gr_{k,n})_{\geq 0}$ to $Gr_{k,n}$.
\end{remark}



Before stating Theorems \ref{th:rank}
and \ref{th:unbounded},
we need to introduce some notation.

Given a  matrix $A$ with $n$ columns,
let $A(k,\dots,\ell)$ be the submatrix of $A$ obtained
from columns $k, k+1,\dots, \ell-1, \ell$, where the columns
are listed in the circular order 
$k, k+1,\dots,n-1, n, 1, 2, \dots,k-1$.  

The following result generalizes \cite[Lemma 3.4]{BC06} from
$(Gr_{k,n})_{\geq 0}$ to $Gr_{k,n}$.  
Our proof of Theorem \ref{th:rank} 
will be similar to that
of \cite{BC06}, but some arguments can be clarified using some 
basic theory of matroids.

\begin{theorem}\label{th:rank}
Let $A \in Gr_{k,n}$ and  consider the contour plot
$\CC_{t}(u_A)$ for any time $t$.  Choose $i, h \in\{1,\dots,n\}$
with $i<h$.

Then there is an unbounded line-soliton of $\CC_t(u_A)$ at $y\ll0$
labeled $[i,h]$
if and only if 
\begin{equation}\label{eq:rank}
\rank A(i,\dots,h-1) = \rank A(i+1,\dots,h) = \rank A(i,\dots,h) = 
\rank A(i+1,\dots,h-1) + 1.
\end{equation}
There is an unbounded line-soliton of $\CC_t(u_A)$ at $y\gg0$
labeled $[i,h]$
if and only if 
\begin{equation}\label{eq:rank2}
\rank A(h,\dots,i-1) = \rank A(h+1,\dots,i) = \rank A(h,\dots,i) = 
\rank A(h+1,\dots,i-1) + 1.
\end{equation}
\end{theorem}

Recall from Section \ref{soliton-background} that 
$\theta_j(x,y,z) = \kappa_j x + \kappa_j^2 y + \kappa_j^3 t$.
Fix $i,j\in \{1,\dots,n\}$, 
and let 
$L_{i,j}$ denote the line defined by $\theta_i = \theta_j$. 
Define subsets of $[n]$ by 
\begin{align*}
P &= \{\max(i,j)+1,\dots,\min(i,j)-1\}:=\{1,\dots,\min(i,j)-1\} \cup 
                                     \{\max(i,j)+1,\dots,n\} \text{ and }\\
Q&= \{\min(i,j)+1,\dots,\max(i,j)-1\}.
\end{align*}

In order to study the unbounded solitons at $y\gg 0$ and $y\ll 0$,
we first record the following lemma.

\begin{lemma}\cite[Lemma 3.1]{BC06}\label{cor:ij-order}
For $|y|\gg 0$, we have the following ordering among the 
$\theta_j$'s on the line $L_{i,j}$:
\begin{enumerate}
\item For $y \gg 0$ on the line $L_{i,j}$, $\theta_m < \theta_i=\theta_j$
for all $m \in Q$, and $\theta_m > \theta_i=\theta_j$
for all $m\in P$.
\item For $y \ll 0$ on the line $L_{i,j}$, $\theta_m > \theta_i=\theta_j$
for all $m \in Q$, and $\theta_m < \theta_i=\theta_j$
for all $m\in P$.
\end{enumerate}
\end{lemma}
\begin{proof}
For a fixed $t$, the equation of the line $L_{i,j}$ (which is defined by $\theta_i=\theta_j$) has the form
\[
x+(\kappa_i+\kappa_j)y={\rm constant}.
\]
Then along $L_{i,j}$, we have
\[
\theta_m-\theta_{m'}=(\kappa_m-\kappa_{m'})[(\kappa_m+\kappa_{m'})-(\kappa_i+\kappa_j)]y+\delta,
\]
where $\delta$ does not depend on $x$ or $y$.  The lemma
now follows from 
the fact that
$\kappa_1<\kappa_2<\cdots<\kappa_n$.
\end{proof}
Then it follows immediately that
\begin{corollary}\label{cor:tot-order}
For $y=y_0 \gg 0$  (respectively $y=y_0 \ll0$) 
there is a well-defined total order on 
$\theta_1,\dots,\theta_n$ on the line $L_{i,j}$ 
(with $\theta_i=\theta_j$), and this order does not 
change if we increase $y$  (resp., decrease $y$). 
\end{corollary}

The following matroidal result will be useful to us.

\begin{proposition}\cite[Theorem 1.8.5]{Oxley} \label{greedy}
Consider a matroid $\M$ of rank $k$ on the set $[n]$, and let 
$\omega = (\omega_1,\dots,\omega_n)\in \R^n$.
Define the \emph{weight} of a basis 
$J = (j_1,\dots,j_k)$ of $\M$ to be $\omega_{j_1} + \dots + \omega_{j_k}$.
Then the basis (or bases) of maximal weight are precisely the possible 
outcomes of the greedy algorithm: 
Start with $J = \emptyset$. At
each stage, look for an $\omega$-maximum element of $[n]$ which can be added to $J$
without making it dependent, and add it. After $k$ steps, output the basis $J$.
\end{proposition}

We now turn to the proof of Theorem \ref{th:rank}.  We will 
prove the result for unbounded line-solitons at $y\gg0$ (the other
part of the proof is analogous).

\begin{proof}
Let $\M$ be the matroid associated to $A$.  Its ground set $[n]$
is identified with the columns of $A$.
First suppose that for $i,j \in [n]$, with $i>j$ we have 
\begin{equation}\label{eq:rankhypothesis}
\rank A(i,\dots,j-1) = \rank A(i+1,\dots,j) = \rank A(i,\dots,j) = 
\rank A(i+1,\dots,j-1) + 1.
\end{equation}
By Corollary \ref{cor:tot-order}, at $y\gg0$ we have a 
well-defined total order on the $\theta_m$'s on the line $L_{i,j}$. 
At $y\gg0$ the problem
of computing the dominant exponential is equivalent to finding 
the basis of $\M$ with the maximal weight with respect to 
$(\theta_1,\dots,\theta_n)$.  

By Proposition \ref{greedy}, 
we can compute such a weight-maximal basis using the greedy algorithm.
By Corollary \ref{cor:ij-order}, the greedy algorithm
will first choose as many columns of $A(i+1,\dots,j-1)$
as possible. All of the $\theta_m$'s are distinct except for 
$\theta_i=\theta_j$, so there will be a unique way to 
add a maximal independent set of columns of $A(i+1,\dots,j-1)$
to the basis we are building. 
Note that by \eqref{eq:rankhypothesis},
the rank of $A(i+1,\dots,j-1)$ is less than $k$, so our 
weight-maximal basis must additionally contain at least one column that 
is not from $A(i+1,\dots,j-1)$.  By Corollary \ref{cor:ij-order},
columns $i$ and $j$ share a weight which is greater than any of the 
other remaining columns, so the next step is to add one of columns $i$ and $j$
to the basis we are building.
By \eqref{eq:rankhypothesis},
we cannot add both columns, because doing so 
will only increase the rank by $1$.  Therefore we now have two ways 
to build a weight-maximal basis, by adding either one of the columns $i$ and $j$.
If the two bases we are building do not yet have rank $k$, then there is now
a unique way to add columns from $A(j+1,\dots,i-1)$ to complete both of them.

We have now shown that along  $L_{i,j}$ at $y\gg0$, there are 
precisely two dominant exponentials,
$E_I$ and $E_J$, where $I = (J \cup \{i\}) \setminus \{j\}$.  Therefore
there is an unbounded line-soliton at $y\gg0$ labeled $[j,i]$.

Conversely, suppose that for $i>j$, there is an unbounded line-soliton
labeled $[j,i]$ at $y\gg0$.  Then on the line $L_{i,j}$ there are two 
dominant exponentials $E_I$ and $E_J$ with 
$J = (I \cup \{j\}) \setminus \{i\}$.  By Proposition \ref{greedy},
these must be the two outcomes of the greedy algorithm.
As before, by Corollary \ref{cor:ij-order}, the greedy 
algorithm will first choose as many columns of $A(i+1,\dots,j-1)$
as possible while keeping the collection
linearly independent, and then the next step will be to add exactly one
of the columns $i$ and $j$.  Since neither dominant exponential 
contains both $i$ and $j$, adding both columns must not increase
the rank more than adding just one of them.  Therefore 
equation \eqref{eq:rankhypothesis} must hold.
\end{proof}




\begin{theorem}\label{th:unbounded}
Let $A \in Gr_{k,n}$ lie in the positroid stratum 
$S_{\pi^:}$ where $\pi^: = (\pi,col)$.
Choose $1 \leq i<h \leq n$.
Then 
$\pi(h)=i$ if and only if equation \eqref{eq:rank} holds,
and $\pi(i)=h$ if and only if equation \eqref{eq:rank2} holds.
\end{theorem}

\begin{proof}
Let $\I = (I_1,\dots,I_n)$ be the Grassmann necklace associated to 
$A$, so $\pi^: = \pi^:(\I).$  Then 
by Lemma \ref{lem:necklace}, 
$I_i=\{x_1, x_2,\dots,x_{k}\}$ is the lexicographically
minimal $k$-subset
with respect to the  order 
$i < i+1 < \dots < n < 1 < \dots < i-1$
such that $\Delta_{I_i}(A) \neq 0$. 
Similarly
$I_{i+1}$ is the lexicographically
minimal $k$-subset 
with respect to the order 
$i+1 < \dots < n < 1 < \dots < i-1 < i$
such that $\Delta_{I_{i+1}}(A) \neq 0$.

We will prove the first statement of the theorem (the 
proof of the second is analogous, so we omit it.)
Suppose that 
$\pi(h)=i$.  
Then $x_1 = i$; otherwise 
the $i$th column of $A$ is the zero-vector and $\pi(i)=i$.
Using Definition \ref{necklace-to-perm} and 
Lemma \ref{lem:necklace},
$h$ has the following characterization.
Consider the column indices
in the order $i+1, i+2, \dots, n, 1, 2, \dots, i$ and greedily choose
the earliest index $h$ such that the columns of $A$
indexed by the set $\{x_2,\dots,x_{k}\} \cup \{h\}$
are linearly independent.
Then $I_{i+1} =(I_i\setminus \{i\}) \cup \{h\}$. 

Now consider the ranks of various submatrices of $A$ obtained
by selecting certain columns.

{\it Claim 0.}  $\rank A(i+1,\dots, h-1,h)=
  1+\rank A(i+1,\dots, h-1)$.  This claim follows from
the characterization of $h$ and the fact that 
$I_{i+1}$ is the lexicographically minimal $k$-subset with 
respect to the order $i+1<\dots<n<1<\dots<i$
such that $\Delta_{I_{i+1}}(A) \neq 0$.

{\it Claim 1.}  $\rank A(i,i+1,\dots,h) = \rank A(i,i+1,\dots,h-1)$.
To prove this claim, we consider two cases.  Either
$i <_i h <_i x_k$ or $i <_i x_k <_i h$, where $<_i$
is the total order $i<i+1<\dots<n<1<\dots<i-1$.
In the first case, the claim follows, because
$h$ is not contained in the set $I_i$.
In the second case,
$\rank A(i,i+1,i+2,\dots,x_k) = k$, and the index set
$\{i,i+1,\dots,x_k\}$ is a proper subset of
$\{i,i+1,\dots,h\}$, so
$\rank A(i,\dots,h) = \rank A(i,\dots,h-1)=k$.

Now let $R = \rank A(i+1,i+2,\dots,h-1)$.
By Claim 0, $\rank A(i+1,\dots,h) = R+1$.
Therefore we have
$\rank A(i,\dots,h) \geq \rank A(i+1,\dots,h) = R+1$.
By Claim 1, $\rank A(i,\dots,h) = \rank A(i,\dots,h-1)$,
but $\rank A(i,\dots,h-1) \leq R+1$, so $\rank A(i,\dots,h) \leq R+1$.
We now have $\rank A(i,\dots,h) = R+1$.
But also $\rank A(i,\dots,h-1)=\rank A(i,\dots,h)=R+1$.
Therefore
$\rank A(i,i+1,\dots,h-1)= \rank A(i+1,\dots,h-1,h)=\rank A(i,\dots,h)=
  \rank A(i+1,\dots,h-1) +1$, as desired.

Conversely, suppose that 
$\rank A(i,i+1,\dots,h-1)= \rank A(i+1,\dots,h-1,h)=\rank A(i,\dots,h)=
  \rank A(i+1,\dots,h-1) +1$.  Let $I_i$ and $I_{i+1}$ be the 
lexicographically minimal $k$-subsets with respect to the 
total orders $<_i$ and $<_{i+1}$, such that 
$\Delta_{I_i}(A) \neq 0$ and $\Delta_{I_{i+1}}(A) \neq 0$.
Since 
$\rank A(i,i+1,\dots,h-1)= \rank A(i,\dots,h)$, we have
$h \notin I_i$.  And since 
$\rank A(i+1,\dots,h-1,h)=
  \rank A(i+1,\dots,h-1) +1$, we have $h \in I_{i+1}$.
We now claim that $i \in I_i$.
Otherwise, by the definition of Grassmann necklace,
$I_{i+1} = I_i$, which contradicts the fact that
$\rank A(i,i+1,\dots,h-1) = \rank A(i+1,\dots,h-1)+1$.
Therefore the claim holds, and by 
Definition \ref{necklace-to-perm}, 
we must have $\pi(h)=i$.
\end{proof}


\section{Soliton graphs and generalized plabic graphs}\label{sec:solitonplabic}

The following notion of {\it soliton graph}
forgets the metric
data of the contour plot, but preserves
the data of how line-solitons interact and which exponentials dominate.

\begin{definition}\label{soliton-graph}
Let $A \in Gr_{k,n}$ and consider a generic contour plot $\CC_{t}(u_A)$ 
for some time $t$.
Color a trivalent
vertex black (respectively, white)
if it has a unique edge extending downwards (respectively, upwards) from it.
We preserve the labeling of regions and edges that was used 
in the contour plot: we label a region by $E_I$ if the dominant 
exponential in that region is $\Delta_I E_I$, and label
each line-soliton by its {\it type} $[i,j]$ 
(see Lemma \ref{separating}).
We also preserve the topology of the graph,
but forget the metric structure.
We call this labeled graph with bicolored vertices
the \emph{soliton graph} $G_{t_0}(u_A)$.
\end{definition}

\begin{example}
We continue 
Example \ref{ex:Gr49CP}.
Figure \ref{contour-soliton} contains
the same contour plot $\CC_t(u_A)$ as that at the left of 
Figure \ref{fig:CP}.
One may use Lemma \ref{separating} to label all regions and edges
in the soliton graph.
After computing the Pl\"ucker coordinates, 
one can identify the singular solitons, which are indicated
by the dotted lines in the soliton graph.
\begin{figure}[h]
\begin{center}
\includegraphics[height=1.8in]{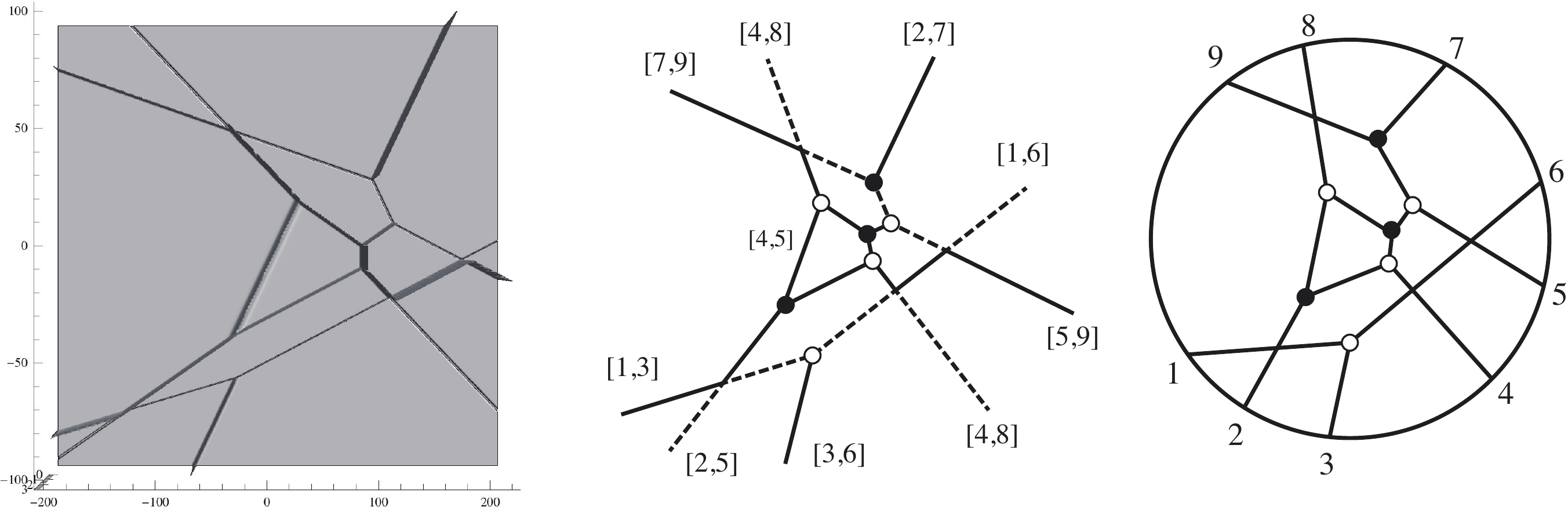}
\end{center}
\caption{Example of a contour plot $\CC_t(u_A)$, its
soliton graph $C=G_t(u_A)$, and its generalized plabic
graph $Pl(C)$.   The parameters used are
those from Example \ref{ex:Gr49CP}.
In particular,
$(\kappa_1,\dots,\kappa_9)=(-5,-3,-2,-1,0,1,2,3,4)$, and 
$\pi = (6,7,1,8,2,3,9,4,5).$
 \label{contour-soliton}}
\end{figure}
\end{example}


We now describe how to pass 
from a soliton graph to a \emph{generalized plabic graph}. 

\begin{definition}
A \emph{generalized plabic graph\/}
is an undirected graph $G$ drawn inside a disk
with $n$ \emph{boundary vertices\/} labeled
$\{1,\dots,n\}$.
We require that each boundary vertex $i$ is either isolated 
(in which case it is colored with color $1$ or $-1$), or
is incident
to a single edge; and each internal vertex is colored black or white.
Edges are allowed to cross each 
other in an $X$-crossing  (which is not considered to be a vertex). 
\end{definition}

By Theorem \ref{thm:soliton-perm}, the following construction
is well-defined.
\begin{definition}\label{soliton2plabic}
Fix a positroid stratum $\S_{\pi^{:}}$ of $Gr_{k,n}$ where
$\pi^{:} = (\pi,col)$.
To each soliton graph $C$ coming from a point of that stratum we associate 
a generalized plabic graph
$Pl(C)$ by:
\begin{itemize}
\item embedding $C$ into a disk, so that each unbounded line-soliton
of $C$ ends at a \emph{boundary vertex};
\item labeling the boundary vertex 
incident to the edge with labels $i$ and $\pi(i)$ 
by $\pi(i)$;
\item adding an isolated boundary vertex labeled $h$ with color
$1$ (respectively, $-1$) whenever $h\in J$ for each region label $E_J$
(respectively, whenever $h\notin J$ for any region label $E_J$);
\item forgetting the labels of all edges
and regions.
\end{itemize}
\end{definition}
See Figure \ref{contour-soliton} for a soliton graph $C$ together with the 
corresponding generalized plabic graph $Pl(C)$.

\begin{definition}\label{gen:trip}
Given a generalized plabic graph $G$,
the \emph{trip} $T_i$ is the directed path which starts at the boundary vertex 
$i$, and follows the ``rules of the road": it turns right at a 
black vertex,  left at a white vertex, and goes straight through the 
$X$-crossings.  
Note that $T_i$  will also 
end at a boundary vertex.  
If $i$ is an isolated vertex, then 
$T_i$ starts and ends at $i$.  
Define $\pi_G(i)=j$ whenever $T_i$
ends at $j$. 
It is not hard to show that $\pi_G$ is a permutation,
which we call the \emph{trip permutation}.
\end{definition}

We use the trips to label the edges and regions
of each generalized plabic graph.

\begin{definition}\label{labels}
Given a generalized plabic graph $G$, 
start at each non-isolated boundary
vertex $i$ and label every edge along trip $T_i$ with $i$.
Such a trip divides the disk containing $G$ into two parts: 
the part to the left of $T_i$, and the part to the right.
Place an $i$ in every region which is to the left of $T_i$.
If $h$ is an isolated boundary vertex with color $1$,
put an $h$ in every region of $G$.
After repeating this procedure for each boundary vertex,
each edge will be labeled by up to two numbers (between $1$ and $n$),
and each region will be labeled by a collection of numbers.
Two regions separated by an edge labeled by both $i$ and $j$ will have 
region labels $S$ and $(S\setminus \{i\}) \cup \{j\}$.
When an edge is labeled by two numbers $i<j$, we write $[i,j]$
on that edge, or $\{i,j\}$ or $\{j,i\}$ if we do not wish to specify
the order of $i$ and $j$.  
\end{definition}

Although the following result was proved for irreducible cells of  
$(Gr_{k,n})_{\geq 0}$, the same proof holds for arbitrary
positroid strata of $Gr_{k,n}$.
\begin{theorem}\label{soliton-plabic}\cite[Theorem 7.6]{KW2}
Consider a soliton graph $C=G_t(u_A)$ coming from a point $A$
of a positroid stratum
$\S_{\pi^:}$, where $\pi^: = (\pi,col)$.
Then the trip permutation of $Pl(C)$ is $\pi$, 
and by labeling edges of $Pl(C)$ according
to Definition \ref{labels}, we will recover the original edge 
and region
labels in $C$.  
\end{theorem}

We invite the reader to verify Theorem \ref{soliton-plabic} for 
the graphs in Figure \ref{contour-soliton}.

\begin{remark} By Theorem \ref{soliton-plabic}, we can identify 
each soliton graph $C$ with its generalized plabic graph $Pl(C)$.
From now on, we will often ignore the labels of edges and regions
of a soliton graph, and simply record the labels on boundary vertices.
\end{remark}

\section{The contour plot  for $t\ll0$}\label{sec:t<<0}

Consider a matroid stratum $S_{\mathcal{M}}$ contained in 
the Deodhar component $S_{{D}}$,
where $D$ is the corresponding or Go-diagram.
From Definition \ref{contour} it is clear that the contour plot associated to 
any $A\in S_{\mathcal{M}}$ depends only on $\mathcal{M}$, not on $A$.  
In fact for $t\ll0$ a stronger statement is true -- 
the contour plot for 
any $A \in S_{\mathcal{M}} \subset S_D$ depends only on $D$, 
and not on $\mathcal{M}$.
In this section
we will explain how to use $D$ to construct first a 
generalized plabic graph $G_-(D)$,
and then the contour plot $\CC_{t}(\mathcal{M})$ for $t\ll 0.$

\subsection{Definition of the contour plot for $t \ll0$.}\label{sec:contour}
Recall from \eqref{f-contour} the definition of $f_{\mathcal{M}}(x,y,t)$.
To understand how it behaves for $t \ll 0$, let us
rescale everything by $t$.
Define $\bar{x} = \frac{x}{t}$ and $\bar{y} = \frac{y}{t}$,
and set
\[
\phi_i(\bar{x}, \bar{y}) = \kappa_i \bar{x} + 
\kappa_i^2 \bar{y} + \kappa_i^3,
\]
that is, $\kappa_i x + \kappa_i^2 y + \kappa_i^3 t = t\phi_i(\bar{x},\bar{y})$.
Note that because $t$ is negative,  $x$ and $y$ have the opposite
signs of $\bar{x}$ and $\bar{y}$.  This
leads to the following definition of the contour plot
for $t \ll 0$.

\begin{definition}\label{contourplot-infinity}
We define the contour plot
$\CC_{-\infty}(\mathcal{M})$
to be the locus in $\R^2$ where
\begin{equation*}
\underset{J\in\mathcal{M}}\min \left\{
                     \sum_{i=1}^k \phi_{j_i}(\bar{x},\bar{y}) \right\}\quad 
\text{ is not linear .}
\end{equation*}
\end{definition}

\begin{remark}\label{rem:rotate}
After a $180^{\circ}$ rotation,  $\CC_{-\infty}(\mathcal{M})$ is the limit of
 $\CC_{t}(u_A)$  as $t\to -\infty$, for any $A\in S_{\mathcal{M}}$.
Note that the rotation is required because the positive $x$-axis
(respectively, $y$-axis)
corresponds to the negative $\bar{x}$-axis (respectively,
$\bar{y}$-axis).
\end{remark}

\begin{definition}\label{def:v}
Define $v_{i,\ell,m}$ to be the point in $\R^2$ where
$\phi_i(\bar{x},\bar{y}) = 
\phi_{\ell}(\bar{x},\bar{y}) = 
\phi_m(\bar{x},\bar{y}).$
A simple calculation yields that 
the point $v_{i,\ell,m}$ has the following
coordinates in the $\bar{x}\bar{y}$-plane:
\[
v_{i,\ell,m}=(\kappa_i \kappa_{\ell} + \kappa_i \kappa_m +
\kappa_{\ell} \kappa_m, -(\kappa_i+\kappa_{\ell}+\kappa_m)).
\]
\end{definition}
Some of the points $v_{i,\ell,m}\in \mathbb{R}^2$ correspond to trivalent vertices in the contour plots we construct; such a point is the location
of the resonant interaction of three line-solitons of types $[i,\ell]$, $[\ell,m]$ and $[i,m]$
(see Theorem \ref{t<<0} below).
Because of our assumption on the genericity of the $\kappa$-parameters,
those points are all distinct.

\subsection{Main results on the contour plot for $t \ll0 $}

The results of this section generalize those of 
\cite[Section 8]{KW} to a soliton solution coming from an arbitrary point of the real Grassmannian (not just the non-negative part).
We start by giving an algorithm to construct a 
generalized plabic graph $G_-(D)$,
which will be used to construct  $\CC_{-\infty}(\mathcal{M})$.
Figure \ref{GoPlabic} illustrates the steps of Algorithm \ref{GoToPlabic},
starting
from the Go-diagram of the Deodhar component
$S_{D}$ where $D$ is as in the upper left corner of Figure \ref{GoPlabic}.

\begin{algorithm}  \label{GoToPlabic} From a Go-diagram $D$ to $G_-(D)$:
\begin{enumerate}
\item Start with a Go-diagram $D$ contained in a $k\times (n-k)$
rectangle, and
replace each \wstn, \bstn, and blank box by a cross, a cross, and a
pair of {\it elbows},
respectively.  Label the $n$ edges along the southeast border 
of the Young diagram by the numbers $1$ to $n$,
from northeast to southwest.  The configuration of crosses and elbows forms $n$ ``pipes" 
which travel from the southeast border to the northwest border; label the endpoint of 
each pipe by the label of its starting point.
\item Add a pair of black and white vertices to each pair of elbows,
and connect them by an edge,
as shown
in the upper right of Figure \ref{GoPlabic}.  Forget the labels
of the southeast border.  If there is an endpoint of a pipe on the east or south border whose pipe
starts by going straight, then erase the straight portion preceding the first elbow.
If there is a horizontal (respectively, vertical) pipe starting at $i$ with no elbows,
then erase it, and add an isolated boundary vertex labeled $i$
with color $1$ (respectively, $-1$).
\item Forget any degree $2$ vertices, and forget
any edges of the graph which end
at the southeast border of the diagram.
Denote the resulting
graph  $G_-(D)$.
\item After embedding the graph in a disk
with $n$ boundary vertices (including isolated vertices)
we obtain a generalized plabic graph,
which we also denote $G_-(D)$.
If desired, stretch and rotate $G_-(D)$ so that the boundary vertices
at the west side of the diagram are at the north instead. 
\end{enumerate}
\end{algorithm}
\begin{figure}[h]
\centering
\includegraphics[height=3.3in]{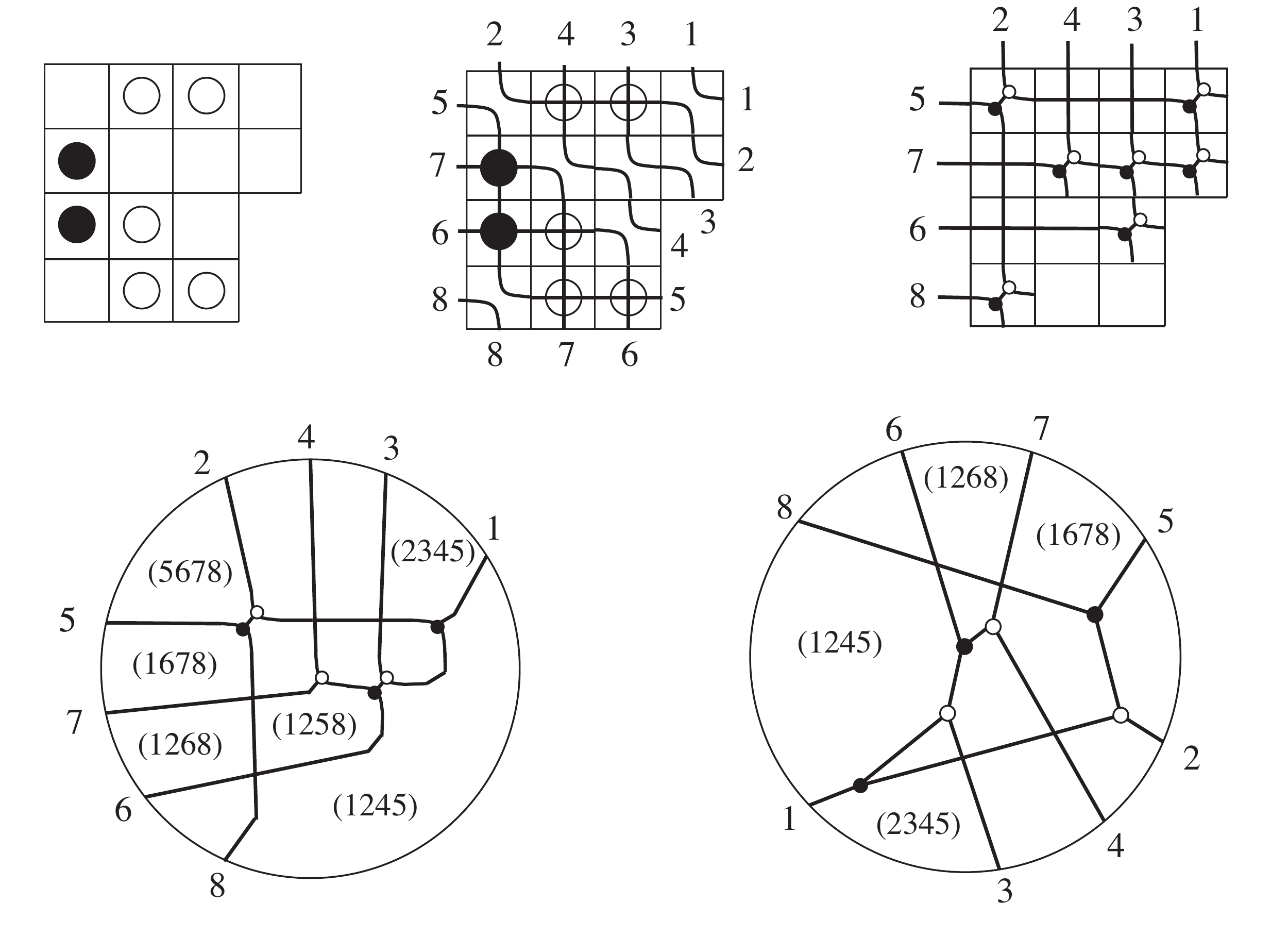}
\caption{Construction of the generalized plabic graph $G_-(D)$ associated to the Go-diagram $D$.  The labels of the regions of the graph indicate the index sets
of the corresponding Pl\"ucker coordinates.
Using the notation of Definition \ref{def:DtoPi}, we have
$\pi(D)=vw^{-1}=(5,7,1,6,8,3,4,2)$.}
\label{GoPlabic}
\end{figure}

\begin{remark}\label{rem:G-Dtest}
If there are no black stones in $D$, then this algorithm reduces to 
\cite[Algorithm 8.7]{KW2}.  In this case, by \cite[Theorem 11.15]{KW2},
the Pl\"ucker coordinates corresponding to the regions of $G_-(D)$ include 
the set of minors $\J$ described in Theorem \ref{th:TPTest}.  In particular,
the set of Pl\"ucker coordinates labeling the regions of $G_-(D)$
comprise a positivity test for $S_D$.
\end{remark}

The following is the main result of this section.   
\begin{theorem}\label{t<<0}
Choose a matroid stratum $S_{\mathcal{M}}$ and let $S_D$ be the 
Deodhar component containing $S_{\mathcal{M}}$.
Recall the definition of $\pi(D)$ from 
Definition \ref{def:DtoPi}.
Use Algorithm \ref{GoToPlabic} 
to obtain $G_-(D)$.  
Then $G_-(D)$ has trip permutation $\pi(D)$, and we can use it to explicitly
construct $\CC_{-\infty}(\M)$ as follows.
Label the edges of $G_-(D)$
according to the rules of the road.  Label by $v_{i,\ell,m}$
each trivalent
vertex which is incident to edges labeled $[i,\ell]$, $[i,m]$, and $[\ell,m]$,
and give that vertex the coordinates
$(\bar{x},\bar{y}) = (\kappa_i \kappa_{\ell}+\kappa_i \kappa_m + \kappa_{\ell} \kappa_m,
-(\kappa_i+\kappa_{\ell}+\kappa_m))$.
Replace each edge labeled $[i,j]$ which ends 
at a boundary vertex by an unbounded line-soliton 
with slope $\kappa_i + \kappa_j$.
(Each edge labeled $[i,j]$ between two trivalent vertices 
will automatically
have slope
$\kappa_i + \kappa_j$.)
In particular, $\CC_{-\infty}(\M)$ is determined by $D$.
Recall from Remark \ref{rem:rotate} that after a
$180^{\circ}$ rotation,  $\CC_{-\infty}(\M)$ is the limit
of $\CC_{t}(u_A)$ as $t\to -\infty$, for any $A\in S_{\mathcal{M}}$.
\end{theorem}
\begin{remark}
Since the contour plot $\CC_{-\infty}(\M)$ depends only on $D$, we also 
refer to it as $\CC_{-\infty}(D)$.
\end{remark}

\begin{remark}\label{rem:t>>0}
The results of this section may be extended to the case $t \gg0$
by duality considerations (similar to the way in which
our previous paper \cite{KW2} described 
contour plots for both $t\ll0$ and $t\gg0$).  Note that 
the Deodhar decomposition of $Gr_{k,n}$ depends on a choice
of ordered basis $(e_1,\dots,e_n)$.  Using the ordered basis
$(e_n,\dots,e_1)$ instead and the corresponding
Deodhar decomposition, one may explicitly describe
contour plots at $t\gg0$.
\end{remark}

\begin{remark}\label{rem:differ}
Depending on the choice of the parameters $\kappa_i$,
the contour plot 
$\CC_{-\infty}(D)$ may have a slightly different topological structure than the soliton
graph $G_-(D)$.  While the incidences of line-solitons
with trivalent vertices are determined by $G_-(D)$, the locations 
of $X$-crossings may vary based on the $\kappa_i$'s. More specifically,
changing the $\kappa_i$'s may change the contour plot via a sequence
of \emph{slides}, see Section \ref{sec:slides}.
\end{remark}

Our proof of Theorem \ref{t<<0}
is similar to the proof of \cite[Theorem 8.9]{KW}.
The main
strategy is to use induction on the number of rows in the 
Go-diagram $D$.  More specifically,
let $D'$ denote the Go-diagram $D$ with its
top row removed.  
In Lemma \ref{lem:algo} we will explain that $G_-(D')$ can be seen 
as a labeled subgraph of $G_-(D)$.  In Theorem \ref{induction}, we will explain
that there is a polyhedral
subset of $\CC_{-\infty}(D)$ which coincides with 
$\CC_{-\infty}({D'})$.  And moreover, every vertex of 
$\CC_{-\infty}({D'})$ appears as a vertex of 
$\CC_{-\infty}({D})$.  By induction we 
can assume that Theorem \ref{t<<0} correctly computes 
$\CC_{-\infty}(D')$, which in turn provides us with 
a description of 
``most" of $\CC_{-\infty}(D)$, including all 
line-solitons and vertices whose indices do not include $1$.  
On the other hand, Theorem \ref{thm:soliton-perm} gives
a complete description of the unbounded solitons of both 
$\CC_{-\infty}({D'})$ and
$\CC_{-\infty}({D})$ in terms of $\pi(D')$ and $\pi(D)$.  
In particular, 
$\CC_{-\infty}({D})$ contains one more unbounded soliton
at $y\gg 0$ than does 
$\CC_{-\infty}({D})$.
This information together with 
the resonance property allows us to complete the description of 
$\CC_{-\infty}({D})$ and match it up with the combinatorics
of $G_-(D)$.

\begin{lemma}\label{lem:asymptotics}
The generalized plabic graph $G_-(D)$ from Algorithm \ref{GoToPlabic} 
has trip permutation  $\pi(D)$.  
\end{lemma}
\begin{proof}
If we follow the rules of the road starting from
a boundary vertex of $G_-(D)$, we will first follow a
``pipe" southeast (compare the lower left and the top middle pictures in Figure \ref{GoPlabic})
and then travel straight west along the row or north along the column where that pipe
ended. Recall 
from Definition \ref{def:DtoPi} that
 $\pi(D)=vw^{-1}$. 
Noting that we can read off $v$ and $w$ from the pipes in the top middle
picture of Figure \ref{GoPlabic}, we see that following the rules of the road 
has the same effect as computing $vw^{-1}$.
\end{proof}

The next lemma 
explains the relationship between $G_-(D)$ and $G_-(D')$,
where $D'$ is the Go-diagram $D$ with the top row removed.
It should be clear after examining  Figure \ref{GoPlabic1}.

\begin{figure}[h]
\centering
\includegraphics[height=3.8in]{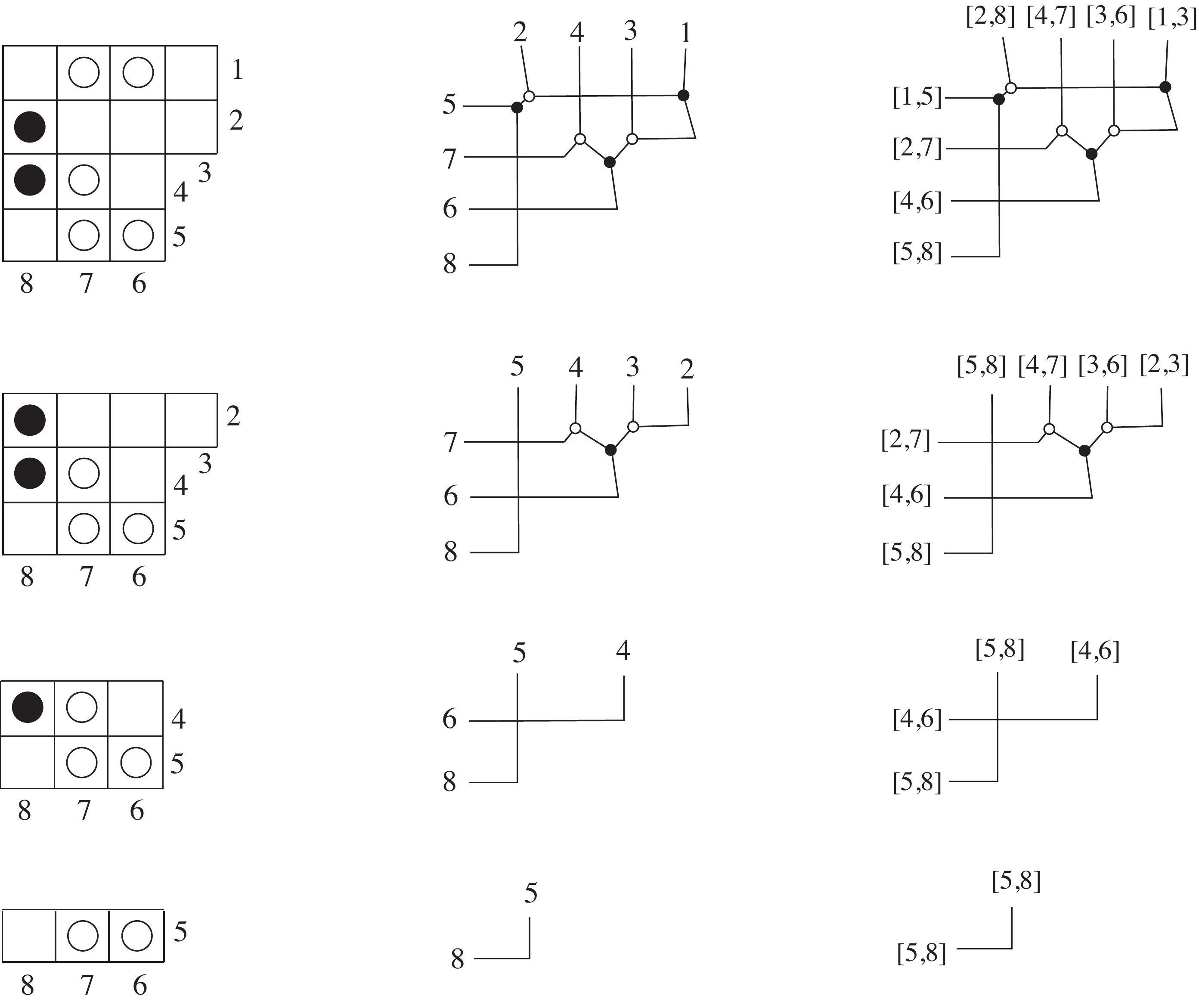}
\caption{Inductive construction of the generalized plabic graph $G_-(D)$
associated to the Go-diagram $D$, cf. Figure \ref{GoPlabic}.}
\label{GoPlabic1}
\end{figure}

\begin{lemma}\label{lem:algo}
Let $D$ be a Go-diagram with $k$ rows and $n-k$ columns, 
and let $G = G_-(D)$ be the edge-labeled plabic graph constructed by
Algorithm \ref{GoToPlabic}.
Form a new Go-diagram $D'$ from $D$ by removing the top row of $D$; 
suppose that $\ell$ is the sum of the 
number of rows and columns in $D'$.
Let $G'$ be the edge-labeled plabic graph associated to $D'$,
but instead of using the labels $\{1,2,\dots,\ell\}$, use the 
labels $\{n-\ell+1, n-\ell+2,\dots,n\}$.  
Let $h$ denote the label of the top row of $D$.
Then $G'$ is obtained from $G$ by removing the trip $T_h$ 
starting at $h$ and all edges to 
its right which have a trivalent vertex on $T_h$.
\end{lemma}

From now on, 
we will {\it assume without loss of generality that $i_1 = 1$}  is a 
pivot for $A \in S_D$.  


\begin{definition}\label{def:M'}
Let $\mathcal{M}$ be a matroid on $[n]$ such that $1$ is contained in at least
one base.  Let $\mathcal{M'}$ be the matroid
$\{ J\setminus \{1\} \ \vert \ 1\in J \text{ and } J\in \mathcal{M} \}.$
\end{definition}

Using arguments similar to those in the proof of Theorem \ref{p:Plucker},
one can verify the following.
\begin{lemma}\label{rem:project}
If $A \in S_{\M} \subset S_D$ is in row-echelon form and 
$A'$ is 
the span of rows $2,3,\dots,k$ in 
$A \in S_{\mathcal{M}} \subset Gr_{k,n}$, 
then $A' \in S_{\mathcal{M'}} \subset S_{D'}$, where 
$D'$ is 
obtained from $D$ by removing its top row.
\end{lemma}

The following result is a combination of \cite[Theorem 8.17]{KW} 
and \cite[Corollary 8.18]{KW}.  Although in \cite{KW} the context was 
$A \in (Gr_{k,n})_{\geq 0}$ and in this paper we are allowing 
$A \in Gr_{k,n}$, the proofs
from \cite{KW}  hold without any modification.
See Figure \ref{fig:induction} for an illustration of  the theorem.

\begin{theorem}\label{induction}\cite{KW}
Let $\mathcal{M}$ be a matroid such that $1$ is contained in at least one base.
Then there is an unbounded polyhedral subset
$\mathcal{R}$ of $\CC_{-\infty}({\mathcal{M}})$ whose boundary
is formed by line-solitons, such that 
every region in $\mathcal{R}$ is labeled by 
a dominant exponential 
$E_J$ such that $1\in J$.  In $\mathcal{R}$,
$\CC_{-\infty}(\mathcal{M})$ coincides with  
$\CC_{-\infty}(\mathcal{M'})$. 
Moreover, every region of 
$\CC_{-\infty}(\mathcal{M'})$ which is incident to a trivalent
vertex and labeled by $E_{J'}$ corresponds to a region of 
$\CC_{-\infty}(\mathcal{M})$ which is labeled by 
$E_{J' \cup \{1\}}$.

In particular, the set of trivalent vertices in $\CC_{-\infty}(\mathcal{M})$ is equal to 
the set of trivalent vertices in $\CC_{-\infty}(\mathcal{M'})$  together
with some vertices of the form $v_{1,b,c}$.  These vertices comprise the
vertices along the trip $T_1$ (the set of line-solitons labeled $[1,j]$ for any $j$).  In particular, every line-soliton
in $\CC_{-\infty}(\mathcal{M})$ which was not present in $\CC_{-\infty}(\mathcal{M'})$ and is not on $T_1$
must be unbounded.  And every new bounded line-soliton in
$\CC_{-\infty}(\mathcal{M})$ that did not come from a line-soliton 
in $\CC_{-\infty}(\mathcal{M'})$ is of type $[1,j]$ for some $j$.
\end{theorem}

\begin{figure}
\centering
\includegraphics[height=6cm]{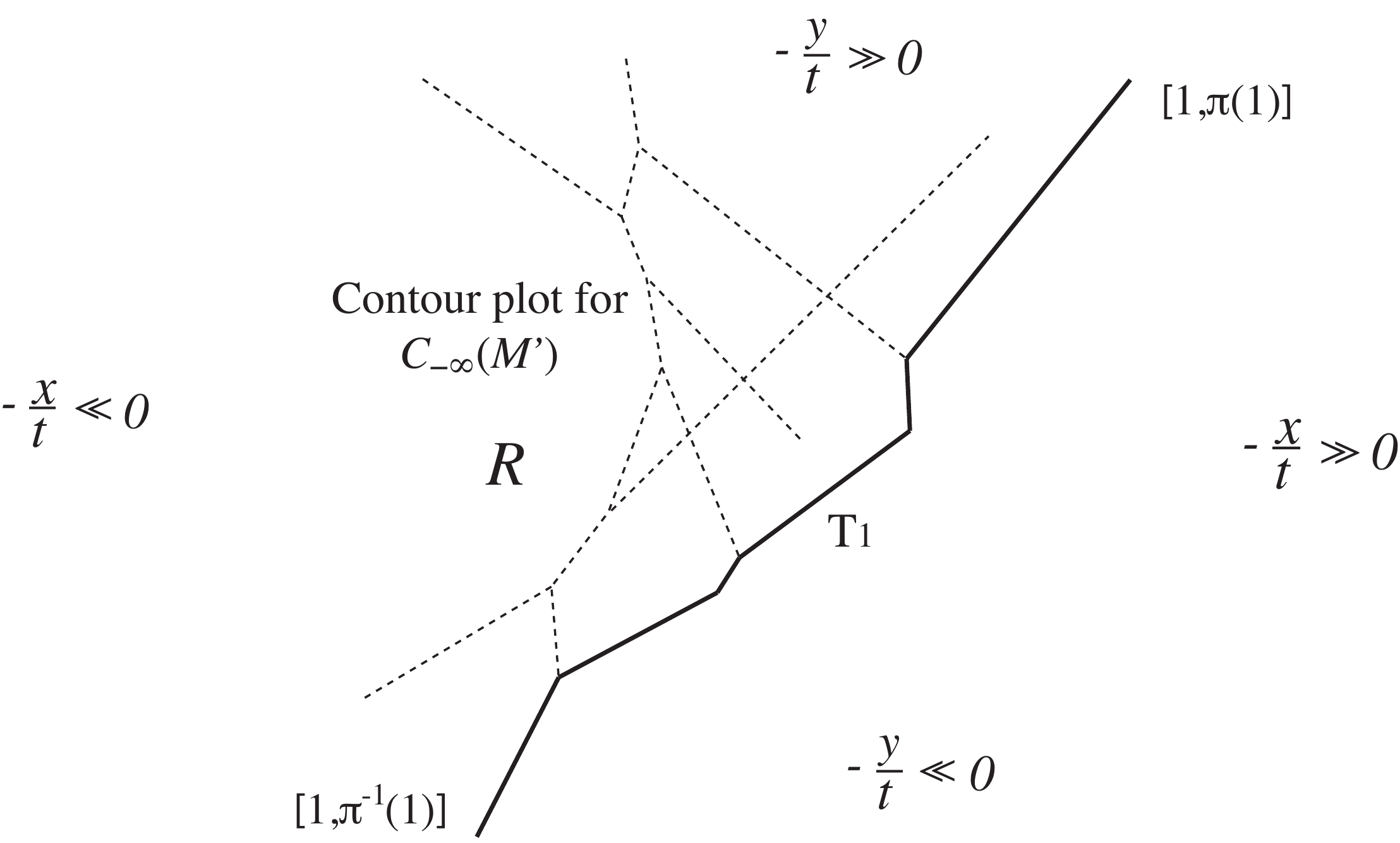}
\caption{The contour plot $\CC_{-\infty}(\mathcal{M'})$ within
the contour plot $\CC_{-\infty}(\mathcal{M}).$} \label{fig:induction}
\end{figure}

We  now prove Theorem \ref{t<<0}, using
the characterization of unbounded line-solitons
in Theorem \ref{thm:soliton-perm}.

\begin{proof}
Choose $A$ in the Deodhar component $S_D$.  Let $\M$ be the 
matroid such that $A\in S_{\M}$.  
We will prove Theorem \ref{t<<0} using induction on the number of rows of $A$.
Using the notation of 
Definition \ref{def:M'} and Lemma \ref{rem:project},
we have that 
$A'\in S_{\M'} \subset S_{D'}$.

By Theorem \ref{induction}, the contour plot
$\CC_{-\infty}(\mathcal{M})$ is equal to 
the contour plot $\CC_{-\infty}(\mathcal{M'})$ together with some trivalent vertices
of the form $v_{1,b,c}$, all edges along the trip $T_1$,
and some new unbounded line-solitons
(which are all to the right of the trip $T_1$).
By the inductive hypothesis, 
$\CC_{-\infty}(\mathcal{M'})$ is constructed by Theorem \ref{t<<0}; 
in particular, Algorithm \ref{GoToPlabic} produces a (generalized) plabic graph which
describes the trivalent vertices of 
$\CC_{-\infty}(\mathcal{M'})$ and the interactions of all line-solitons 
at trivalent vertices.

Using Lemma \ref{lem:asymptotics} 
and Theorem \ref{thm:soliton-perm}, we see that Algorithm \ref{GoToPlabic} produces
a (generalized) plabic graph whose labels on unbounded edges agree with 
the labels of the unbounded line-solitons for the contour plot $\CC_{-\infty}(\mathcal{M})$
of any $A\in S_{D}$.  The same is true for $A' \in S_{D'}$.

By Lemma \ref{lem:algo}, the plabic graph $G$ which Algorithm \ref{GoToPlabic}
associates to $D$ is equal to $G'$ together with the trip $T_1$ starting at $1$ 
at some new line-solitons emanating right from trivalent vertices of $T_1$.

We  now characterize the new vertices and line-solitons
which $\CC_{-\infty}(\mathcal{M})$ contains, but which 
$\CC_{-\infty}(\mathcal{M'})$ did not.  
We claim that the set of new vertices is precisely 
the set of $v_{1,b,c}$
(where $1<b<c$), such that 
either $c\to b$ is a nonexcedance of
$\pi = \pi(\mathcal{M})$, or $c\to b$ is a nonexcedance of 
$\pi' = \pi(\mathcal{M'})$, but not both.
Moreover, if $c\to b$ is a nonexcedance of $\pi$, then 
$v_{1,b,c}$ is white, while if $c\to b$ is a nonexcedance of 
$\pi'$, then $v_{1,b,c}$ is black.  The proof is identical
to that of the same claim in the proof of \cite[Theorem 8.8]{KW2}.

Now, if one analyzes the steps of Algorithm \ref{GoToPlabic} (see in particular the second and third 
diagrams in Figure \ref{GoPlabic}), it becomes apparent that the above description
also characterizes the set of new vertices which the algorithm associates to the top
row of the Go-diagram $D$.  In particular, the nonexcedances of the corresponding permutation 
$\pi$ correspond to the vertical edges at the top of the second and third diagrams; when
one labels these edges using the rules of the road, each edge gets the label $[b,c]$,
where $b$ comes from the label of its pipe, and $c$ comes from the label of its column
(shown at the bottom of the second diagram).  The nonexcedances of $\pi'$ are labeled in the 
same way but come from vertical edges which are present in the second row of $D$.
Therefore each new trivalent vertex in 
the top row gets the label $v_{1,b,c}$ where $b$ and $c$ are as above,
and where $c\to b$ is a nonexcedance of precisely one of $\pi$ and $\pi'$.

Finally, we discuss the order in which the vertices
$v_{1,b,c}$ occur along the trip $T_1$ in the contour plot.
First note that the trip $T_1$ starts at $y\ll 0$ and along each line-soliton it 
always heads up (towards $y\gg 0$). This follows from the resonance condition
(see e.g. \cite[Figure 9]{KW2} and take $i=1$).
Therefore the 
order in which we encounter the vertices $v_{1,b,c}$ along the trip
is given by the total order on the $y$-coordinates of the vertices,
namely $\kappa_1 + \kappa_b + \kappa_c$.

We now claim that this total order is identical 
to the total order on the positive integers $1+b+c$ -- 
that is, it does not depend on the 
choice of $\kappa_i$'s, as long as 
$\kappa_1 < \dots < \kappa_n$.  If we can show this,
then we will be done, because this is precisely the order
in which the new vertices occur along the trip 
$T_1$ in the graph $G_-(L)$.

To prove the claim,
it is enough to show that among the set of new vertices
$v_{1,b,c}$, there are not two of the form 
$v_{1,i,\ell}$ and $v_{1,j,k}$ where 
$i<j<k<\ell$.  To see this, recall
that the indices $b$ and $c$ of the new vertices $v_{1,b,c}$ 
can be read off from the second and third diagrams illustrating 
Algorithm \ref{GoToPlabic}:
$c$ will come from
the bottom label of the corresponding column, while
$b$ will come from the label of the pipe
that $v_{1,b,c}$ lies on.  Therefore, 
if there are two new vertices $v_{1,i,\ell}$
and $v_{1,j,k}$, then they must come from 
a pair of pipes which have crossed each 
other an odd number of times, as in Figure 
\ref{fig:pipecrossing2}.

 \begin{figure}[h]
\centering
\includegraphics[height=3.8cm]{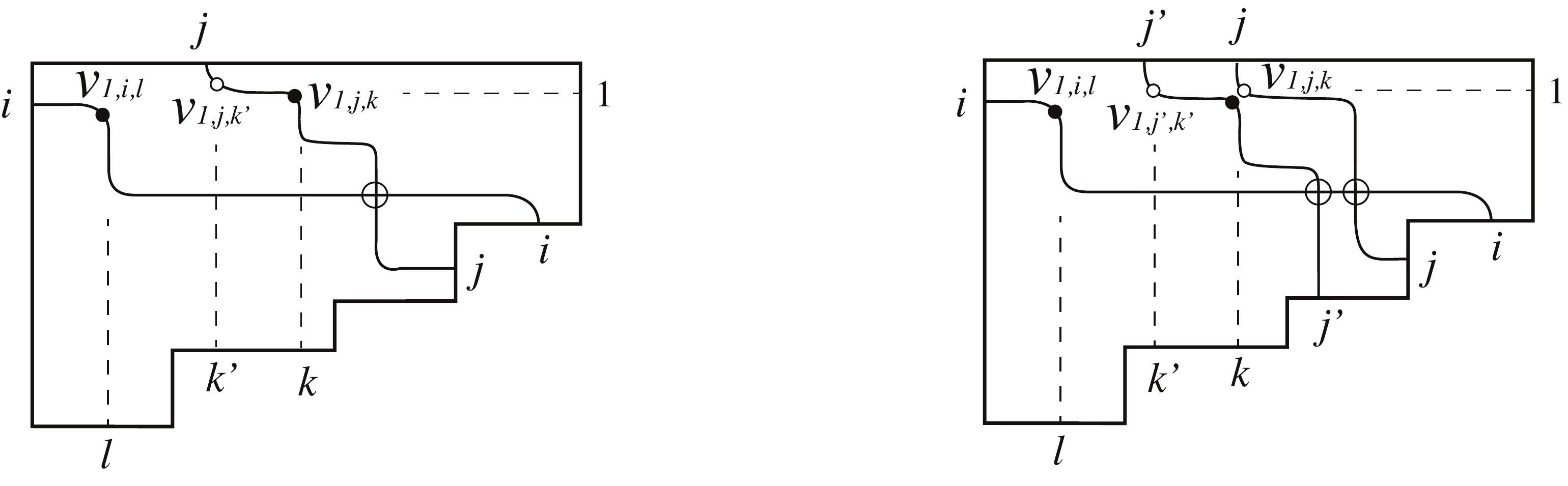}
\caption{}
\label{fig:pipecrossing2}
\end{figure}

Note that the second diagram of 
Figure \ref{GoPlabic} depicts a ``pipe dream" (or ``wiring diagram") 
encoding the 
distinguished subexpression $\v$ of a reduced expression $\w$.
If two pipes pass over each other in a given box
we will say that they {\it cross} at that box,
while if two pipes pass through the same box without crossing,
we will say that they {\it kiss} at that box.
Let us now follow a pair of pipes from southeast to northwest.  
The property of $\v$ being distinguished mean that two pipes starting
at $i$ and $j$ must not kiss each other after having 
crossed each other an odd number of times.

Assume that Algorithm \ref{GoToPlabic} produces 
two vertices $v_{1,i,\ell}$ and $v_{1,j,k}$ where 
$i<j<k<\ell$.   Choose such a pair of vertices which minimizes
$|\ell-k|$.  We consider two cases, based on whether 
$v_{1,j,k}$ is black or white.  In the first case
(see the left of Figure \ref{fig:pipecrossing2}),
since $v_{1,j,k}$ is black, its pipe $j$ will continue west
from $v_{1,j,k}$ and must eventually turn up, at some column $k'$
such that $k < k' < \ell$.  But then Algorithm \ref{GoToPlabic}
produces another vertex $v_{1,j,k'}$ such that 
$i<j<k'<\ell$, so this vertex together with $v_{1,i,\ell}$
form a pair of vertices where $|\ell-k'|<|\ell-k|$, contradicting
our assumption of minimality of $|\ell-k|$.

In the second case (see the right of Figure \ref{fig:pipecrossing2}),
since $v_{1,j,k}$ is white, there is another black 
vertex $v_{1,j',k}$
to its left in the same box $b$, whose pipe
starts at $j'$.  Because $\v$ is distinguished,
$j'$ must be greater than $j$. (Otherwise the 
pipes starting at $j$ and $j'$ would cross each 
other an odd number of times and then kiss at box $b$.)
Now since $v_{1,j',k}$ is black, its pipe
must travel west from it and eventually turn up, 
at some column $k'$ such that $j'<k'<\ell$.
But then Algorithm \ref{GoToPlabic}
produces another vertex $v_{1,j',k'}$ such that 
$i<j'<k'<\ell$.  But now we have 
a pair of vertices $v_{1,i,\ell}$ and $v_{1,j',k'}$
such that $i<j'<k'<\ell$ where 
$|\ell-k'|<|\ell-k|$.  This contradicts our assumption
of minimality of $|\ell-k|$, and completes the proof of the claim.

Finally, using Definition \ref{def:v} for the vertex $v_{i,\ell,m}$, we obtain the contour plot
from $G$ by giving the trivalent vertices the explicit coordinates
from Theorem \ref{t<<0}.
\end{proof}


\section{$X$-crossings, slides, and contour plots}\label{sec:slides}

In this section we discuss how 
our choice of the parameters $\kappa_i$ may affect the 
topology of the contour plot
$\CC_{-\infty}(D)$
(and hence $\CC_{t}(u_A)$ for $t\ll 0$ and $A \in S_D$),
 namely, by changing the locations of the 
$X$-crossings.  See Remark \ref{rem:differ}.  We also discuss
the relation between $X$-crossings and Pl\"ucker coordinates.

\subsection{Slides and the topology of contour plots}
The following definition will be useful for understanding
the dependence of the contour plot on the $\kappa_i$'s.

\begin{figure}[h]
\centering
\includegraphics[height=1.6in]{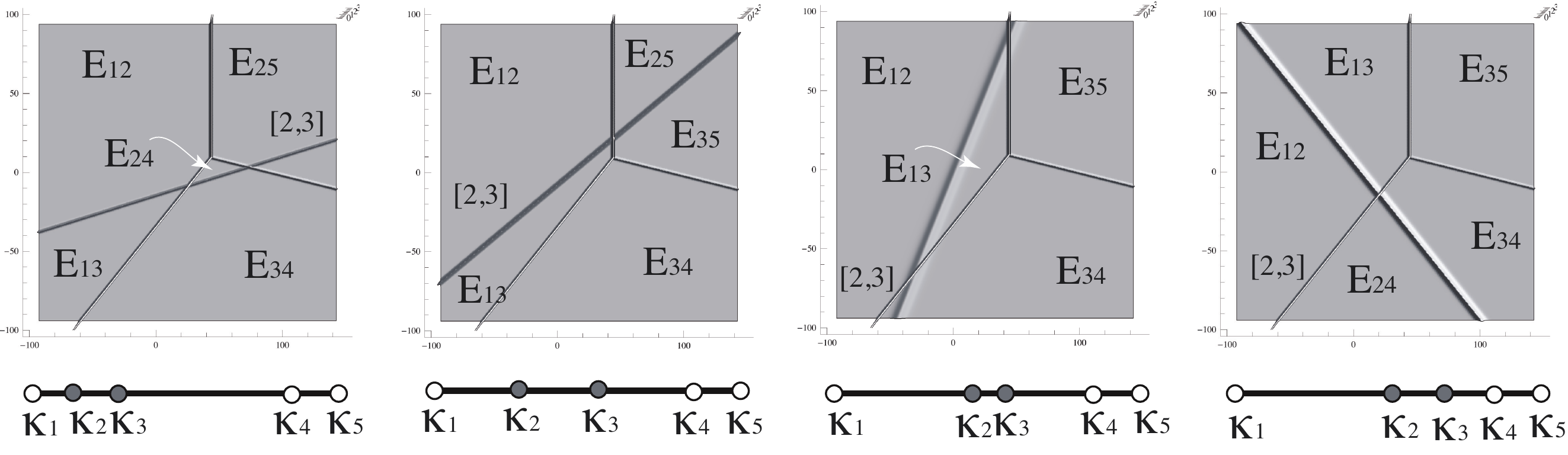}
\caption{Some slides involving white $X$-crossings. 
These contour plots correspond to the same Le-diagram $D$ with $\pi(D) = (5,3,2,1,4)$,
but they  differ from $G_-(D)$.}
\label{fig:slide}
\end{figure}

\begin{definition}\label{def:slide}
Consider a generalized plabic graph $G$ with at least one $X$-crossing.
Let $v_{a,b,c}$ be a trivalent vertex
(with edges labeled $[a,b]$, $[a,c]$, and $[b,c]$)
which has a small neighborhood $N$ containing one or two $X$-crossings
with a line labeled $[i,j]$, but no other trivalent vertices or $X$-crossings.
Here $\{a,b,c\}$ and $\{i,j\}$ must be disjoint.
Then a \emph{slide} is a local deformation of the graph $G$ which moves
the line $[i,j]$ so that it intersects a different set of edges of $v_{a,b,c}$,
creating or destroying at most one region in the process.
\end{definition}

See Figure \ref{fig:slide} for  examples.  
Recall the notions of black and white $X$-crossings from 
Definition \ref{def:blackwhiteX}.

\begin{remark}\label{rem:slides}
Theorem \ref{t<<0} determines everything about
the combinatorics and topology of the contour plot 
$\CC_{-\infty}(D)$ except for which pairs of 
line-solitons form an $X$-crossing.  Therefore if one 
deforms the parameters $\kappa_i$, the only way that 
the contour plot can change so as to change the topology
is via a sequence of slides.
\end{remark}


See Figure \ref{fig:Gr48AB} for an example
of two different contour plots associated to the same Go-diagram
and element $A\in Gr_{4,8}$, but 
obtained using different choices of the $\kappa$-parameters. The two contour plots differ
by precisely one slide. 
For another example, compare Figure \ref{contour-soliton} to Figure \ref{contour-soliton2}.
Both of them are based on the Go-diagram from Example \ref{ex:Gr49CP}
and the same matrix $A$.  The only difference is the value of 
$\kappa_1$.  Note that this affects the $X$-crossings formed by the unbounded
$[1,6]$
line-soliton, and that one contour plot can be obtained from the other via
a sequence of three slides.

\begin{figure}
\centering
\includegraphics[height=4.5cm]{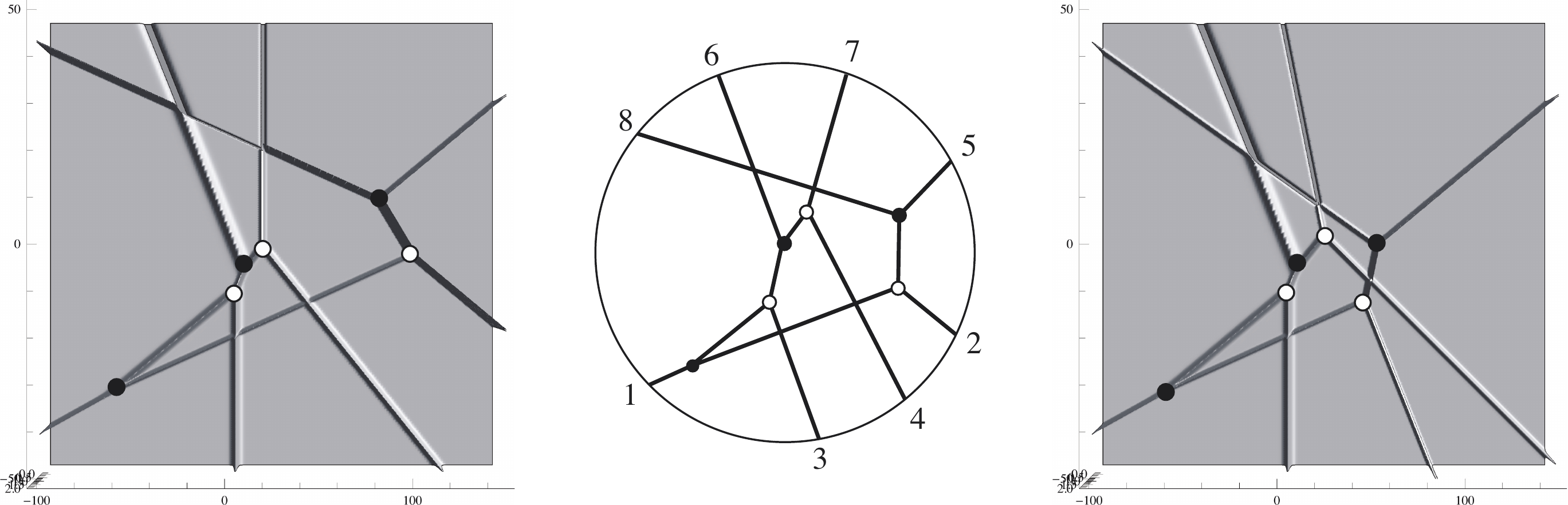}
\caption{Contour plots $\CC_t(u_A)$ constructed using 
the same $t$ and $A\in S_D \subset Gr_{4,8}$ 
but with different choices of the $\kappa$-parameters.
The left plot uses $(\kappa_1,\ldots,\kappa_8)=(-3.5,-2,-1,0,0.5,1,2,5)$ while
the right one uses $(-3.5,-2,-1,0,0.5,1,2.5,3)$.  This
affects the location of the $[4,7]$ line-soliton.
In the middle we have the generalized plabic graph $G_-(D)$
using the Go-diagram $D$ of Figure \ref{GoPlabic}.} \label{fig:Gr48AB}
\end{figure}

\begin{figure}[h]
\begin{center}
\includegraphics[height=1.8in]{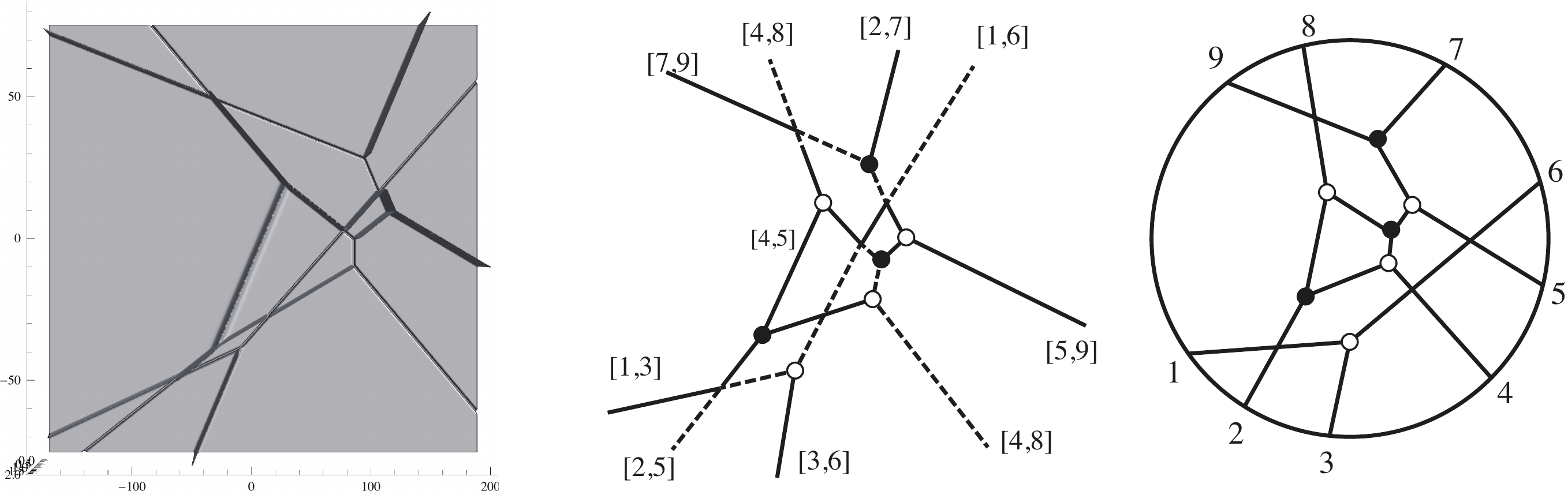}
\end{center}
\caption{A contour plot $\CC_t(u_A)$, soliton graph $C=G_t(u_A)$ and generalized plabic
graph $G_-(D)$ coming from a Go-diagram where $A\in S_D$.  The $\kappa$-parameters are the same as
those used for Figure \ref{contour-soliton} except that $\kappa_1=-3.1$ now, i.e. $(\kappa_1,\ldots,\kappa_9)=(-3.1,-3,-2,-1,0,1,2,3,4)$.
 \label{contour-soliton2}}
\end{figure}

We now show that a slide on a contour plot preserves the number of black $X$-crossings.
\begin{theorem}\label{th:blackX}
Consider two contour plots $\CC$ and $\CC'$ (for the same $A\in Gr_{k,n}$ and time $t$
but for different $\kappa$-parameters) which differ by a slide.
Then $\CC$ and $\CC'$ have the same number of black $X$-crossings.
\end{theorem}

\begin{proof}
Suppose that $\CC$ and $\CC'$ differ by a slide involving the trivalent vertex
$v_{a,b,c}$ and the line-soliton $[i,j]$ for $a<b<c$ and $i<j$, where the 
sets $\{a,b,c\}$ and $\{i,j\}$ are disjoint.
We assume that $v_{a,b,c}$ is white.  (The case where it is black is analogous.)
There are five cases to consider:
\begin{enumerate}
\item[Case 1.] $i<a<j<b<c$, which implies that $\kappa_i+\kappa_j < \kappa_a+\kappa_b < 
                                          \kappa_a+\kappa_c < \kappa_b+\kappa_c.$
\item[Case 2.] $i<a<b<j<c$, which implies that 
                 (a.) $\kappa_i+\kappa_j < \kappa_a+\kappa_b < 
                     \kappa_a +\kappa_c  < \kappa_b+\kappa_c$, or \\
                 (b.) $\kappa_a+\kappa_b < \kappa_i+\kappa_j < 
                     \kappa_a +\kappa_c  < \kappa_b+\kappa_c$.
\item[Case 3.] $a<i<b<j<c$, which implies that 
     (a.) $\kappa_a+\kappa_b < \kappa_i+\kappa_j < \kappa_a+\kappa_c < \kappa_b+\kappa_c$, or \\
     (b.) $\kappa_a+\kappa_b < \kappa_a+\kappa_c < \kappa_i+\kappa_j < \kappa_b+\kappa_c$.
\item[Case 4.] $a<i<b<c<j$, which implies that 
    (a.) $\kappa_a+\kappa_b < \kappa_a+\kappa_c < \kappa_i + \kappa_j < \kappa_b + \kappa_c$, or \\
    (b.) $\kappa_a+\kappa_b < \kappa_a+\kappa_c < \kappa_b + \kappa_c < \kappa_i + \kappa_j$.
\item[Case 5.] $a<b<i<c<j$, which implies that 
    $\kappa_a + \kappa_b < \kappa_a + \kappa_c < \kappa_b + \kappa_c < \kappa_i + \kappa_j.$
\end{enumerate}
(Note that any other ordering on $a,b,c,i,j$, such as $i<j<a<b<c$, would imply that
there are no black $X$-crossings involving the edges incident to 
$v_{a,b,c}$ and the $[i,j]$ soliton.)

\begin{figure}[h]
\begin{center}
\includegraphics[height=7.2cm]{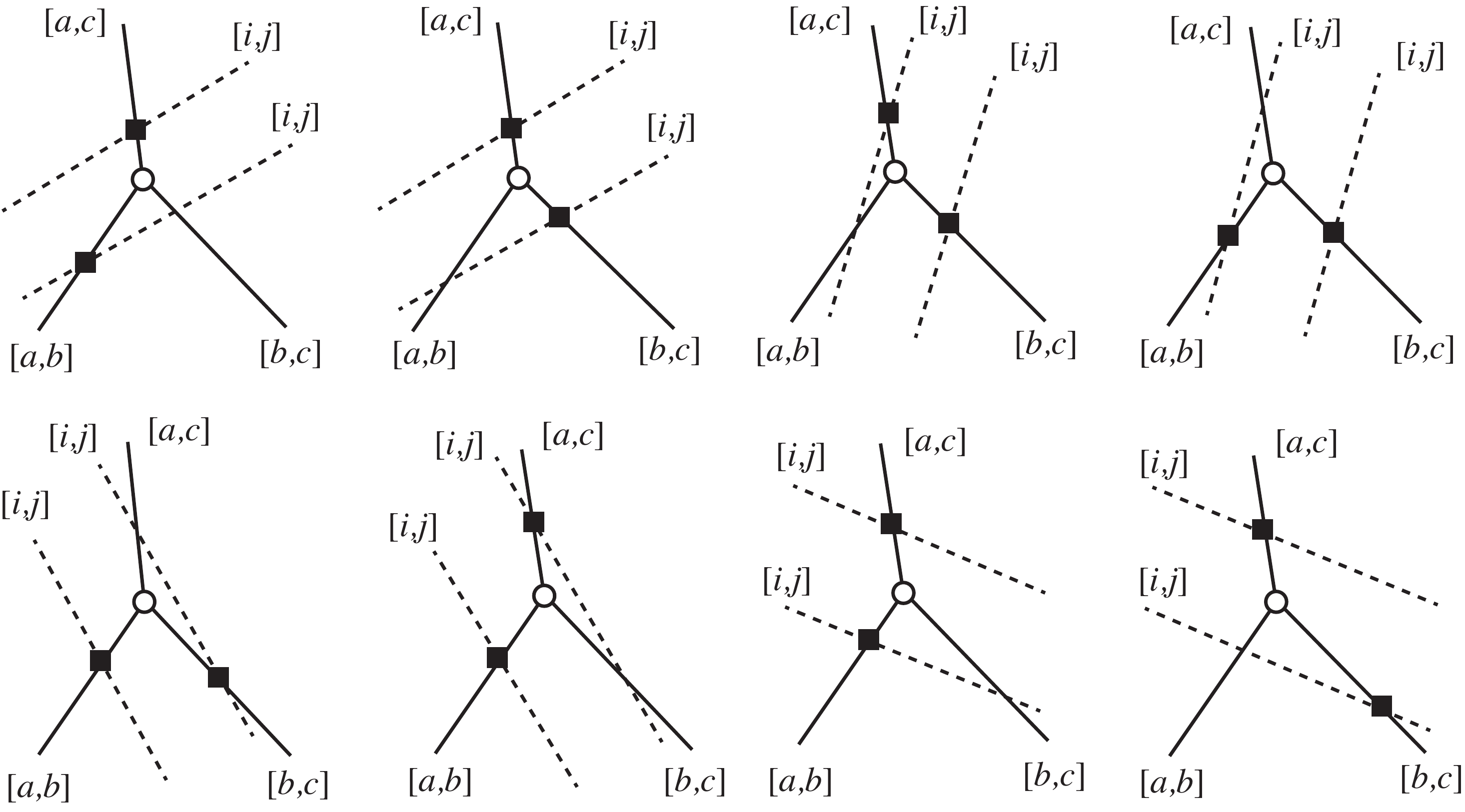}
\end{center}
\caption{Various types of $X$-crossings involving the line-solitons incident to $v_{a,b,c}$
and the $[i,j]$ line-soliton.  The top row shows 
Cases 1, 2a, 2b, and 3a from left to right, while the bottom row shows
Cases 3b, 4a, 4b, and 5 from left to right.}
\label{fig:Xslide}
\end{figure}

Consider Case 1.  Recall that ``slope" of the $[i,j]$ line-soliton -- 
that is, the tangent of the angle measured counterclockwise from the positive 
$y$-axis to the $[i,j]$ line-soliton -- 
is equal to $\kappa_i+\kappa_j$.   Therefore from the order on 
the slopes, the $[i,j]$ soliton may intersect either the $[a,c]$ soliton
or both the $[a,b]$ and $[b,c]$ solitons, as in the top-left diagram of Figure \ref{fig:Xslide}. 
The black $X$-crossings are denoted by a solid black square.
In both cases, precisely one of the intersections is a black $X$-crossing.
The other cases are similar -- see Figure \ref{fig:Xslide}.
\end{proof}

\begin{remark}
In fact one can show that the slides from Cases 3a and 3b in Figure \ref{fig:Xslide}
are impossible at $t\ll0$.  More specifically, it is impossible for 
the $[i,j]$ line-soliton to intersect the $[b,c]$ line-soliton.  To show this,
one may compute the coordinates $(x_v,y_v)$ of the trivalent vertex $v$ where
the $[a,b]$, $[a,c]$, and $[b,c]$ solitons intersect.  Then one can show that
the intersection of the $[i,j]$ soliton and the line $y = y_v$ has 
$x$-coordinate which is strictly less than $x_v$.
\end{remark}

\subsection{Slides and Pl\"ucker coordinates}

In \cite[Theorem 9.1]{KW2}, we proved that the presence of 
$X$-crossings in contour plots at $|t|\gg0$ implies 
that there is a two-term Pl\"ucker relation.

\begin{theorem}\cite[Theorem 9.1]{KW2}\label{2term}
Suppose that there is an 
$X$-crossing in a contour plot $\CC_t(u_A)$ 
for some $A \in Gr_{k,n}$ where $|t|\gg0$.  
Let $I_1$, $I_2$, $I_3$, and $I_4$ be the $k$-element subsets of $\{1,\dots,n\}$
corresponding to the dominant exponentials incident
to the $X$-crossing listed in circular order.  
\begin{itemize}
\item If the $X$-crossing is white, 
we have $\Delta_{I_1}(A) \Delta_{I_3}(A) = \Delta_{I_2}(A) \Delta_{I_4}(A).$
\item If
the $X$-crossing is black, 
we have $\Delta_{I_1}(A) \Delta_{I_3}(A) = -\Delta_{I_2}(A) \Delta_{I_4}(A).$
\end{itemize}
\end{theorem}

The following corollary is immediate.
\begin{corollary}\label{prop:opposite}
If there is a black $X$-crossing in a contour plot at $t\ll 0$ or $t\gg 0$, 
then among the Pl\"ucker coordinates associated to 
the dominant exponentials incident 
to that black $X$-crossing, three must be positive and one negative,
or vice-versa. 
\end{corollary}

\begin{corollary}\label{cor:white}
Let $D$ be a $\Le$-diagram, that is, a Go-diagram with no black stones.
Let $A \in S_D$ and $t \ll0$.  Choose any $\kappa_1 < \dots < \kappa_n$.
Then the contour plot $\CC_t(u_A)$ can have only white $X$-crossings.
\end{corollary}

\begin{proof}
From Theorem \ref{t<<0}, it follows that 
the contour plot $\CC_{-\infty}(u_A)$ has no dependence on the signs of the Pl\"ucker
coordinates of $A$.  (In fact it has no dependence on $A$, only on the Deodhar stratum 
$S_D$ containing $A$.) Since $D$ is a $\Le$-diagram, we can choose
an element $A' \in S_D \cap (Gr_{k,n})_{\geq 0}$, and 
$\CC_{-\infty}(u_A) = 
\CC_{-\infty}(u_{A'})$.   But now since the Pl\"ucker coordinates of $A'$ are
all non-negative, by Theorem \ref{2term}, there cannot be any black $X$-crossings
in the contour plot.
\end{proof}

\begin{lemma}\label{lem:slide2}
Consider two contour plots for $A \in Gr_{k,n}$ which differ by a single slide.
Let $\J$ and $\J'$ denote the two sets of Pl\"ucker coordinates
corresponding to the dominant exponentials in the two contour plots.
Then from the values of the Pl\"ucker coordinates in $\J$,
one can reconstruct the values of the Pl\"ucker coordinates in $\J'$,
and vice-versa.  
\end{lemma}
\begin{proof}
By Theorem \ref{2term}, the four Pl\"ucker coordinates
incident to an $X$-crossing  satisfy
a ``two-term" Pl\"ucker relation.
Now it is easy to verify the lemma by inspection, since each slide only creates or removes
one region, and there is a dependence among the Pl\"ucker coordinates
labeling the dominant exponentials.  The reader may wish to check this
by looking at the first and second, or the second and third, or the third and fourth
contour plots in Figure
\ref{fig:slide}.
\end{proof}

\begin{corollary}\label{cor:pos-pos}
Let $D$ be a $\Le$-diagram,
such that $S_D \subset Gr_{k,n}$.
Let $\CC_{-\infty}(D)$ and 
$\CC'_{-\infty}(D)$ be two contour plots
defined using two different sets of parameters
$\kappa_1 < \dots < \kappa_n$ and 
$\kappa'_1 < \dots < \kappa'_n$.
Let $\J$ and $\J'$ be the $k$-element subsets
corresponding to the dominant exponentials in 
$\CC_{-\infty}(D)$ and 
$\CC'_{-\infty}(D)$.
If $\Delta_I(A)>0$ for each
$I \in \J$, then $\Delta_I(A)>0$ for each $I \in \J'$.
In particular, if $\J$ is a positivity test for $S_D$ then 
so is $\J'$.
\end{corollary}
\begin{proof}
One may use a continuous deformation of the parameters to get 
from $\kappa_1 < \dots < \kappa_n$ to
$\kappa'_1 < \dots < \kappa'_n$.  As one deforms the parameters
the contour plot will change by a sequence of slides.  At each step
along the way, the contour plot will contain only white $X$-crossings
(by Corollary \ref{cor:white}).
By Lemma \ref{lem:slide2}, if we know the values of the Pl\"ucker
coordinates labeling dominant exponentials before a slide,
then we can compute the Pl\"ucker coordinates labeling dominant exponentials
after a slide.  Moreoever, since this computation involved only 
two-term Pl\"ucker relations and all the $X$-crossings are white,
the positivity of the Pl\"ucker coordinates in $\J$ implies 
the positivity of the Pl\"ucker coordinates in $\J'$.
\end{proof}




\section{The regularity problem for KP solitons}\label{sec:regularity}
In this section, we first discuss the regularity of KP solitons.
Given a soliton solution $u_A$ coming from an element $A\in Gr_{k,n}$,
we show that if $u_A(x,y,t)$ is regular for $t\ll0$, then in fact 
$A$ must lie in the totally non-negative part
$(Gr_{k,n})_{\ge 0}$ of the Grassmannian.
We then discuss the uniqueness (and lack thereof) of the pattern
when the soliton solution is not regular.

Our main theorem is the following.
\begin{theorem}\label{th:regularity}
Fix parameters $\kappa_1 < \dots < \kappa_n$ and 
an element $A \in Gr_{k,n}$.  Consider the corresponding soliton solution
$u_A(x,y,t)$ of the KP equation.  This solution is regular at $t\ll0$
if and only if $A \in (Gr_{k,n})_{\geq 0}$.
Therefore this solution is regular for all times $t$ 
if and only if $A \in (Gr_{k,n})_{\geq 0}$.
\end{theorem}

We will prove Theorem \ref{th:regularity} in Section 
\ref{sec:pos-reg}, after establishing some results on 
black $X$-crossings.

\subsection{Lemmas on black $X$-crossings}


Recall from Section \ref{sec:contour}
that $\phi_i(\bar{x},\bar{y}) = \kappa_i \bar{x} + \kappa_i^2 \bar{y} + \kappa_i^3.$
The following lemma is easy to check.
\begin{lemma}
For $1 \leq i < j \leq n$, let $L_{ij}$ be the line in the $\bar{x}\bar{y}$-plane
where $\phi_i(\bar{x},\bar{y})=\phi_j(\bar{x},\bar{y})$.
For $i<j<k<\ell$, let $b_{i,j,k,\ell}$ be the point where
the lines $L_{ik}$ and $L_{j\ell}$ intersect.
Then $L_{ij}$ has the equation
$$\bar{x} + (\kappa_i + \kappa_j) \bar{y} + (\kappa_i^2+\kappa_i \kappa_j + \kappa_j^2) = 0,$$ and the point $b_{i,j,k,\ell} = (b_{i,j,k,\ell}^{\bar{x}},
b_{i,j,k,\ell}^{\bar{y}})$ has the coordinates
\begin{align*}
b_{i,j,k,\ell}^{\bar{x}} &= \frac{\kappa_i^2 \kappa_j+\kappa_i^2 \kappa_{\ell} - 
\kappa_i \kappa_j^2+\kappa_i \kappa_j \kappa_k-\kappa_i \kappa_j \kappa_{\ell} + \kappa_i \kappa_k \kappa_{\ell} - \kappa_i \kappa_{\ell}^2 - \kappa_j^2 \kappa_k + \kappa_j \kappa_k^2-\kappa_j \kappa_k \kappa_{\ell} + \kappa_{k}^2 \kappa_{\ell} - \kappa_k \kappa_{\ell}^2}{\kappa_i - \kappa_j + \kappa_k - \kappa_{\ell}}\\
b_{i,j,k,\ell}^{\bar{y}} &= 
\frac{-\kappa_i^2 - \kappa_i \kappa_k + \kappa_j^2+\kappa_j \kappa_{\ell}  
-\kappa_k^2 + \kappa_{\ell}^2}{\kappa_i - \kappa_j + \kappa_k - \kappa_{\ell}}.
\end{align*}
\end{lemma}

\begin{lemma}\label{lem:less}
Consider the point $b_{i,j,k,\ell}$ where $1 \notin \{i,j,k,\ell\}$.
Then at this point we have 
$\phi_1 < \phi_i = \phi_k$ and $\phi_1 < \phi_j = \phi_{\ell}$.
\end{lemma}

\begin{proof}
By definition of $b_{i,j,k,\ell}$ we have that at this point
$\phi_i = \phi_k$ and $\phi_j = \phi_{\ell}$.  So we just need
to show that at $b_{i,j,k,\ell}$, 
$\phi_1 <\phi_i$ and $\phi_1 < \phi_j$.
A calculation shows that
$\phi_i(b_{i,j,k,\ell}) - \phi_1(b_{i,j,k,\ell})$ is equal to 
\begin{equation*}
\frac{(\kappa_k - \kappa_1)(\kappa_i-\kappa_1)[(\kappa_j - \kappa_1)(\kappa_j-\kappa_i + \kappa_{\ell}-\kappa_k) + (\kappa_{\ell}-\kappa_i)(\kappa_{\ell}-\kappa_k)]}{\kappa_j - \kappa_i + \kappa_{\ell} - \kappa_k}, 
\end{equation*}
and $\phi_j(b_{i,j,k,\ell}) - \phi_1(b_{i,j,k,\ell})$ is equal to 
\begin{equation*}
\frac{(\kappa_{\ell} - \kappa_1)(\kappa_j-\kappa_1)[(\kappa_i - \kappa_1)(\kappa_j-\kappa_i + \kappa_{\ell}-\kappa_k) + (\kappa_{\ell}-\kappa_k)(\kappa_{k}-\kappa_j)]}{\kappa_j - \kappa_i + \kappa_{\ell} - \kappa_k}. 
\end{equation*}
Because $\kappa_1 < \kappa_i < \kappa_j < \kappa_k < \kappa_{\ell}$,
we can readily verify that the above quantities are positive.
\end{proof}
\begin{remark}
Lemma \ref{lem:less}
 will be instrumental in proving Proposition \ref{prop:blackX}
below regarding black $X$-crossings.  Note that 
if in the lemma we took the order 
$i<k<j<\ell$ or $i<j<\ell<k$ then our proof would
not work.  So Proposition \ref{prop:blackX} does not
necessarily hold for white X-crossings.
\end{remark}

\begin{proposition}\label{prop:blackX}
Use the hypotheses and notation of Theorem \ref{induction}.
Then every black $X$-crossing of  
$\CC_{-\infty}(\mathcal{M'})$ remains a 
black $X$-crossing in 
$\CC_{-\infty}(\mathcal{M})$; and
each region in 
$\CC_{-\infty}(\mathcal{M'})$ 
which is incident to a 
black $X$-crossing 
and is labeled by $E_{J'}$ corresponds to a region of 
$\CC_{-\infty}(\mathcal{M})$ which is labeled by 
$E_{J' \cup \{1\}}$.
\end{proposition}

\begin{proof}
Consider a black $X$-crossing $b_{a,b,c,d}$ of 
$\CC_{-\infty}(\mathcal{M'})$ in which the 
line-solitons $[a,c]$ and $[b,d]$ intersect
(here $a<b<c<d$).  Since this is taking place in 
$\CC_{-\infty}(\mathcal{M'})$,
$1 \notin \{a,b,c,d\}$.  
The four regions $R_1, R_2, R_3, R_4$
incident to  $b_{a,b,c,d}$
are labeled by $E_{J_1}, E_{J_2}, E_{J_3}, E_{J_4}$.
In particular,  this means that at region 
$R_1$, $J_1$ is the subset 
$\{j_1,\dots,j_{k-1}\}$ of $\mathcal{M'}$ which 
minimizes the value $\theta_{j_1} + \dots + \theta_{j_{k-1}}$.
Without loss of generality we can assume that $a \in J_1$.
But then by 
Lemma \ref{lem:less}, there is a neighborhood $N$
of $b_{a,b,c,d}$ where $\phi_1$ is less than  
$\phi_a$.
It follows that 
in $N \cap R_1$, $J_1 \cup \{j_k = 1\}$ is the subset of $\M$
that minimizes the value 
$\theta_{j_1} + \dots + \theta_{j_{k}}$.
Therefore the region $R_1$ of 
$\CC_{-\infty}(\mathcal{M'})$ which is labeled 
by $E_{J_1}$ corresponds to a region of 
$\CC_{-\infty}(\mathcal{M})$ which is labeled 
by $E_{J_1 \cup \{1\}}$.  Similarly for $R_2$, $R_3$, and $R_4$.
In particular, the black $X$-crossing from 
$\CC_{-\infty}(\mathcal{M'})$  
will remain a black $X$-crossing in 
$\CC_{-\infty}(\mathcal{M})$.
\end{proof}



Recall the notion of a slide from Definition \ref{def:slide}.
\begin{proposition}\label{prop:slide}
Choose a Go-diagram $D$ such that $S_D \subset Gr_{k,n}$.
Let $\kappa_1 < \dots < \kappa_n$ and 
$\kappa'_1 < \dots < \kappa'_n$ be two choices of parameters,
and let $\CC_{-\infty}(D)$ and 
$\CC'_{-\infty}(D)$ be the corresponding contour plots.
Then if 
$\CC_{-\infty}(D)$ has $r$ black $X$-crossings, then 
$\CC'_{-\infty}(D)$ has $r$ black $X$-crossings.
\end{proposition}

\begin{proof}
By Remark \ref{rem:slides}, the two contour plots differ by a series of slides.  
And by Theorem \ref{th:blackX}, each slide preserves the number of black $X$-crossings.
\end{proof}

\begin{theorem}\label{cor:bX}
If $D$ is a Go-diagram with at least one black stone, then the
contour plot $\CC_{-\infty}(D)$ 
contains a black $X$-crossing.
\end{theorem}

\begin{proof}
Let $i$ denote the bottom-most row of $D$ which contains 
a black stone.  Choose $A \in S_D$ and put it in row-echelon form;
let $A'$ denote the span of rows $i,i+1,\dots,k$
of $A$.  So $A' \in S_{D'} \subset Gr_{k-i+1,n}$, 
where $D'$ is the Go-diagram obtained from 
rows $i,i+1,\dots,k$ of $D$.
Then by Proposition \ref{prop:blackX}, 
if we can show that 
the contour plot $\CC_{-\infty}(D')$ contains a black 
$X$-crossing, then 
$\CC_{-\infty}(D)$ must also contain a black $X$-crossing.

Our goal now is to show that there is a choice of the 
$\kappa$-parameters such that 
$\CC_{-\infty}(D')$ contains a black $X$-crossing.
If we can show this, then by Proposition \ref{prop:slide},
we will be done.

Note that for $t=-1$, we have the following.
\begin{itemize}
\item[(i)]  If $i<j<k$, then the y-coordinate $y_{i,j,k}$ of the 
trivalent vertex $v_{i,j,k}$ where the $[i,j]$, $[j,k]$ and $[i,k]$ solitons meet is:
\begin{equation*}
y_{i,j,k}=\kappa_i+\kappa_j+\kappa_k.
\end{equation*}
\item[(ii)]  If $i<j<k<\ell$, then the $y$-coordinate $y_{i,j,k,\ell}$ 
of an $X$-crossing between the $[i,k]$ and $[j,\ell]$ solitons is:
\begin{equation*}
y_{i,j,k,\ell}=\kappa_i+\kappa_j+\kappa_k+\kappa_{\ell}-\frac{\kappa_i\kappa_k-\kappa_j\kappa_{\ell}}{(\kappa_i+\kappa_k)-(\kappa_j+\kappa_{\ell})}.
\end{equation*}
\end{itemize}

Consider the left-most black stone $b$ in $D'$.
Let $[i,b]$ and $[a,j]$ with $i<a<b<j$ be the pair of lines in $G_-(D')$
which cross at this black stone.
Then there are two cases:
\begin{itemize}
\item[(a)] There is no empty box to the left of $b$ in $D'$,
and so there is an unbounded $[i,b]$-soliton at $y\gg 0$ in the corresponding contour plot.
Because $b$ is a black stone, the $[i,b]$-soliton must have a trivalent vertex $v_{i,b,j'}$ at one end,
where $j' \geq b$.  Additionally, $[a,j]$ is an unbounded soliton at $y\ll0$,
and it has a trivalent vertex $v_{i',a,j}$ at one end, where $i' \leq a$.
See Figure \ref{fig:ChoiceK}.
\begin{figure}[h]
\begin{center}
\includegraphics[height=4.5cm]{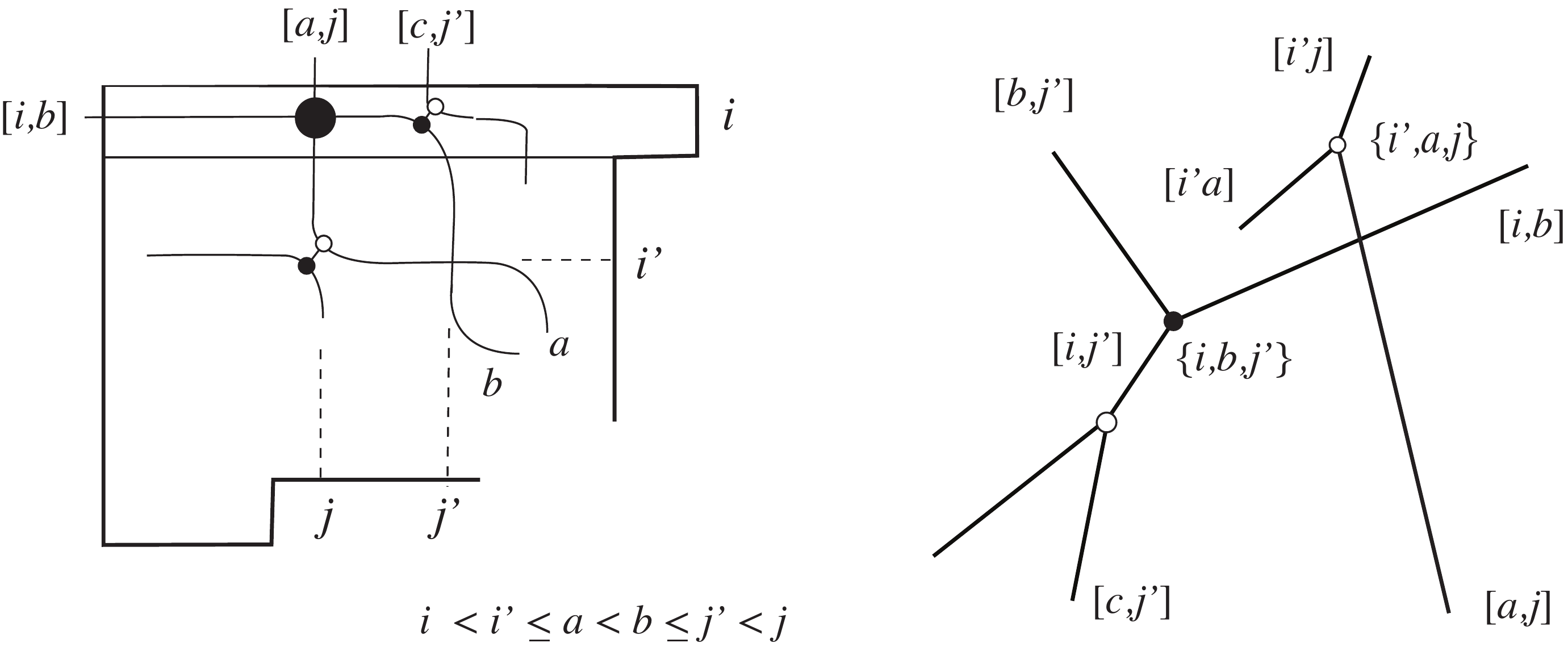}
\end{center}
\caption{
}
\label{fig:ChoiceK}
\end{figure}


If we can choose the $\kappa$-parameters such that 
$y_{i',a,j} > y_{i,a,b,j} > y_{i,b,j'}$ then it follows that there
is an intersection of the $[a,j]$ and $[i,b]$ line-solitons in the contour plot.


One simple choice is to require that 
\begin{align}\label{ineq1}
\kappa_j&=-\kappa_i>0\quad \text{ and }\quad  \kappa_b=-\kappa_a>0; \quad\text{ and also} \\
\kappa_{i'}&>\frac{1}{2}\kappa_i \quad\text{ and }\quad  \kappa_{j'}<\frac{1}{2}\kappa_j.\label{ineq2}
\end{align}
By \eqref{ineq1}, we  have 
$y_{i,a,b,j} = 0$.
By \eqref{ineq2}, 
together with 
$\kappa_a > \kappa_{i'}$ and $\kappa_b < \kappa_{j'}$, we have that 
\[
y_{i',a,j} = \kappa_{i'}+\kappa_a+\kappa_j~>~
0~ > ~\kappa_{i}+\kappa_b+\kappa_{j'} = y_{i,b,j'}.
\]
One concrete choice of parameters satisfying the required inequalities is 
$(\kappa_i, \kappa_{i'}, \kappa_a, \kappa_b, \kappa_{j'}, \kappa_j) = 
  (-4r, -2r, -r, r, 2r, 4r)$ where $r>0$.\\

\item[(b)] The second
 case is that there is an empty box to the left of $b$ in $D'$,
and so the $[i,b]$ line-soliton has trivalent vertices at both ends.
Figure \ref{fig:ChoiceK2} illustrates this situation.
These vertices are the white vertex $v_{i,b,j''}$ 
and the black vertex $v_{i,b,j'}$ where $i<i'\leq a < b \leq j' < j < j''$.
As before, $[a,j]$ is an unbounded line-soliton at $y\ll0$ which is incident
to the trivalent vertex $v_{i',a,j}$.
\begin{figure}[h]
\begin{center}
\includegraphics[height=5cm]{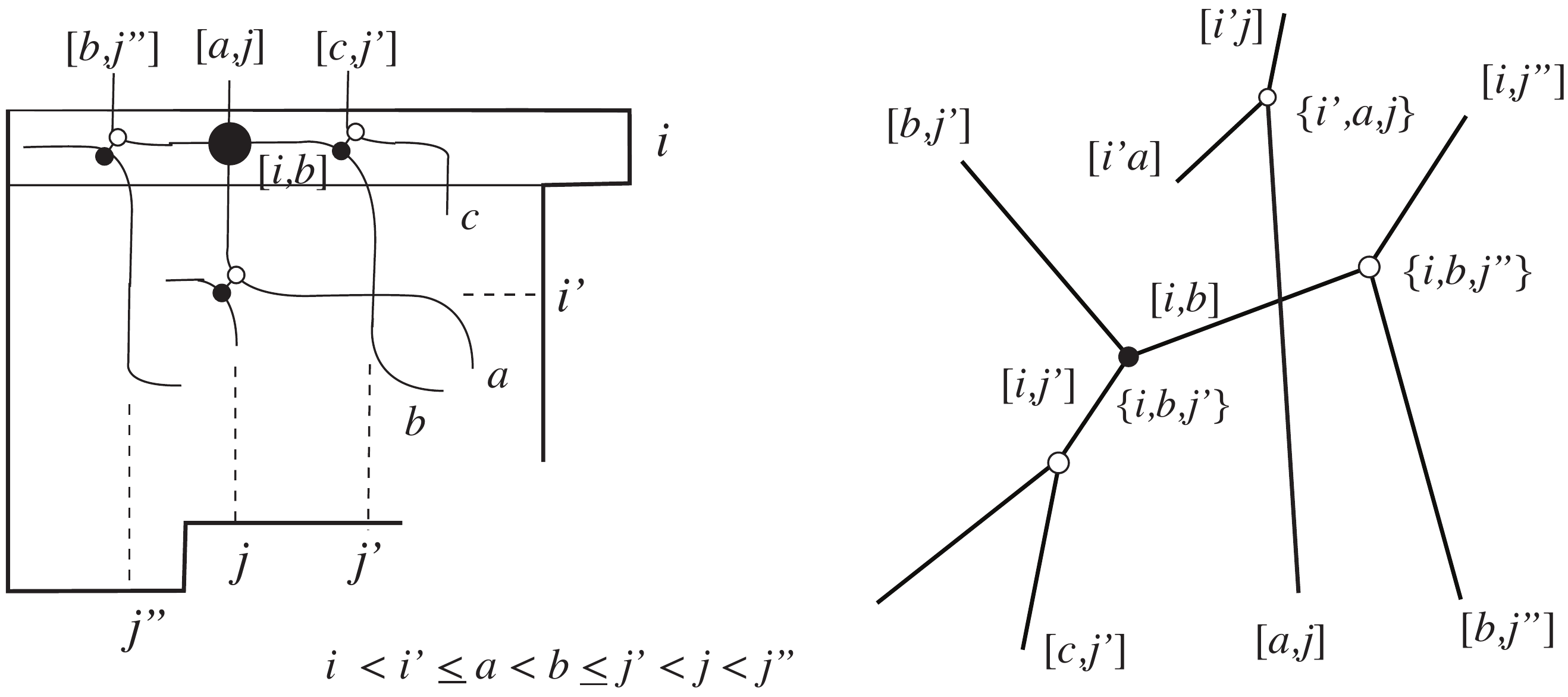}
\end{center}
\caption{
}
\label{fig:ChoiceK2}
\end{figure}
Since $v_{i,b,j''}$ is a white vertex, if  we can show that 
\[
y_{i,b,j''}>y_{i,a,b,j}>y_{i,b,j'},\qquad {\rm and}\qquad y_{i'a,j}>y_{i,a,b,j},
\]
then it follows that the line-solitons of type $[a,j]$ and $[i,b]$ intersect in the 
contour plot.


As before, we choose the $\kappa$-parameters so that \eqref{ineq1} and \eqref{ineq2}
are satisfied.
Then again we have $y_{i,a,b,j} = 0$,
$y_{i',a,j}>0$, and $y_{i,b,j'} < 0$.
Note that  any choice of $\kappa_{j''} > \kappa_j$ gives
$y_{i,b,j''}>0$, since 
$\kappa_i+\kappa_b + \kappa_{j''} > 
 \kappa_i + \kappa_b + \kappa_j = \kappa_b > 0$.
\end{itemize}
This completes the proof.
\end{proof}

\subsection{Positivity of dominant exponentials and the proof of 
Theorem \ref{th:regularity}}\label{sec:pos-reg}

In this section we prove Theorem \ref{th:Le} below.  Once we have proved it,
the proof of Theorem \ref{th:regularity} will follow easily.

\begin{theorem}\label{th:Le}
Let $A \in S_D \subset Gr_{k,n}$, where $D$ is a $\Le$-diagram, and let
$t \ll0$.  If $\Delta_J(A)>0$ for each dominant exponential $E_J$
in the contour plot $\CC_t(u_A)$,
then  $A \in (Gr_{k,n})_{\geq 0}$.  In other words,
the Pl\"ucker coordinates corresponding to the 
dominant exponentials in $\CC_t(u_A)$ comprise a positivity test for 
$S_D$.
\end{theorem}

\begin{lemma}\label{lem:basecase}
Theorem \ref{th:Le} holds for elements $A\in Gr_{1,n}$.
\end{lemma}
\begin{proof}
Let $A\in S_D \subset Gr_{1,n}.$  If 
$D$ contains $r$ empty boxes, then $S_D$ has dimension $r$.  Meanwhile,
the element $A$ will have precisely $r+1$ nonzero Pl\"ucker coordinates.
(We can normalize the lexicographically
minimal one to be $1$.)  It is 
easy to see that $G_-(D)$ and hence $\CC_{t}(u_A)$ will have $r+1$ 
regions, each one labeled by a different dominant exponential
corresponding to a Pl\"ucker coordinate $\Delta_J(A)$ such that 
$\Delta_J(A) \neq 0$.
Therefore if each such $\Delta_J(A) >0$,
then $A \in (Gr_{1,n})_{\geq 0}$.
\end{proof}

\begin{lemma}\label{lem:choosekappa}
Let $A \in S_D \subset Gr_{k,n}$.
Then it is possible to choose $\kappa_1 < \kappa_2 < \dots < \kappa_n$ such that 
the unbounded line-solitons at $y\ll 0$ in the corresponding contour plot $\CC_t(u_A)$
(for any time $t$)
appear in the same order as they do in the generalized plabic graph $G_-(D)$.
\end{lemma}

\begin{proof}
Recall that in a contour plot, the unbounded line-solitons $[i,j]$ at $y\ll 0$
appear from left to right in increasing order  of the slope $\kappa_i + \kappa_j$.
While in $G_-(D)$, one may easily check that the unbounded line-solitons
$[i,j]$ at $y\ll 0$  appear from left to right  in increasing order of $j$.

Now let us choose $\kappa_1, \dots, \kappa_n$ so that 
$\kappa_{i} - \kappa_{i-1} = r^i$ for some constant $r > 1$.
To prove the lemma, it suffices to prove that given 
two line-solitons $[a,b]$ and $[c,d]$ at $y \ll0$, where $b<d$,
we have that 
\begin{equation} \label{check}
\kappa_a + \kappa_b < \kappa_c + \kappa_d, \text{ or equivalently, }
\kappa_d - \kappa_b > \kappa_a - \kappa_c.
\end{equation}

Since $a<b$ and $c<d$, we have
$a<d$.
By our choice of the $\kappa_i$'s, 
$\kappa_d - \kappa_b \geq r^d$.
If $a<c$ then $\kappa_a - \kappa_c <0$, so \eqref{check} is obvious.
On the other hand, if $a>c$, then $\kappa_a - \kappa_c \leq r^a + r^{a-1} + \dots + 1 < r^{a+1}.$
And since $a<d$,  equation \eqref{check} follows.
\end{proof}

We now prove Theorem \ref{th:Le}.
\begin{proof}
Our strategy is to use induction on the number of rows of $A$.
Lemma \ref{lem:basecase} takes care of the base case of the induction.
We suppose that $A$ is in row-echelon form, and let $A'$ be the element of $Gr_{k-1,n}$ 
obtained from the bottom $k-1$ rows of $A$.   Then $A' \in S_{D'}$
where $D'$ is also a $\Le$-diagram (it is the restiction of $D$ to its  bottom $k-1$
rows).
Recall from Theorem \ref{induction} that ``most" of
the contour plot $\CC_t(u_{A'})$ is contained in the contour plot
$\CC_t(u_{A})$.  More precisely,
every region of 
$\CC_t(u_{A'})$ which is incident to a trivalent vertex 
and labeled by $E_{J'}$ corresponds to a region of 
$\CC_t(u_{A})$ which is labeled by $E_{J' \cup \{1\}}$.
Because $A$ is in row-echelon form with a pivot in row $1$,
$\Delta_{J' \cup \{1\}}(A) = 
\Delta_{J'}(A')$, so 
the fact that each $\Delta_{J' \cup \{1\}}(A)>0$ implies
that $\Delta_{J'}(A')>0$.

We now claim that all Pl\"ucker coordinates
corresponding to the dominant exponentials of the 
contour plot $\CC_t(u_{A'})$ are positive.
To prove this, note that 
from $\CC_t(u_A)$ we can in fact construct $\CC_t(u_{A'})$: 
all of the trivalent vertices of $\CC_t(u_{A'})$ are present in 
$\CC_t(u_{A})$, so it is just a matter of extending some line-solitons
that were finite in 
$\CC_t(u_{A})$ but are unbounded in 
$\CC_t(u_{A'})$.  These line-solitons may create some new
white $X$-crossings but cannot create black $X$-crossings,
because $D'$ is a $\Le$-diagram.   If a single white $X$-crossing
is created, then because three of its four regions are incident
to a trivalent vertex, three of the four corresponding Pl\"ucker
coordinates are positive.  But then by the two-term
Pl\"ucker relation relating the four Pl\"ucker coordinates,
the fourth Pl\"ucker coordinate is positive as well.
If multiple white $X$-crossings are created, then one can 
iterate the above argument, starting with a white $X$-crossing
with three of its four regions incident to a trivalent vertex
in the contour plot.  This proves the claim.  So 
by the inductive hypothesis, $A' \in (Gr_{k-1,n})_{\geq 0}$.

Since $A' \in (Gr_{k-1,n})_{\geq 0}$, it follows that 
all the Pl\"ucker coordinates labeling the regions of $G_-(D')$
are positive.  And so all of the Pl\"ucker coordinates
labeling the regions of $G_-(D)$ which correspond to the bottom
$k-1$ rows of $D$ are positive.  (Recall again that 
$\Delta_{J' \cup \{1\}}(A) = 
\Delta_{J'}(A')$.)  If we can show that the Pl\"ucker coordinates
labeling the regions of $G_-(D)$ which come from the top row of 
$D$ are positive, then by Remark \ref{rem:G-Dtest}, it will follow
that $A \in (Gr_{k,n})_{\geq 0}$.

By Lemma \ref{lem:choosekappa},  we can deform
the $\kappa$-parameters so that the resulting contour plot $\CC'_t(u_A)$
has its unbounded line-solitons at $y\ll0$ in the same order as those in $G_-(D)$.
Then the dominant exponentials at $y\ll0$ in $\CC'_t(u_A)$
are precisely those of $G_-(D)$,
which in turn come from the top row of $D$.  By Corollary \ref{cor:pos-pos}, since 
the dominant exponentials of $\CC_t(u_{A})$ are positive, so are those of 
$\CC'_t(u_{A})$.  In particular, the dominant exponentials of 
$\CC'_t(u_A)$ at $y\ll0$ are positive, so we can conclude that 
all of the Pl\"ucker coordinates labeling the regions of $G_-(D)$ are positive.
Therefore $A\in (Gr_{k,n})_{\geq 0}$.
\end{proof}

Finally we are ready to prove Theorem \ref{th:regularity}.

\begin{proof}
Recall the definition of $u_A(x,y,t)$ in terms of the $\tau$-function from 
Section \ref{sec:tau}.
It is easy to verify that if $A \in (Gr_{k,n})_{\geq 0}$, then
$u_A(x,y,t)$ is regular for all times $t$: the reason is that 
$\tau_A(x,y,t)$ is  strictly positive, and hence its logarithm is well-defined.

Conversely, let $A \in Gr_{k,n}$,
and suppose that 
$u_A(x,y,t)$ is regular for 
$t\ll0$.
This means that the Pl\"ucker coordinates $\Delta_J$ corresponding
to the dominant exponentials in the contour plot
$\CC_t(u_A)$ must all have the same sign.
Since the Grassmannian is a projective variety, we may assume
that all of these Pl\"ucker coordinates are positive.

Let $S_D$ be the Deodhar stratum containing $A$.  If $D$ has a black stone,
then by Theorem \ref{cor:bX}, the contour plot $\CC_{-\infty}(u_A)$ contains
a black $X$-crossing.  But then by Corollary \ref{prop:opposite}, two dominant 
exponentials incident to that black $X$-crossing must have opposite signs, which is 
a contradiction.  Therefore we conclude that $D$ has no black stones.  It follows that
the Deodhar diagram corresponding to $D$ is a $\Le$-diagram.  
But now by Theorem \ref{th:Le}, it follows that $A \in (Gr_{k,n})_{\geq 0}$.

Finally, note that if $A \in Gr_{k,n}$ and $u_A(x,y,t)$ is 
regular for all times $t$, then in particular it is regular for $t\ll0$,
so the arguments of the previous two paragraphs apply.  Therefore 
$A \in (Gr_{k,n})_{\geq 0}$.
\end{proof}

\begin{remark}
Corollary \ref{prop:opposite} 
implies that 
there are singularities among
the line-solitons forming a black $X$-crossing in a contour plot,
and the singular solitons form 
a V-shape.
\end{remark}

\begin{example}
We revisit the example from Figures \ref{GoPlabic}
and \ref{fig:Gr48AB}.  Note 
that the contour plot at the left of Figure \ref{fig:Gr48AB}
is topologically identical to 
$G_-(D)$.  The 
Go-diagram and labeled Go-diagram are as follows.
\[
\young[4,4,3,3][10][,\hskip0.4cm\circle{5},\hskip0.4cm\circle{5},,\hskip0.4cm\circle*{5},,,,\hskip0.4cm\circle*{5},\hskip0.4cm\circle{5},,,\hskip0.4cm\circle{5},\hskip0.4cm\circle{5}]\hskip2cm \young[4,4,3,3][10][$p_{14}$,$1$,$1$,$p_{11}$,$-1$,
$p_{9}$,$p_{8}$,$p_7$,$-1$,$1$,$p_4$,$p_3$,$1$,$1$]
\]
\medskip\noindent
The $A$-matrix is given by
\[
A=\begin{pmatrix}
p_{11}p_{14}&  p_{14} & 0 & 0& 1 & 0 & 0 & 0 \\
0 &- p_7p_{8}p_{9} &-p_8p_9 & - p_{9} & -m_{10} & 0 & -1 & 0 \\
0  &  0 &  0  & -p_4 &-m_6 & 1 & 0&0 \\
  0 & 0 & 0 & 0 & p_3 & 0 &0 & 1
\end{pmatrix}.
\]
Recall from Theorem \ref{p:Plucker} 
that we associate a Pl\"ucker coordinate 
$\Delta_{I_b}$ to each box $b$ of the Go-diagram, via
$I_b=v^{\rm in}(w^{\rm in})^{-1}\{1,2,4,5\} = 
\{j_1,j_2,j_3,j_4\}$.
For brevity, we simply write 
$(j_1j_2j_3j_4)$ below.
Because the contour plot at the left of Figure \ref{fig:Gr48AB}
is topologically identical to 
$G_-(D)$, all of these Pl\"ucker coordinates $\Delta_{I_b}$ correspond
to dominant exponentials in the contour plot.
\setlength{\unitlength}{0.7mm}
\begin{center}
  \begin{picture}(60,40)
  
\put(5,35){\line(1,0){68}}
  \put(5,25){\line(1,0){68}}
  \put(5,15){\line(1,0){68}}
  \put(5,5){\line(1,0){51}}
  \put(5,-5){\line(1,0){51}}
  \put(5,-5){\line(0,1){40}}
  \put(22,-5){\line(0,1){40}}
  \put(39,-5){\line(0,1){40}}
  \put(56,-5){\line(0,1){40}}
  \put(73,15){\line(0,1){20}}

 \put(7,29){$(5678)$}
 \put(24,29){$(2567)$}
 \put(41,29){$(2456)$}
   \put(58,29){$(2345)$}
  \put(7,19){$(1678)$}
  \put(24,19){$(1567)$}
\put(41,19){$(1456)$}
  \put(58,19){$(1345)$}
  \put(7,9){$(1268)$}
 \put(24,9){$(1256)$}
  \put(41,9){$(1256)$}
    \put(7,-1){$(1248)$}
 \put(24,-1){$(1245)$}
  \put(41,-1){$(1245)$}
    \end{picture} 
    \hskip 3cm
 $
\raisebox{1cm}{\young[4,4,3,3][10][$1$,$1$,$1$,$1$,$1$,$1$,$1$,$1$,$-1$,$1$,$1$,$1$,$1$,$1$]}
  $
\end{center}

\vskip0.5cm
\noindent
The diagram at the right shows the values of the 
corresponding Pl\"ucker coordinates when we choose
all $p_j=1$ (regardless of the choice of the $m_j$ parameters).  
Since only the Pl\"ucker coordinate
$\Delta_{1,2,6,8}(A)=-1$ is negative, 
the singular line-solitons in the contour plot
are precisely those at the boundary of the 
corresponding region; these line-solitons have types $[4,6]$, $[5,8]$,
and $[2,7]$-types.

\end{example}

\subsection{Non-uniqueness of the evolution of the contour plots for $t\gg 0$}

Consider $A \in S_D \subset Gr_{k,n}$.  
If the contour plot $\CC_{-\infty}(D)$
is topologically identical to $G_-(D)$, then 
the contour plot has almost no dependence 
on the parameters $m_j$ from the 
parameterization of $S_D$.  
This is because the Pl\"ucker coordinates
corresponding to the regions of 
$\CC_{-\infty}(D)$ (representing the dominant exponentials)
are either among the collection
of minors given in Theorem \ref{p:Plucker} 
(by Remark \ref{rem:G-Dtest}), or determined
from these by a ``two-term" Pl\"ucker relation.
Note that the minors given in Theorem \ref{p:Plucker}
are computed in terms of the parameters $p_i$ but have no 
dependence on the
$m_j$'s.

Therefore it is possible to choose two different points $A$ and $A'$ in $S_D \subset Gr_{k,n}$
whose contour plots for a fixed $\kappa_1 < \dots \kappa_n$ and fixed $t \ll0$ are identical
(up to some exponentially small difference); 
we use the same parameters $p_i$
but different parameters $m_j$ for defining $A$ and $A'$.  However, as $t$ increases, those
contour plots may evolve to give different patterns.

Consider the Deodhar stratum $S_D \subset Gr_{2,4}$, corresponding to
\[
{\bf w}=s_2s_3s_1s_2\text{ and }  {\bf v}=s_211s_2.
\]
The Go-diagram and labeled Go-diagram are given by
\[
\young[2,2][10][\hskip0.4cm\circle*{5},,,\hskip0.4cm\circle{5}]\qquad\qquad
\young[2,2][10][$-1$,$p_3$,$p_2$,$1$].
\]
\medskip
The matrix $g$ is calculated as $g=s_2y_3(p_2)y_1(p_3)x_2(m)s_2^{-1}$,
and its
projection to $Gr_{2,4}$ is 
\[
A=\begin{pmatrix}
-p_3 & -m & 1 & 0 \\
0 & p_2 & 0   & 1
\end{pmatrix}.
\]
The $\tau$-function is then given by
\[
\tau_A=-(p_2p_3E_{1,2}+p_3E_{1,4}+mE_{2,4}+p_2E_{2,3}-E_{3,4}),
\]
where $E_{i,j}:=(\kappa_j-\kappa_i)\exp(\theta_i+\theta_j)$.
The contour plots of the solutions with $m=0$ and $m\ne0$ are the same
(except for some exponentially small difference) 
when $t\ll0$.  In both cases, the plot consists of 
two line-solitons forming an $X$-crossing, where the parts of those solitons adjacent to the region
with dominant exponential $E_{3,4}$ (i.e. for $x\gg 0$) are singular,
see the left of Figure \ref{fig:NUness}.  

On the other hand, for $t\gg0$, 
the contour plot with $m=0$ is topologically the same as it was for $t\ll0$,
while the contour plot with $m\ne0$ has a box with 
dominant exponential $E_{2,4}$,
surrounded by four bounded solitons
(some of which are singular).  See the middle and right of 
Figure \ref{fig:NUness}.  
So not only the contour plots but also the soliton graphs are 
different for $t\gg0$!
\begin{figure}[h]
\begin{center}
\includegraphics[height=1.7in]{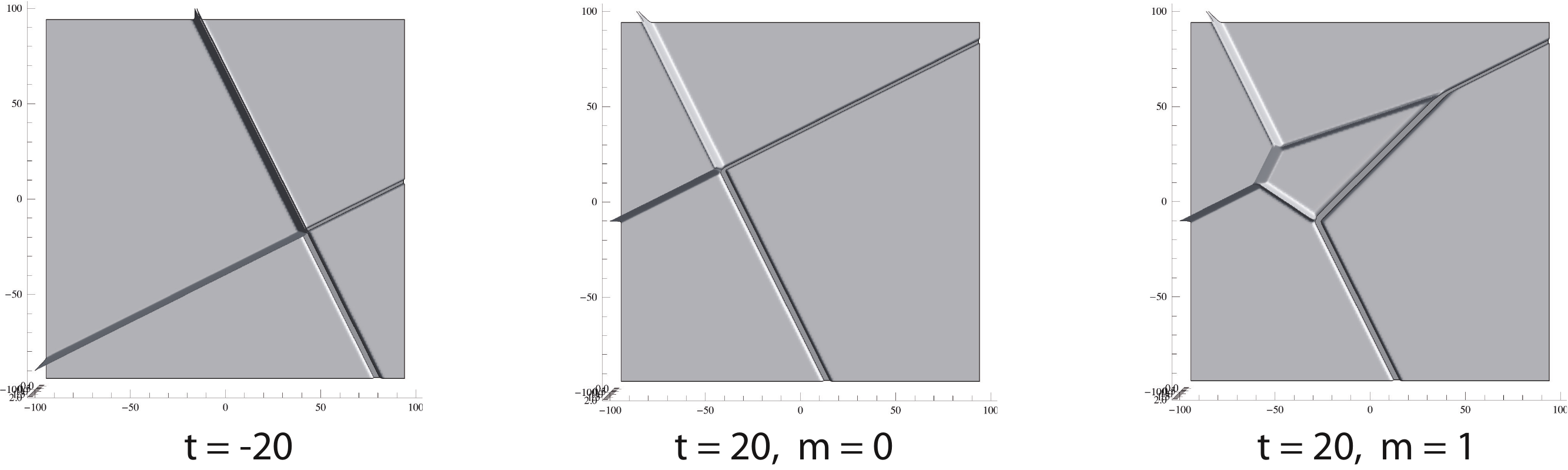}
\end{center}
\caption{The non-uniqueness of the evolution of the contour plots (and soliton graphs).  
The left panel shows the contour plot
at $t=-20$ for any value of $m$.  The middle panel shows the graph at $t=20$ with $m=0$,
and the right one shows the graph at $t=20$ with $m=1$.
These contour plots were made using the choice $p_i = 1$ for all $i$,
and 
$(\kappa_1,\ldots,\kappa_4)=(-2,-1,0,1.5)$.
In all of them,
the region at $x\gg 0$ has a positive sign ($\Delta_{3,4}=1$) and other
regions have negative signs. 
This means that the solitons adjacent to the region for $x\gg0$ are singular.}
 \label{fig:NUness}
\end{figure}

Note that the non-uniqueness of the evolution of the 
contour plot (a tropical approximation) does not imply the non-uniqueness
of the evolution of the solution of the KP equation as $t$ changes.
If one makes two different choices for the $m_i$'s, 
the corresponding $\tau$-functions are different, but 
there is only an exponentially small difference in the corresponding
contour plots (hence the topology of the contour plots is identical).
This is particularly interesting to compare 
with the totally non-negative case,
where the soliton solution can be uniquely 
determined by the information in the contour plot at $t\ll 0$.
For more details, 
see the results on the inverse problem in \cite{KW2}.

\raggedright

\end{document}